\documentclass[english]{article}
\pdfoutput=1
\usepackage{lmodern}
\usepackage{anyfontsize}
\usepackage[T1]{fontenc}

\usepackage[utf8]{inputenc}
\usepackage{abstract}
\usepackage{geometry}
\usepackage{extpfeil}
\usepackage[new]{old-arrows}
\usepackage{comment}
\usepackage{scalerel}[2016/12/29]
\usepackage[usenames,dvipsnames,svgnames,table]{xcolor}

\usepackage[maxbibnames=99, style=numeric]{biblatex}
\usepackage{mathtools}
\usepackage{palatino}
\usepackage{mathpazo}
\usepackage{eulervm}
\usepackage{mathrsfs}

\usepackage{amsmath}
\usepackage{amsthm}
\usepackage{amssymb}
\usepackage{mathdots}
\usepackage{todonotes}
\usepackage{nicefrac}
\usepackage{graphicx}
\usepackage{caption}
\usepackage{float}
\usepackage[stable]{footmisc}
\DeclareFontFamily{U}{min}{}
\DeclareFontShape{U}{min}{m}{n}{<-> udmj30}{}

\usepackage[all]{xy}
\usepackage[unicode=true,pdfusetitle,
 bookmarks=true,bookmarksnumbered=false,bookmarksopen=false,
 breaklinks=false,pdfborder={0 0 0},pdfborderstyle={},backref=false,colorlinks=true]{hyperref}

\usepackage{graphicx}

\hypersetup{
    colorlinks, linkcolor=OliveGreen,
    citecolor=OliveGreen, urlcolor=OliveGreen
}

\usepackage[nameinlink,capitalise,noabbrev]{cleveref}

\usepackage{bookmark}
\bookmarksetup{numbered}

\makeatletter

\RequirePackage{tikz-cd}
\RequirePackage{amssymb}
\usetikzlibrary{calc}
\usetikzlibrary{decorations.pathmorphing}
\usetikzlibrary{decorations.markings}

\tikzset{curve/.style={settings={#1},to path={(\tikztostart)
    .. controls ($(\tikztostart)!\pv{pos}!(\tikztotarget)!\pv{height}!270:(\tikztotarget)$)
    and ($(\tikztostart)!1-\pv{pos}!(\tikztotarget)!\pv{height}!270:(\tikztotarget)$)
    .. (\tikztotarget)\tikztonodes}},
    settings/.code={\tikzset{quiver/.cd,#1}
        \def\pv##1{\pgfkeysvalueof{/tikz/quiver/##1}}},
    quiver/.cd,pos/.initial=0.35,height/.initial=0}

\tikzset{tail reversed/.code={\pgfsetarrowsstart{tikzcd to}}}
\tikzset{2tail/.code={\pgfsetarrowsstart{Implies[reversed]}}}
\tikzset{2tail reversed/.code={\pgfsetarrowsstart{Implies}}}
\tikzset{no body/.style={/tikz/dash pattern=on 0 off 1mm}}

\newtheoremstyle{ctheorem}{}{}{}{}{\color{black}\bfseries}{}{ }{}
\theoremstyle{ctheorem}
\theoremstyle{definition}
\newtheorem{thm}{Theorem}[subsection]
\theoremstyle{definition}
\newtheorem{mainthm}{Theorem}
\newtheorem{defn}[thm]{Definition}
\newtheorem{notation}[thm]{Notation}

\newtheorem{construction}[thm]{Construction}

\theoremstyle{definition}
\newtheorem{prop}[thm]{Proposition}
\theoremstyle{definition}
\newtheorem{lem}[thm]{Lemma}
\theoremstyle{definition}

\theoremstyle{definition}
\newtheorem{rem}[thm]{Remark}
\theoremstyle{definition}
\newtheorem{example}[thm]{Example}
\theoremstyle{definition}
\newtheorem{cor}[thm]{Corollary}
\theoremstyle{definition}
\newtheorem{war}[thm]{Warning}
\theoremstyle{definition}

\theoremstyle{definition}

\theoremstyle{definition}

\newtheorem{var}[thm]{Variant}
\theoremstyle{definition}
\newtheorem{rec}[thm]{Recollection}
\theoremstyle{definition}

\usepackage{mymacros}
\usepackage{mathtools}

\DeclarePairedDelimiter\floor{\lfloor}{\rfloor}

\usepackage{citationscheme}

\author{Qingyuan Bai\thanks{
Department of Mathematical Sciences, University of Copenhagen.} \and Yuxuan Hu\thanks{Department of Mathematics, Northwestern University.}}

\DeclareUnicodeCharacter{221E}{$\infty$}
\DeclareUnicodeCharacter{03B9}{$\iota$}
\DeclareUnicodeCharacter{03BA}{$\kappa$}

\begin{document}
\title{Toric Mirror Symmetry for Homotopy Theorists}
\date{}
\maketitle
\vspace{-20pt}
\begin{abstract}
    We construct functors sending torus-equivariant quasi-coherent sheaves on toric schemes over the sphere spectrum to constructible sheaves of spectra on real vector spaces. This provides a spectral lift of the toric homolgoical mirror symmetry theorem of Fang-Liu-Treumann-Zaslow \cite{FLTZ}.
    Along the way, we obtain symmetric monoidal structures and functoriality results concerning those functors, which are new even over a field $k$.
    We also explain how the `non-equivariant' version of the theorem would follow from this functoriality via the de-equivariantization technique.
    As a concrete application, we obtain an alternative proof of Beilinson's linear algebraic description of quasi-coherent sheaves on projective spaces with spectral coefficients.
\end{abstract}
\vspace{-5pt}
\begin{center}
   \includegraphics[width=0.6\textwidth]{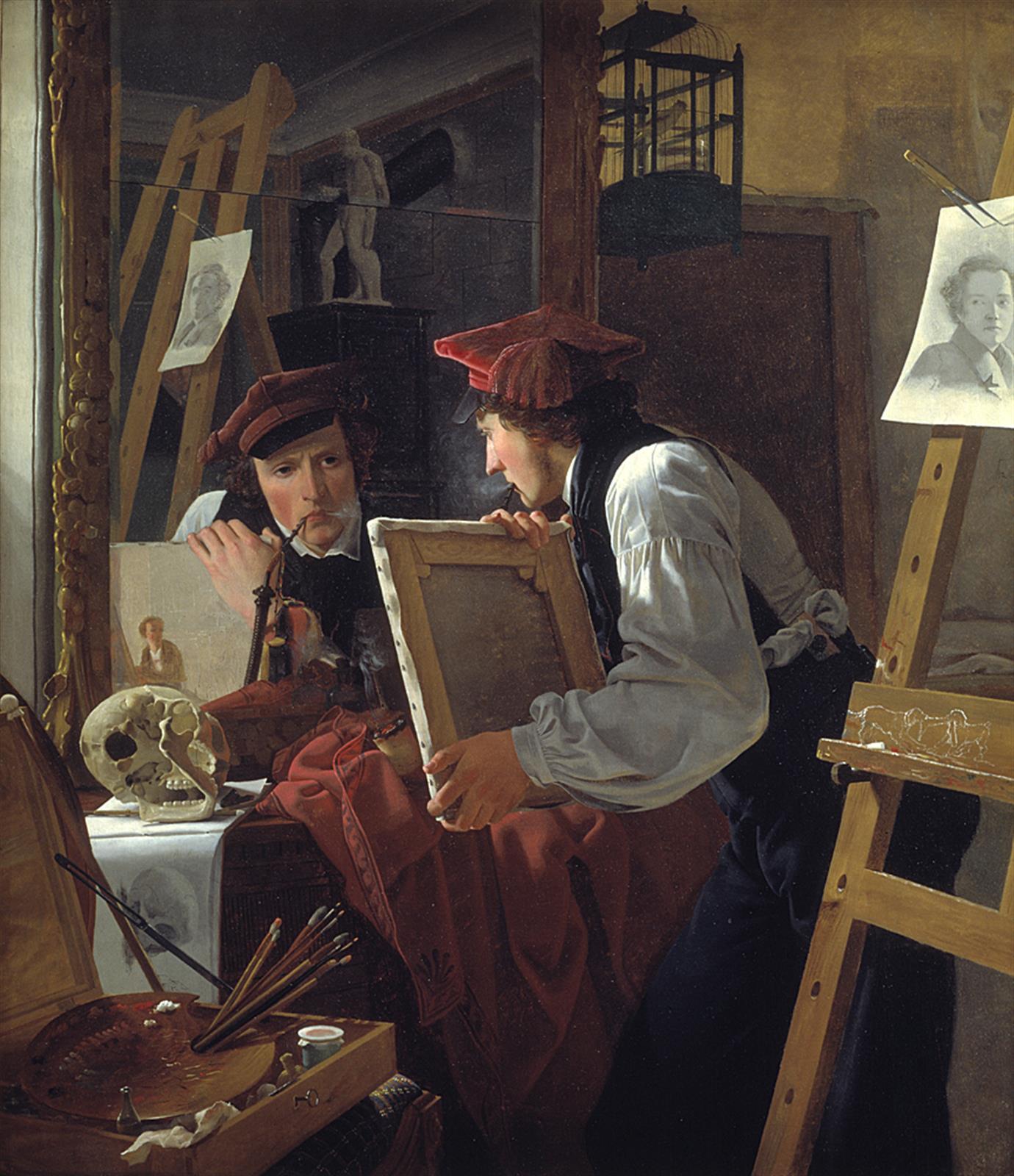}
    \\Bendz, Wilhelm. \textit{A Young Artist (Ditlev Blunck) Examining a Sketch in a Mirror.} \\1826, oil on canvas. Statens Museum for Kunst, København.
\end{center}


\newpage
\tableofcontents
\newpage
\section{Introduction}
In the classcial study of smooth projective toric varieties over $\CC$, there is a dictionary between ample line bundles and their moment polytopes as explained in \cite[Section 3.4]{fulton}.
It was observed by Robert Morelli that vector bundles also fit into this dictionary. He proved in \cite[Theorem 7]{Morelli1993TheKT} that there is an injective map from the torus-equivariant Grothendieck K-group of an $n$-dimensional smooth projective toric variety $X$ to the set of $\ZZ$-valued constructible functions on the real vector space $\RR^n$ spanned by the character lattice of the torus $T$:
$$K_0^T(X)\longrightarrow\Fun^\mathrm{cons}(\RR^n;\ZZ).$$
It becomes a map of commutative rings if one equips the set of constructible functions with pointwise addition and convolution product. This map generalizes the original dictionary: it takes the class of an ample line bundle to the characteristic function on the moment polytope.\par
In Morelli's theorem, each side admits a natural categorification. On the left hand side, one replaces $K_0^T(X)$ by $\mathrm{D}^b_T(X)$, the bounded derived category of torus-equivariant coherent sheaves on $X$. On the right hand side, one replaces the ring of constructible functions on $\RR^n$ by $\mathrm{D}^b_\mathrm{cc}(\RR^n;\Sc_\Sigma)$, the bounded derived category of sheaves of $\CC$-vector spaces on $\RR^n$ which are compactly supported and constructible (in the strong sense: the stalks have to be perfect) for a stratification $\Sc_\Sigma$. This stratification $\Sc_\Sigma$ comes from an affine hyperplane arrangement determined by toric fan $\Sigma$ for $X$ (see \cref{definition of FLTZ stratification}).
Fang-Liu-Treumann-Zaslow \cite[Theorem 1.1]{FLTZ} have constructed a fully faithful functor between dg-categories (named as \stress{coherent-constructible correspondence})
\[\kappa:\mathrm{D}^b_T(X)\longrightarrow \mathrm{D}^b_\mathrm{cc}(\RR^n;\Sc_\Sigma)\]
which recovers Morelli's theorem upon taking $K_0$.
Furthermore, they have provided a description of the image of $\kappa$ in terms of singular support.\par
Considering another perspective, one might anticipate eliminating the extra data of torus-equivariance.
In \cite{bondal2006}, Alexey Bondal had independently suggested that, given certain mild assumptions, there exists a fully faithful functor from the bounded derived category of coherent sheaves on $X$ to the bounded derived category of sheaves of $\CC$-vector spaces on the topological torus.
\[
\overline{\kappa}:\mathrm{D}^b(X)\longrightarrow \mathrm{D}^b(\RR^n/\ZZ^n),
\]
whose image is constructible for a specific stratification.
As it turns out (see \cite{treumann2010remarksnonequivariantcoherentconstructiblecorrespondence}) one can define
a functor $\overline\kappa$ in a way similar to $\kappa$, and describe the image of $\overline\kappa$ in terms of singular supports.
This line of work has been further pursued in
\cite{SS2016, zhou2017twistedpolytopesheavescoherentconstructible, Kuwagaki_2020}.

In this note, we provide an exposition of this narrative in the context of spectral algebraic geometry.
We carefully construct the relevant functors and explain how to derive formal consequences from the equivalences, leveraging available technologies in higher algebra.
\subsection{What is done in this note?}
This note begins with the realization that, on the `constructible' aspect of the discussion, there is a clear extension to the sphere spectrum. Instead of working with the bounded derived category of sheaves of $\CC$-vector spaces, one can consider the large category of sheaves of spectra on a real vector space,
$$\Shv(\RR^n;\Sp)$$
where the convolution product is naturally defined. This is made possible by recent advances in the yoga of six-operations.
On the `coherent' side, it is generally difficult to lift varieties to the sphere spectrum.
However, it is straightforward to construct lifts of toric varieties, as they are Zariski locally monoid schemes glued together along maps induced by monoid homomorphisms.
Indeed, given a toric fan $\Sigma$, one can define the flat toric scheme $X_\Sigma$ over the sphere spectrum, equipped with an action of the flat torus $\TT$.
The main purpose of this note is to provide the following construction.
\begin{mainthm}
\label{equivariantmainthm}
Let $N$ be a lattice and $\Sigma$ be a smooth projective fan in $N_\RR\defeq N\otimes \RR$ (see \cref{notation:fan}). Let $M$ and $M_\RR$ be the dual lattice and vector space.
There exists a fully-faithful, symmetric monoidal functor 
$$\kappa:\QCoh([X_\Sigma/\TT])\longrightarrow\Shv(M_\RR;\Sp),$$
where $X_\Sigma$ is the flat toric scheme associated to $\Sigma$ and $\TT=\Spet(\SS[M])$ is a flat torus, both defined over the sphere spectrum.
The category $\QCoh([X_\Sigma/\TT])$ has the standard tensor product while the category $\Shv(M_\RR;\Sp)$ has the convolution product as \cref{construction:convolution}.
One can explicitly describe the image of this functor:
\[
\im(\kappa)=\Shv_{\Lambda_\Sigma}(M_\RR;\Sp)\subseteq\Shv(M_\RR;\Sp).
\]
On the right-hand side is the subcategory of sheaves characterized by the following two conditions\footnote{It is possible to remove the assumption on constructibility once one has a good understanding of singular supports in greater generality, see \cref{dependence of singular support on stratification}.}: 
    
\begin{itemize}
\item It is constructible\footnote{Unless specified, we always mean constructible in the weak sense: there will be no constraints on the size of the stalk.} for the stratification $\Sc_\Sigma$ given by the affine hyperplane arrangement $H_\Sigma$, which is determined by $1$-cones in the fan $\Sigma$
\[H_\Sigma\defeq\{m+\sigma^\perp:m\in M,\sigma\in \Sigma(1)\}.\]
\item It has singular support contained in the conic Lagrangian $\Lambda_\Sigma$:
\[
\Lambda_\Sigma\defeq\bigsqcup_{m\in M;\sigma\in \Sigma} m+\sigma^\svee\times-\sigma\subseteq M_\RR\times N_\RR=T^*M_\RR.
\]
\end{itemize}
\end{mainthm}
\begin{proof}
The construction of $\kappa$ is a combination as in \cref{definition of the CCC functor} of \cref{theorem of combinatorics compares to quasicoherent} and \cref{construction of combinatorial to constructible comparison functor}. In \cref{comparison functor is fully faithful when the fan is smooth} we show that $\kappa$ is fully faithful. The characterization of the image of $\kappa$ is  \cref{Microlocal characterization of the image}.
\end{proof}
\begin{rem}
    This note treats the particular case of the coefficient ring being $\SS$, but we have used nothing about $\SS$ other than $\SS\in\CAlg(\Sp)$ is a connective commutative ring spectrum. One can similarly make relevant constructions over any connective commutative ring spectrum, and if one works with $\mathbb{C}$ this recovers a large category version of \cite{FLTZ}.
    On the other hand, such a lift to spectral coefficient is already hinted at implicitly in \cite{FLTZ} and explicitly in \cite{Vaintrob1}. Nevertheless, the construction of the symmetric monoidal structure on the functor seems new, even over the complex numbers.
    It should be noted that the compatibility of $\kappa$ with the convolution operation was formulated and used in \cite{FLTZ} in the context of dg-categories.
\end{rem}
We also provide compatibility of the functor $\kappa$ with the action of $\QCoh(B\TT)$ on both sides.
\begin{mainthm}
\label{mainthmB}
The functor $\kappa$ in \cref{equivariantmainthm} fits into a diagram of symmetric monoidal categories and symmetric monoidal functors:
\[\begin{tikzcd}
	{\QCoh([X_\Sigma/\TT)]} & {\Shv_{\Lambda_\Sigma}(M_\RR;\Sp)} \\
	{\QCoh(B\TT)} & {\Shv(M;\Sp)}
	\arrow["\kappa"', from=1-1, to=1-2]
	\arrow["{\pi^*}", from=2-1, to=1-1]
	\arrow["\simeq"', from=2-1, to=2-2]
	\arrow["{i_!}", from=2-2, to=1-2].
\end{tikzcd}\]
Both categories on the left have standard tensor product and both categories on the right have convolution product. The functor $\pi^*$ is $*$-pullback along the projection $\pi:[X_\Sigma/\TT]\rightarrow B\TT$. The functor $i_!$ is !-pushforward along the inclusion of the topological group $i:M\rightarrow M_\RR$. We used implicitly the  identification of symmetric monoidal categories
$$\QCoh(B\TT)\simeq\Fun(M;\Sp)\simeq\Shv(M;\Sp)$$
coming from \cref{theorem of quasicoherent sheaves of torus compare to graded spectra} and proof of \cref{equivariantization for large sheaf category}. Here the functor category has the Day convolution tensor product.
\end{mainthm}
\begin{proof}
    This is a combination of \cref{global combinatorial-coherent functor is compatible with torus action} and \cref{comb-constructible functor is compatible with lattice}.
\end{proof}
From this we may deduce some formal consequences. 
First of all, we apply the technique of de-equivariantization and obtain the following:
\begin{mainthm}[\cref{nonequivariant version of the correspondence}]
\label{nonequivariant main theorem}
There is a fully-faithful symmetric monoidal functor
$$\overline{\kappa}:\QCoh(X_\Sigma)\longrightarrow\Shv(M_\RR/M;\Sp).$$
The category $\QCoh(X_\Sigma)$ has the standard tensor product while $\Shv(\MR/M;\Sp)$ has the convolution product. The image of $\overline\kappa$ is described by constructibility and singular support similar to above:
$$\im(\overline\kappa)=\Shv_{\overline{\Lambda}_\Sigma}(M_\RR/M;\Sp)$$
where the right-hand side is the subcategory of sheaves constructible for $\overline{\Sc}_\Sigma$ with singular support contained in $\overline{\Lambda}_\Sigma$, where $\overline{\Sc}_\Sigma$ ($\overline{\Lambda}_\Sigma$) is the image of $\Sc_\Sigma$ (${\Lambda}_\Sigma$) under projection map $\pi:M_\RR\rightarrow M_\RR/M$ (see \cref{equivariantization for microlocal sheaf category}).
\end{mainthm}
As an application, one obtains Beilinson's theorem on the projective line (and projective spaces) by combining \cref{nonequivariant main theorem} and the exodromy equivalence.
\begin{mainthm}[\cref{beilinson}]
There is an equivalence of categories:
\[\QCoh(\mathbb{P}^1_\SS)\simeq \Fun(\bullet\rightrightarrows\bullet;\Sp).\]
\end{mainthm}
The de-equivariantization in \cref{nonequivariant main theorem} can be thought of as performing base change along the symmetric monoidal functor 
\[\colim:\Fun(\ZZ^n,\Sp)\longrightarrow\Sp.\]
More generally, one can obtain a relative version of toric constructions as in \cref{relative toric construction} by base changing along other colimit-preserving symmetric monoidal functors out of $\Fun(\ZZ^n;\Sp)$: in particular this recovers the result of the second named author and Pyongwon Suh \cite{hu2023coherentconstructiblecorrespondencetoricfibrations} relating quasi-coherent sheaves on a toric fibration to a category of sheaves on the topological torus with twisted coefficient category, see \cref{toric fibration}. 

\par
We also offer a conceptual framework for the `log-perfectoid mirror symmetry' as introduced by Dmitry Vaintrob, ensuring that (a large category version of) \cite[Theorem 2]{Vaintrob2} holds over $\SS$.
See \cref{Remark on idempotent algebra and log perfectoid quasicoherent of vaintrob} for the connection to his work on log quasi-coherent sheaves. This may serve as a motivation for Sasha Efimov's computation with continuous K-theory of $\Shv(\RR^1;\Sp)$ in \cite{efimov2024ktheorylocalizinginvariantslarge}. See \cite{baiburklund} for an expository account of these materials, where the first named author made some computation of Picard groupoid out of this with Robert Burklund.
\subsection{Inspirations and technicalities}
Needless to say, there have been numerous papers on this story and we could only mention an incomplete list of references in this introduction. We will now list some of them that inspired our project. We then provide some justifications for our (unfortunately, long) writing here.
Finally, we briefly mention some of the technical details, which should be interesting to the devoted readers.
\begin{rem}[Proof ideas from the literature]

Most of the ideas in this paper have appeared in one way or another in the literature.
The main proof method is to adapt the constructions of \cite{FLTZ} in the context of large categories, $\SS$ coefficient, and with symmetric monoidal structures.
The method of localization along idempotent algebras has been used in \cite{Kuwagaki_2020} disguised as the Tamarkin projector.
The proof we present for the characterization of the image in terms of singular supports is taken from \cite{zhou2017twistedpolytopesheavescoherentconstructible}.
Finally, the idea of applying de-equivariantization in this story was spelled out in \cite{shende2021toricmirrorsymmetryrevisited}.
\end{rem}
\begin{rem}[Dropping assumptions on smoothness and projectivity]
The restriction on the smoothness and projectivity of the fan is removed in \cite{Kuwagaki_2020}.
But we do not pursue the generality as in there.
\end{rem}
\begin{rem}[Necessity of higher algebra]
It is clear that in this story of coherent-constructible correspondence, higher categorical techniques are needed in constructing the functors and characterizing images.
Here we give a presentation without directly using model categories or dg-categories, of all the functors and categories. For comparison, it would be difficult to articulate the convolution product on the category of sheaves on a real vector space as a symmetric monoidal structure in terms of derived category of sheaves. This kind of difficulties would only add up when one works with spectral coefficient.
It appears to us that applying the language of higher algebra is the most convenient way to spell out the details.
\end{rem}
\begin{rem}[Large categories]
In this note, we systematically work with large (presentable stable) categories.
This approach simplifies certain constructions involving 'generators', which tend to be more intricate in small categories.
Another significant reason for maintaining this level of generality stems from our interest in $\Shv(\RR^n;\Sp)$: since it is not compactly generated, there is no obvious reason to hope for an algebro-geometric mirror object $Y$ such that
$$\QCoh(Y)\overset\simeq\longrightarrow\Shv(\RR^n;\Sp).$$
The sheaf category is however dualizable in the sense of \cite{efimov2024ktheorylocalizinginvariantslarge} with a presentably symmetric monoidal structure of convolution. Inspired by utility of such (dualizable but not compactly generated) categories in analytic geometry, one would hope to get a better understanding of them.
For example, Dmitry Vaintrob \cite{Vaintrob2} constructs an almost mathematics object $Y$ as a mirror for $\Shv(\RR^n;\Sp)$.
In other words, his `log-perfectoid' construction provides such $Y$ with $\QCoh(Y)\simeq\Shv(\RR^n;\Sp)$.
This should be thought of as algebraization of the sheaf category.
\end{rem}
\begin{rem}[Mirror symmetry over the sphere spectrum] 
It is widely expected that one can define a version of Fukaya categories over the sphere spectrum (see, for example, \cite{abouzaid2024foundationfloerhomotopytheory,lurie2018associativealgebrasbrokenlines,porcelli2024bordismflowmodulesexact,Jin_2024} for various perspectives of works towards a definition).
The $\ZZ$-linear equivalence between Fukaya categories and (microlocal) sheaf categories supplied by \cite{ganatra2023microlocalmorsetheorywrapped} should carry over to this new setting.
In particular, modeling the $\SS$-linear Fukaya category of $T^*\MR$ (with stop given by $\Lambda_\Sigma$) by $\Shv_{\Lambda_\Sigma}(\MR;\Sp)$, our result may be interpreted as an instance of $\SS$-linear mirror symmetry.
\end{rem}
\begin{rem}[Higher structures from mirror symmetry]
Another reason for us to implement mirror symmetry over $\SS$ is the hope that it would motivate constructions in category theory and homotopy theory.
A wonderful example of such an advance is provided in \cite{LurieRotation} where Jacob Lurie made the observation that Waldhausen's $S$-construction is corepresented by cosimplicial objects $\mathrm{Quiv}^\bullet$ and this family of objects has certain coparacyclic structure.
In fact, symplectic geometry provides a motivation for such an observation (see \cite[Section 1.2]{tanaka2019cyclicstructuresbrokencycles} for more on this): each $\mathrm{Quiv}^n$ could be seen (after 2-periodization) as an $\mathbb{S}$-linear topological Fukaya category on the $2$-dimensional disc with $n+1$ stoppings on the boundary, and the (para)cyclic symmetry comes from rotations of the disc. The actual construction of $\mathrm{Quiv}^\bullet$ however, runs on the `mirror' side, i.e., as the category of matrix factorizations in spectral algebraic geometry.
It is possible to relate the content of this note to the above story in the following way: it was explained in \cite{nadler2016wrappedmicrolocalsheavespairs,ganatra2023microlocalmorsetheorywrapped} that topological Fukaya categories could be modeled locally, by (the microlocalization of) the category of sheaves with prescribed singular supports.
We hope that the description of such categories in terms of algebraic geometry might help with constructions of higher structures, such as those suggested in \cite{tanaka2019cyclicstructuresbrokencycles,Dyckerhoff_2021}.
Regrettably, our grasp of symplectic geometry limits our ability to offer further insights.
\end{rem}
Now we highlight some technicalities in the paper that we find interesting.
\begin{rem}[Strategy for the construction of $\kappa$]
The idea behind the construction of the functor $\kappa$ comes in two parts.
\par
First, we construct $\kappa$ for \stress{affine} toric varieties indexed by $\sigma\in\Sigma$.
This is implemented by the following correspondence:
$$\QCoh([X_\sigma/\TT])\overset{\simeq}\longleftarrow\Psh{\Theta(\sigma)}{\Sp}\longrightarrow\Shv(M_\RR;\Sp)$$
where the functor on the right is lax symmetric monoidal and fully faithful.
The middle category is the presheaf category on a symmetric monoidal 1-category (of combinatorial nature).
With the universal properties of Day convolution monoidal structures,
it suffices to construct symmetric monoidal functors out of $\Theta(\sigma)$ - which is still a laborious work: see the following remarks for a quick idea of how to write down these functors.
With the functors at hand, one can follow the arguments from \cite{Moulinos} to prove that the functor on the left is an equivalence. 

The second step involves \stress{gluing}: for the inclusion of cones $\sigma\subseteq \tau$, there is a symmetric monoidal restriction functor 
$$\QCoh([X_\tau/\TT])\longrightarrow\QCoh([X_\sigma/\TT]).$$
One can think of this as a diagram indexed by $\sigma\in\Sigma^\op$ and Zariski descent implies that the limit of this diagram is the category  $\QCoh([X_\Sigma/\TT])$.
The construction in the first step is compatible with the restriction functor, allowing one to take limits on the sheaf category side to obtain the functor $\kappa$.
\end{rem}
\begin{rem}[Constructing functors into $\QCoh$] A prototypical example of the functor we will construct that maps into $\QCoh([X_\sigma/\TT])$ is the symmetric monoidal functor
$$\Fun(\ZZ_\leq;\Sp)\longrightarrow\QCoh([\AA^1/\GG_m])$$
which classifies the universal line bundle $\Oc(1)$ and the universal section $\cdot x:\Oc\rightarrow\Oc(1)$ (see \cite{Moulinos}). Note that this says in particular that the line bundle $\Oc(1)$ is a strict Picard element as in \cite{carmeli2022strictpicardspectrumcommutative}. See \SAG{Warning}{5.4.3.3} for more on this notion of strictness. Our method of construction passes through an unstable (set-valued, actually) model of such data, which provides an alternative construction of the functor in the proof of \cite[Theorem 4.1]{Moulinos}. We will also construct a slightly generalized version of this with target being $\QCoh([\AA^n/\GG_m^n])$.
\end{rem}
\begin{rem}[Constructing functors into $\Shv$] A prototypical example of the functor we will construct that maps into $\Shv(M_\RR;\Sp)$ (equipped with convolution) is a lax symmetric monoidal functor
$$\Fun(\ZZ_\leq;\Sp)\longrightarrow\Shv(\RR^1;\Sp)$$
which sends $n\in \ZZ$ to the dualizing sheaf on the open half line $\omega_{(-\infty,n)}$.
This is achieved by making a more general construction: given a commutative monoid $M$ in $\LCH$, we construct a lax symmetric monoidal structure on the relative homology functor taking a pair $(X,f:X\rightarrow M)$ to $f_!f^!\omega_M$.
With this functor at hand, the problem is reduced to $1$-categorical manipulations.
This general construction is very much inspired by \cite[Chapter 3]{GL}, and we believe that it has other interesting uses.
\end{rem}
\begin{rem}[Gluing in $\Shv$]
To make the gluing procedure precise, we will prove a sheaf-theoretic counterpart of Zariski descent in $\Shv(\RR^n;\Sp)$.
This is implemented with idempotent algebras as in \HA{Definition}{4.8.2.8}.
For example, the dualizing sheaf $\omega_{(-\infty,0)}$ is an idempotent algebra in $\Shv(\RR^1;\Sp)$ for the convolution product, and this phenomenon generalizes to other cones. Given a smooth projective fan $\Sigma$, we will produce a collection of idempotent algebras in $\Shv(\RR^n;\Sp)$ and show that their meet is the unit $\mathbb{1}_{\Shv(\RR^n;\Sp)}$ as an idempotent algebra.
\end{rem}
\begin{rem}[Singular supports for polyhedral sheaves] To characterize the image of $\kappa$, we make use of the recent advances \cite{clausenJansen2023reductiveborelserre,haine2024exodromyconicality} of the exodromy equivalence with large category of constructible sheaves. Following \cite{FLTZ}, we supply a definition of singular support for sheaves constructible for affine hyperplane arrangements - via the Fourier-Sato transform (compare the general definition laid out in \cite{Jin_2024}).
This definition of singular support provides a convenient setup for us to apply the non-characteristic deformation lemma \cite{robalo2016lemmamicrolocalsheaftheory} and obtain the characterization of image of $\kappa$. The proof presented here supplements some details missing in \cite{zhou2017twistedpolytopesheavescoherentconstructible}.
\end{rem}
\subsection{Acknowledgments}
The inception of this project dates back to 2022, when YH traveled to Copenhagen and shared a roof with QB, for \emph{Masterclass: Cluster Algebra and Representation Theory} hosted by the Copenhagen Centre for Geometry and Topology. This document would not exist without the encouragement of Shachar Carmeli.
We thank Robert Burklund, Peter Haine, Maxime Ramzi, and Jan Steinebrunner for their time and patience in answering our questions.
Many people have taken their time to listen to the progress and outcome of this writing, including Dustin Clausen, Sasha Efimov, Lars Hesselholt, Gus Schrader, and Hiro Lee Tanaka, and we are grateful for their interest and comments.
QB was supported by the Danish National Research Foundation through the Copenhagen Centre for Geometry and Topology (DNRF151). YH was supported by the NSF grant DMS 2302624.
\subsection{Conventions}
\begin{notation}[Category theory]
    We don't touch on set-theoretic issues in this note. We write $\Cat$ for the $(\infty,1)$-category of quasi-categories, functors, natural isomorphisms and so on. We refer to objects in $\Cat$ as 'categories' to avoid putting $\infty$ in front of everything. This however makes us write `stable category' instead of more established name `stable $\infty$-category'.
    We identify a $1$-category with its nerve in $\Cat$ and stress that it is $1$-category when we have one. We write $\Spc$ for the category of spaces (or homotopy types, or anima) and $\Sp$ for the stable category of spectra. We write $\Map$ for mapping spaces in a category and $\map$ for mapping spectra in a stable category.
\end{notation}
\begin{notation}[Simplicial stuff]
    By $\Delta$ we mean (a skeleton of) the (1-)category of non-empty ordered finite sets and order preserving maps between them. A (co)simplicial diagram in $\Cc$ is a functor from $(\Delta)\Delta^\op$ to $\Cc$. We only draw face maps when visualizing a (co)simplicial diagram. We write $d^i$ for the structure (face) maps in a cosimplicial diagrams.
\end{notation}
\begin{notation}[Symmetric monoidal categories] We write $(\Cc,\otimes)$ for a symmetric monoidal category and often refer to $\Cc$ as a symmetric monoidal category, omitting the monoidal structure. We write $\Cc^\otimes$ for the underlying operad of $(\Cc,\otimes)$. We write $\CAlg(\Cc,\otimes):=\Alg_{\mathbb{E}_\infty}(\Cc^\otimes)$ for the category of $\EE_\infty$-algebras in $\Cc$. And when there is no danger of confusion, we will omit the monoidal structure and write $\CAlg(C)$. For example, $\CAlg(\mathrm{Sp})$ would refer to the category of $\mathbb{E}_\infty$-ring spectra.  In the special case of Set or Spc equipped with Cartesian symmetric monoidal structure, we also write CMon for the category of commutative monoids and CGrp for the category of commutative groups. All the presheaf categories are assumed to carry Day convolution structure when considered as a symmetric monoidal category.
\end{notation}
\begin{notation}[(Lax) symmetric monoidal functors]
    For two symmetric monoidal categories $\Cc$ and $\Dc$, we write $\Fun^\otimes(\Cc,\Dc)$ for the category of symmetric monoidal functors from $\Cc$ to $\Dc$. We write $\Fun^{\lax\otimes}(\Cc,\Dc)$ for the category of symmetric monoidal functors from $\Cc$ to $\Dc$.
    We write $\SMCat$ for the category of symmetric monoidal categories and (strongly) symmetric monoidal functors between them.
\end{notation}
\begin{notation}[Algebraic geometry]
    We approach spectral algebraic geometry through the functor of points. We write $\Stk$ for the full subcategory of fpqc sheaves inside $\Fun(\CAlg^\mathrm{cn},\Spc)$ (what's better, the objects we are dealing with in this note are all geometric stacks in the sense of \SAG{Definition}{9.3.0.1}), and we write $\Spet(-)$ for the Yoneda functor $\CAlg^{\mathrm{cn},\op} \rightarrow\Fun(\CAlg^\mathrm{cn},\Spc)$ which factors through $\Stk$ (In SAG, $\Spet$ was used for another construction, but Lurie has provided a comparison with this Yoneda point of view in \SAG{Proposition}{1.6.4.2}).
\end{notation}
\begin{notation}[Topological spaces]
    We write $\LCH$ for the (1-)category of locally compact Hausdorff space and continuous maps between them. However, all topological spaces of interest in this note are finite dimensional manifolds. We often write $j_U:U\rightarrow X$ for the inclusion of an open subset and $i_Z:Z\rightarrow X$ for the inclusion of a closed subset. We say a map $f:X\rightarrow Y$ is \'etale if it is a local homeomorphism.
\end{notation}
\begin{notation}[Sheaf theory]
\label{convention on sheaf theory}
It will be very convenient for us to extract a `six-functor formalism' out of \cite{volpe2023sixoperation} on the category of locally compact Hausdorff topological spaces. We write $\Shv(X;\Sp)$ for the category of sheaves of spectra on a locally compact Hausdorff topological space $X$, and we write $f^*\dashv f_*$, $f_!\dashv f^!$ and $\otimes\dashv\Hom$ for the six functors that come with it.
For an open $U\subseteq X$, we write $\underline\SS_U\in\Shv(X;\Sp)$ for the sheafification of the $\SS$-linearized presheaf represented by $U$. In other words, if we write $j_U:U\rightarrow X$ for the inclusion map and $\underline\SS\in\Shv(U;\Sp)$ for the constant sheaf valued at $\SS$, $\underline\SS_U$ is equivalently 
    $$\underline\SS_U\defeq j_{U!}\underline\SS\in\Shv(X;\Sp)$$
    and we abusively call it the representable sheaf on $U$. Note that $\underline\SS_X$ is just constant sheaf valued at $\SS$ on $X$. Similarly for a closed subset $Z\subseteq X$ we write 
    $$\underline\SS_Z\defeq i_{Z*}\underline\SS\in\Shv(X;\Sp).$$
    
    We reserve the symbol $\omega$ for \stress{dualizing sheaves}. Let $p:X\rightarrow*$ be the canonical map to the final object. The dualizing sheaf of X is defined to be
    $$\omega_X\defeq p^!\underline\SS\in\Shv(X;\Sp).$$
    Similarly, when we work with an open subset $U$ or closed subset $Z$ in $X$, we write
    \[\omega_U\defeq j_{U!}j_U^!\omega_X\in\Shv(X;\Sp)\]
    and
    \[\omega_Z\defeq i_{Z!}i_Z^!\omega_X\in\Shv(X;\Sp).\]
    
\end{notation}
\section{Combinatorial model}
In \cite[Section 3]{FLTZ} the authors defined a poset $\Gamma(\Sigma,M)$ that interpolates between the category of quasi-coherent sheaves and the category of constructible sheaves. In this section, we recall their definition and present functoriality of the definition. Refer to \cref{notation:fan} for explanations of lattices, cones, fans, and associated topics if you are unfamiliar with them.

\begin{defn}[Poset of cones]
    Given a pair of lattice and fan $(N,\Sigma)$, we \stress{consider $\Sigma$ as a poset} as follows: the objects of $\Sigma$ are cones $\sigma\in\Sigma$ and morphisms between two cones are inclusions.
\end{defn}
\begin{defn}
    Let $\Closed(M_\RR)$ be the poset of closed polyhedral subsets in $\MR$ (the subsets that can be written as a Minkowski sum of a polytope and a polyhedral cone) with the morphisms being inclusions. This is a symmetric monoidal 1-category if we take \stress{Minkowski sum} $+$.
\end{defn}
\begin{defn}[The $\Theta$ category]
    Fix a cone $\sigma\subset N_\mathbb{R}$, the (1-)category $\Theta(\sigma)$ is defined as the full subcategory of posets of closed subsets in $M_\RR$:$$\Theta(\sigma)\subseteq\Closed(M_\mathbb{R}).$$ It is spanned by objects of the form $m+\sigma^{\svee}$ for $m\in M$. 
\end{defn}
Observe that the association $\sigma\mapsto\Theta(\sigma)$ is functorial in $\sigma$ that it assembles into a functor $$\Theta(-):\Sigma^{\op}\rightarrow\Cat.$$
Given an inclusion $i:\sigma\rightarrow\tau\in\sigma$ of cones, the induced functor is 
\begin{align*}
\Theta(i)\colon \Theta(\tau) & \rightarrow\Theta(\sigma) \\
m+\tau^\svee & \mapsto m+\sigma^{\svee}.
\end{align*}
\begin{rem}[Symmetric monoidal structures on $\Theta(-)$] We make the following observations: 
\label{alternative way to see symmetric monoidal structure on Theta}
\begin{enumerate}
    \item Each $\Theta(\sigma)$ has the structure of a symmetric monoidal (1-)category. This could be obtained by observing that as a full subcategory, $\Theta(\sigma)$ inherits a (non-unital) symmetric monoidal structure from the symmetric monoidal category $(\Closed(M_\RR),+)$. To make it unital, it suffices to note that $\sigma^\svee\in\Theta(\sigma)$ acts as a tensor unit. 
    \item We might also observe that $\sigma^\svee$ is an idempotent algebra in $(\Closed(M_\RR),+)$ and define $\Theta(\sigma)$ to be a full subcategory of $\Mod_{\sigma^\svee}(\Closed(M_\RR))$, and it directly follows that $\Theta(\sigma)$ inherits a symmetric monoidal structure.
    \item For each inclusion $i:\sigma\rightarrow\tau$, $\Theta(i)$ has the structure of a symmetric monoidal functor which can be observed directly since we are working with posets: there is no coherence issue. \emph{In conlusion}, $\Theta(-)$ lifts to a functor $\Sigma^\op\rightarrow\SMCat$.
    \item For later use, consider the discrete category of $M$ with symmetric monoidal structure given by addition. There are symmetric monoidal functors \[p_\sigma:M\longrightarrow\Theta(\sigma):m\mapsto m+\sigma^\svee\] and they assemble into a natural transformation between diagrams in $\SMCat$ indexed by $\sigma\in \Sigma^\op$ where the source is thought of as a constant diagram.
\end{enumerate}

\end{rem}

\begin{rem}[Comparison with other models]
    Our definition of $\Theta(-)$ works cone by cone, while in \cite[Section 5]{Vaintrob1}\cite[Section 3]{FLTZ} global categories were proposed. Later we will see that one wants to compute 
    $$\lim_{\Sigma^\op}\Fun(\Theta(-)^\op,\Sp).$$
    An explicit presentation of the limit of such a diagram of presheaf categories (with arrows given by left Kan extensions of functors) remains unclear to us. However, `(co)sheaves for Morelli topology' as in \cite[Section 6]{Vaintrob1} seems like a combinatorial presentation of the limit.
\end{rem}
\section{Toric geometry}
Classically, toric geometry builds on the linearization functor $\mathbb{Z}[-]:\CMon(\Set)\rightarrow\CAlg(\mathrm{Ab})$. For example, $\mathbb{Z}[\mathbb{N}]=\mathbb{Z}[X]$ is the one-variable polynomial ring. Toric schemes are constructed from the $\Spec(-)$ of these monoid algebras by gluing along maps coming from $\CMon(\Set)$. In this section we present some basic materials on \stress{flat} toric geometry over $\SS$
\footnote{While it's possible to make sense of, say, a non-flat $\mathbb{P}^1$ as in \SAG{Construction}{19.2.6.1}, we don't know of a general construction of non-flat toric schemes.}
. Although we will not prove it, the toric schemes we define here will be flat over the base ring $\SS$. The adjective `flat' is also reminiscent of the fact that after base changing to $\ZZ$,
it recovers the classical construction of toric schemes, which are flat over $\ZZ$. The idea of looking at flat toric scheme over $\SS$ is certainly well-known, going back to \cite{LurieEllipticSurvey} and \SAG{Remark}{5.4.1.9}. Most of the discussions here would be rather formal: we are mainly interested in the category of quasi-coherent sheaves and related categorical nonsense.\par
In the first part, we fix notations for the toric construction and explain how the action diagram presents the quotient stack by the torus action. In the second part, we recall the functoriality of quasi-coherent sheaves and provide an unstable model for quasi-coherent sheaves on the quotient stack. This is used in the third part, where we construct combinatorial-coherent comparison functor. Finally we follow the approach of \cite{Moulinos} to show that this functor is an equivalence.
\subsection{Recollections on toric geometry}
\begin{notation} \label{notation:fan}We recall the following notations useful in the combinatorics of toric varieties.
\begin{itemize}
    \item A \stress{lattice} is a finitely generated free abelian group $N\in \mathrm{Ab}=\CGrp(\Set)$.
    \item The \stress{dual lattice} $M$ of $N$ is $M\defeq\Hom_\mathrm{Ab}(N, \ZZ)\in\mathrm{Ab}$.
    \item A \stress{cone} $\sigma \subset N_{\RR} \defeq N \otimes_{\ZZ} \RR$ for us is  a  rational polyhedral cone in $N_\RR$.
    \item The \stress{dual cone} of $\sigma\subset N_\RR$ is $\sigma^\svee\defeq\{m \in M_\RR : \langle m, n\rangle \geq 0, \forall n \in \sigma\}\subseteq M_\RR.$
    \item A \stress{fan} $\Sigma$ in $N$ is a collection of strongly convex cones in $N$ closed under taking faces,
    such that every pair of cones either are disjoint or meet along a common face.
    \item A fan $\Sigma$ is \stress{smooth} \cite[Section 2.1]{fulton} if each of the cone $\sigma$ is spanned by part of a basis of $N$.
    \item A fan $\Sigma$ is \stress{projective} \cite[Section 3.4]{fulton} if it admits an integral moment polytope $P\subset M_\RR$: a polytope such that its faces are in bijection with cones in $\Sigma$ and the cone $\sigma$ corresponding to a face $F$ is precisely the dual cone of the angle spanned by $P$ along $F$.
\end{itemize}
\end{notation}
\begin{construction}[Flat toric scheme]
    Given a pair $(N, \Sigma)$ of lattice and fan.
    The assignment \[
    \sigma \mapsto S_\sigma \defeq \sigma^\svee \cap M \in \CMon(\Set)
    \]
    gives rise to a functor $\Sigma^\op \rightarrow \CMon(\Set) = \CAlg(\Set)$.
    On the other hand, the symmetric monoidal functors $$\Set \into \Spc \xrightarrow{\splus} \Sp$$
    induce a functor $\SS[-]: \CAlg(\Set) \rightarrow \CAlg(\Sp)$.
    Consider the image of $\sigma$ under this composite functor
    \[
    \Oc_\sigma \defeq \SS[\sigma^\svee \cap M]\in\CAlg(Sp)
    \]
    which should be thought of as the ring of functions on the affine toric scheme $X_\sigma$ associated to the cone $\sigma$.
    Postcomposing with $\Spet$, we get a functor $\Sigma\longrightarrow\Stk$:
    $$\sigma\mapsto\Spet(\Oc_\sigma).$$
    The \stress{flat toric scheme $X_\Sigma$ associated to $(N,\Sigma)$} is defined to be the colimit of this diagram
    \[
    X_\Sigma \defeq \colim_{\Sigma} \Spet \Oc_\sigma\in \Stk.
    \]
    computed in the category of  stacks.
\end{construction}

\begin{rem}[An alternative version of `toric geometry']
    Motivated by the fact that $\mathbb{N}^{\times k}$ is the free object on $k$ points in CMon(Set) (and similarly $\mathbb{Z}^{\times k}$ is the free object on $k$ points in CGrp(Set)), one might want to reimagine a toric geometry over the sphere spectrum building upon monoid algebra of free objects in CMon(Spc) (or CGrp(Spc)). We don't carry out the construction here, but only point out the following subtleties:
    \begin{enumerate}
        \item The natural numbers $\mathbb{N}$ (resp. $\mathbb{Z}$) is the free $\mathbb{E}_1$-monoid (resp. $\mathbb{E}_1$-group) on a point.
        However, when viewed as an $\mathbb{E}_\infty$-monoid, $\mathbb{N}$ is far from being a free object: a map in $\CMon(\Spc)$ from $\mathbb{N}$ instead picks out a `strictly commutative element' in the target.
        \item The flat affine line $\text{Sp\'et}(\mathbb{S}[\mathbb{N}])$ doesn't support the addition map, see \cite[Section 3.5]{LurieElliptic1}.
    \end{enumerate}
\end{rem}
\begin{example}[Flat tori over $\SS$]
    If one picks the fan to consist only of the origin, the associated flat toric scheme (named $\TT$) is the \stress{torus associated to $M$}:
    $$\TT\defeq\Spet(\SS[M]).$$
    Note that $\TT$ has the structure of a group object (and we will call it a group scheme) given that $M$ is a cogroup object in $\CMon(\Spc)$.
\end{example}
Recall that a toric variety over a field $k$ contains a torus as an open-dense subset and the torus action extends continuously to the whole variety. Now we explain the torus action in the setting of flat toric geometry.

\begin{construction}\label{construction: torus action}
    Recall that given a category $\Cc$ with all limits, and considering $\Cc$ as a Cartesian symmetric monoidal category,
    every object $X\in \Cc$ acquires a canonical commutative coalgebra structure,
    informally specified by regarding the diagonal as the comultiplication map
    \[
    \Delta: X \rightarrow X\times X.
    \]
    In particular, every map $f:Y \rightarrow X$ exhibits $Y$ as a comodule over $X$,
    with the coaction map informally specified by
    \[
    \mu: Y \xrightarrow{\Delta} Y \times Y \xrightarrow{(f, \id)} X \times Y.
    \]
    In fact this map is induced by the lift of $f:Y\rightarrow X$ to a map of coalgebras. Specializing to the situtation $\Cc = \CMon(\Set)$
    \footnote{Note that $\CMon(\Set)$ is preadditive.},
    we see that every submonoid $S_\sigma$ of $M$ is canonically coacted on by $M$. Moreover, these coactions are compatible with inclusions among $S_\sigma$. Therefore, $\Oc_\sigma = \SS[S_\sigma]$ acquires a canonical $\SS[M]$-comodule structure.
    Further passing to $\Spet$,
    this gives a compatible family of actions of the group scheme $\TT=\Spet \SS[M]$ on $\Spet \Oc_\sigma$,
    each encoded by a simplicial diagram
\[\begin{tikzcd}
	\cdots & {\Spet \Oc_\sigma \times \TT \times\TT} & {\Spet \Oc_\sigma \times \TT} & {\Spet \Oc_\sigma}
	\arrow[shift right=2, from=1-2, to=1-3]
	\arrow[shift left=2, from=1-2, to=1-3]
	\arrow[from=1-2, to=1-3]
	\arrow[shift left, from=1-3, to=1-4]
	\arrow[shift right, from=1-3, to=1-4]
	\arrow[shift left=3, from=1-1, to=1-2]
	\arrow[shift right=3, from=1-1, to=1-2]
	\arrow[shift right, from=1-1, to=1-2]
	\arrow[shift left, from=1-1, to=1-2].
\end{tikzcd}\]
Taking colimits along $\sigma$, we obtain the diagram
\[\begin{tikzcd}
	\cdots & {\left(\colim_\sigma\Spet \Oc_\sigma \right) \times \TT \times\TT} & {\left(\colim_\sigma\Spet \Oc_\sigma \right) \times \TT} & {\colim_\sigma\Spet \Oc_\sigma}
	\arrow[shift right=2, from=1-2, to=1-3]
	\arrow[shift left=2, from=1-2, to=1-3]
	\arrow[from=1-2, to=1-3]
	\arrow[shift left, from=1-3, to=1-4]
	\arrow[shift right, from=1-3, to=1-4]
	\arrow[shift left=3, from=1-1, to=1-2]
	\arrow[shift right=3, from=1-1, to=1-2]
	\arrow[shift right, from=1-1, to=1-2]
	\arrow[shift left, from=1-1, to=1-2],
\end{tikzcd}\]
because colimits are universal in $\Stk$.
\footnote{In particular, taking colimits commutes with taking finite products.}
We therefore obtain an action of $\TT$ on
\[
    X_\Sigma = \colim_{\sigma\in\Sigma} \Spet \Oc_\sigma,
\]
to which we refer as \stress{the torus action} of $\TT$ on $X_\Sigma$,
and the corresponding simplicial diagram
$(X_\Sigma//\TT)_\bullet$
the \stress{action diagram} of $\TT$ on $X_\Sigma$. Formally, one might think of each $\Spet\Oc_\sigma$ as an object in $\Mod_{\Spet(\SS[M])}\Stk$ and take colimits along $\Sigma$ in the module category. Given that the forgetful functor 
\[\Mod_{\Spet(\SS[M])}(\Stk)\longrightarrow\Stk\]
commutes with colimits, one sees that $X_\Sigma$ acquires an action of $\TT$, and its action diagram can be identified with the above action diagram.
\end{construction}
\begin{defn}
The quotient stack $[X_\Sigma/\TT]$ is the geometric realization of the action diagram of $\TT$ on $X_\Sigma$:
\[
[X_\Sigma/\TT] \defeq \colim_{\Delta^\op}
\left(
\begin{tikzcd}[cramped]
	\cdots & {X_\Sigma\times \TT \times\TT} & {X_\Sigma \times \TT} & {X_\Sigma}
	\arrow[shift right=2, from=1-2, to=1-3]
	\arrow[shift left=2, from=1-2, to=1-3]
	\arrow[from=1-2, to=1-3]
	\arrow[shift left, from=1-3, to=1-4]
	\arrow[shift right, from=1-3, to=1-4]
	\arrow[shift left=3, from=1-1, to=1-2]
	\arrow[shift right=3, from=1-1, to=1-2]
	\arrow[shift right, from=1-1, to=1-2]
	\arrow[shift left, from=1-1, to=1-2]
\end{tikzcd}
\right)\in\Stk.
\]
\end{defn}
\begin{rem}
The \v{C}ech nerve of the projection $X_\Sigma \rightarrow [X_\Sigma/\TT]$
is canonically identified with the action diagram of $\TT$ on $X_\Sigma$.
This is a direct consequence of \cref{Grouplike action gives groupoid object} and the fact that every groupoid object in an $\oo$-topos is effective \HTT{Theorem}{6.1.0.6}.
\end{rem}

\begin{rem}
    Alternatively, one might take the quotient affine locally on each $X_\sigma$ by defining
    $$[X_\sigma/\TT]\defeq\colim_{\Delta^\op}
\left(
\begin{tikzcd}[cramped]
	\cdots & {X_\sigma\times \TT \times\TT} & {X_\sigma \times \TT} & {X_\sigma}
	\arrow[shift right=2, from=1-2, to=1-3]
	\arrow[shift left=2, from=1-2, to=1-3]
	\arrow[from=1-2, to=1-3]
	\arrow[shift left, from=1-3, to=1-4]
	\arrow[shift right, from=1-3, to=1-4]
	\arrow[shift left=3, from=1-1, to=1-2]
	\arrow[shift right=3, from=1-1, to=1-2]
	\arrow[shift right, from=1-1, to=1-2]
	\arrow[shift left, from=1-1, to=1-2]
\end{tikzcd}
\right)\in\Stk$$
    via the action diagram. Then one can perform gluing
    $$[X_\Sigma/\TT]=\colim_{\sigma\in\Sigma}[X_\sigma/\TT]$$
    and obtain the same stack, since colimit commutes with colimit.
\end{rem}

\subsection{Quasi-coherent sheaves}
\label{section on quasi-coherent sheaves}
The functor  $$\QCoh:\Stk^\op\rightarrow\Cat$$
given in \SAG{Definition}{6.2.2.1} is lax symmetric monoidal in view of \cite[\SAGsubsec{6.2.6}]{SAG} and \HA{Theorem}{2.4.3.18}. To each stack it assigns a symmetric monoidal category:
\[X\mapsto\QCoh(X)\in\SMCat\]
such that to affines $\Spet(R)$ it assigns $\QCoh(\Spet(R))\simeq\Mod_R(\Sp)$.
In fact this functor preserves limits as in \SAG{Proposition}{6.2.3.1}, hence one gets a presentation of quasi-coherent sheaves on quotient stack as
$$\QCoh([X_\Sigma/\TT])\simeq\lim_{\Sigma^\op}\QCoh([X_\sigma/\TT])$$
while in turn each piece is presented by
$$\QCoh([X_\sigma/\TT])\simeq\lim_{\Delta}\left(
\begin{tikzcd}
	\cdots & {\QCoh(X_\sigma\times \TT \times\TT)} & {\QCoh(X_\sigma \times \TT)} & {\QCoh(X_\sigma)}
	\arrow[shift left=2, from=1-3, to=1-2]
	\arrow[shift right=2, from=1-3, to=1-2]
	\arrow[shift right, from=1-4, to=1-3]
	\arrow[shift left, from=1-4, to=1-3]
	\arrow[shift right=3, from=1-2, to=1-1]
	\arrow[shift left=3, from=1-2, to=1-1]
	\arrow[shift left, from=1-2, to=1-1]
	\arrow[shift right, from=1-2, to=1-1]
	\arrow[from=1-3, to=1-2]
\end{tikzcd}\right).
$$
Note that this is actually a limit of symmetric monoidal categories \cite[\SAGsubsec{6.2.6}]{SAG}. At first glance, it might seem difficult to write down objects explicitly in this category. Motivated by \SAG{Construction}{5.4.2.1}, we proceed by making the following unstable construction.
\begin{construction}[\stress{Unstable} analogue] Fix a cone $\sigma$ in a lattice $N$, recall that \cref{construction: torus action} provides an coaction of $M$ on $S_\sigma=\sigma^\svee\cap M$. The coaction is presented by the following cosimplicial diagram in $\CMon(\Spc)$:
\[\begin{tikzcd}
	\cdots & {S_\sigma\times M\times M} & {S_\sigma\times M} & {S_\sigma}.
	\arrow[shift left=2, from=1-3, to=1-2]
	\arrow[shift right=2, from=1-3, to=1-2]
	\arrow[shift right, from=1-4, to=1-3]
	\arrow[shift left, from=1-4, to=1-3]
	\arrow[shift right=3, from=1-2, to=1-1]
	\arrow[shift left=3, from=1-2, to=1-1]
	\arrow[shift left, from=1-2, to=1-1]
	\arrow[shift right, from=1-2, to=1-1]
	\arrow[from=1-3, to=1-2]
\end{tikzcd}\]
Passing to module categories (with the functors being extension-of-scalar), one obtains
\[\begin{tikzcd}
	\cdots & {\Mod_{S_\sigma\times M\times M}(\Spc)} & {\Mod_{S_\sigma\times M}(\Spc)} & {\Mod_{S_\sigma}(\Spc)}.
	\arrow[shift left=2, from=1-3, to=1-2]
	\arrow[shift right=2, from=1-3, to=1-2]
	\arrow[shift right, from=1-4, to=1-3]
	\arrow[shift left, from=1-4, to=1-3]
	\arrow[shift right=3, from=1-2, to=1-1]
	\arrow[shift left=3, from=1-2, to=1-1]
	\arrow[shift left, from=1-2, to=1-1]
	\arrow[shift right, from=1-2, to=1-1]
	\arrow[from=1-3, to=1-2]
\end{tikzcd}\]
This is a cosimplicial diagram of symmetric monoidal categories, and we write $\Mod_{S_\sigma}(\Spc)^M$ for the limit.
\end{construction}
\begin{rem}
Replacing $\Spc$ by $\Set$ in the above, we get a new category which we name as $\Mod_{S_\sigma}(\Set)^M$. Recall from \cite[Theorem 2.2]{hovey2002morita} that in classical algebraic geometry, the 1-category of comodules over a Hopf algebroid $(A,\Gamma)$ is the same as the 1-category of quasi-coherent sheaves on the stack $X_{(A,\Gamma)}$ it presents.
In a similar way, one can think of the pair $(S_\sigma,S_\sigma \times M)$ as a Hopf algebroid in the symmetric monoidal category $\CMon(\Set)$.\footnote{Recall that, traditionally, a Hopf algebroid is defined in the symmetric monoidal category $\Mod_R$.}
The limit definition of $\Mod_{S_\sigma}(\Set)^M$ mimics the definition of the category quasi-coherent sheaves on a stack.
Indeed, we can think of $\Mod_{S_\sigma}(\Set)^M$ as the category of comodules over the Hopf algebroid $(S_\sigma,S_\sigma\times M)$.
The object in such comodule category can be constructed with a finite amount of data while the category $\Mod_{S_\sigma}(\Set)^M$ relates to the category of quasi-coherent sheaves, as we explain below.
\end{rem}
\begin{rem}[1-categorical analogue and degeneracy]
Consider the limit presentation of $\Mod_{S_\sigma}(\Set)^M$. As the categories involved are all 1-categories, the limit is canonically identified with the limit of the diagram restricted to $\Delta_{\leq2}$ (see \cite[\href{https://arxiv.org/pdf/2207.09256v4.pdf\#appendix.A}{Proposition A.1}]{Hesselholt-Pstragowski}.
Note also that one can produce objects and morphisms in the limit with a finite amount of data (actually very little is needed).
More precisely, consider a cosimplicial diagram of $1$-categories $\Cc_\bullet$, the limit is still a $1$-category whose objects are pairs $(x,f)$ where $x$ is an object in $\Cc_0$, $f:d^1x\rightarrow d^0x$ is an isomorphism in $\Cc_1$ such that $d^0f\circ d^2f=d^1f$ in $\Cc_2$. A map from $(x,f)$ to $(y,g)$ is a map $\varphi:x\rightarrow y$ in $\Cc_0$ that commutes with the structure maps $f$ and $g$.
\end{rem}
\begin{example}[Explicit construction of the tautological object in $\Mod_{S_\sigma}(\Set)^M$]
\label{M as an object in the unstable quasicoherent category}
    Here we provide a concrete example of how to down objects in the category $\Mod_{S_\sigma}(\Set)^M$. We supply a particular lift\footnote{There are obviously others.} of $M$  to an object in this category, where $M$ is an $S_\sigma$-module via the canonical inclusion. To provide the lift is to provide the structure map (note that the relative tensor products are induced by different maps $S_\sigma\rightarrow S_\sigma\times M$ where the left one is $(\id,\mathrm{inclusion})$ and the right one is $(\id,0)$)
    $$f:M\times_{S_\sigma} S_\sigma \times M\longrightarrow M\times_{S_\sigma}S_\sigma\times M\in \Mod_{S_\sigma\times M}(\Set)$$
    which is an isomorphism and we define $f$ such that
    $$f(m,s,n)\defeq (m,s,n+m)\in M\times_{S_\sigma}S_\sigma\times M.$$
    It is a tedious exercise to check that $f$ satisfies the cocycle conditions as above and we leave it to the reader.
    Upon linearization, this object corresponds to the $*$-pushforward
    of $\Oc_{[\TT/\TT]}$ along the open immersion $[\TT/\TT]\hookrightarrow[X_\Sigma / \TT]$.
\end{example}
\begin{war}
    Given a symmetric monoidal functor $F:\Cc\rightarrow\Dc$ and $A\in\CAlg(\Cc)$, it induces a functor $F_A:\Mod_A(\Cc)\rightarrow\Mod_{F(A)}(\Dc)$. If $\Cc$ and $\Dc$ have geometric realizations and  tensor products in $\Cc$ and $\Dc$ commutes with geometric realizations, then both $\Mod_A(\Cc)$ and $\Mod_{F(A)}(\Dc)$ have symmetric monoidal structures given by relative tensor products. However, the functor $F_A$ lifts to a symmetric monoidal functor only when $F$ commutes with geometric realizations. The lift is functorial in the sense of \HA{Theorem}{4.8.5.16} (see below). The example to keep in mind is the following:
    \[\Set\rightarrow\Spc\rightarrow\Sp\]
     is a sequence of symmetric monoidal functors; The latter  preserves geometric realization while the former doesn't. For instance, the relative tensor product $X \times_\ZZ Y$ is in general not the same computed in $\Spc$ as done in $\Set$. When $X$ and $Y$ are both a point, in $\Set$ the outcome is still a point while in $\Spc$ one gets $B\ZZ$.
\end{war}

\begin{rem}[An antidote to the warning] As explained by above warning, for a given monoid $S\in\CMon(\Set)$, we don't have a symmetric monoidal structure on the inclusion functor $\Mod_S(\Set)\rightarrow\Mod_S(\Spc)$. One can, however, define a symmetric monoidal category sitting in both of them: take $\Mod_S(\Spc)^\free\subset\Mod_S(\Spc)$ to be the full subcategory spanned by coproducts of $S$. This
category inherits a symmetric monoidal structure and can be identified, symmetric monoidally, with the full subcategory spanned by coproducts of $S$ in $\Mod_S(\Set)$. To be very rigorous with the construction that follows, one should construct symmetric monoidal functor directly into $\Mod_S(\Spc)$, but we will construct functors into $\Mod_S(\Set)$ and observe that they lift to $\Mod_S(\Spc)$. We will also use the following fact obtained from \HA{Theorem}{4.8.5.16}.
\end{rem}

\begin{prop}
    Let $\Cc$ and $\Dc$ be symmetric monoidal categories  admitting all geometric realizations. Let $F:\Cc\rightarrow\Dc$ be a symmetric monoidal functor. Assume that:
    \begin{enumerate}
        \item Tensor products in $\Cc$ and $\Dc$ commute with geometric realizations.
        \item The functor $F$ commutes with geometric realizations.
    \end{enumerate}
    Then there is a diagram
\[\begin{tikzcd}
	{\CAlg(\Cc)} && {\SMCat}
	\arrow[""{name=0, anchor=center, inner sep=0}, "{\Mod_{(-)}(\Cc)}", curve={height=-12pt}, from=1-1, to=1-3]
	\arrow[""{name=1, anchor=center, inner sep=0}, "{\Mod_{F(-)}(\Dc)}"', curve={height=12pt}, from=1-1, to=1-3]
	\arrow[shorten <=3pt, shorten >=3pt, Rightarrow, from=0, to=1]
\end{tikzcd}.\]
    When evaluated at $A\rightarrow B\in\CAlg(\Cc)$, the diagram reads
\[\begin{tikzcd}
	{\Mod_A(\Cc)} & {\Mod_B(\Cc)} \\
	{\Mod_{F(A)}(\Dc)} & {\Mod_{F(B)}(\Dc)}
	\arrow[from=1-1, to=1-2]
	\arrow[from=2-1, to=2-2]
	\arrow[from=1-1, to=2-1]
	\arrow[from=1-2, to=2-2]
\end{tikzcd}.\]
\end{prop}
\begin{proof}
    See \cref{Taking module category}.
\end{proof}
The linearization functor $\SS[-]:\Spc\rightarrow\Sp$ is symmetric monoidal and preserves geometric realizations. So it induces, functorially, symmetric monoidal functors on module categories. This implies that there is a natural transformation from the cosimplicial diagram that presents $\Mod_{S_\sigma}(\Spc)^M$ to the cosimplicial diagram that presents $\QCoh([X_\sigma/\TT])$. We write
$$\Oc[-]:\Mod_{S_\sigma}(\Spc)^M\rightarrow\QCoh([X_\sigma/\TT])$$
for the symmetric monoidal functor one obtains after taking limit along $\Delta$. Note that both sides of above are indexed over $\sigma\in \Sigma^\op$, and for the same reason, $\Oc[-]$ assembles into a natural transformation of diagrams. In the next subsection we will use this natural transformation to produce a comparison functor from combinatorial models. 
    

\subsection{Combinatorial v.s. quasi-coherent}
The goal of this subsection is to provide the following construction.
\begin{prop}\label{theorem of combinatorics compares to quasicoherent}
    There exists a symmetric monoidal equivalence of categories
    $$\Phi_\sigma:\Fun(\Theta(\sigma)^\op,\Sp)\overset{\simeq}{\longrightarrow}\QCoh([X_\sigma/\TT])$$
    where the left-hand side has the Day convolution tensor product and the right-hand side has the standard tensor product of quasi-coherent sheaves. Moreover, these equivalences are functorial in $\sigma\in \Sigma^\op$ in that they assemble into a natural transformation of diagrams in $\SMCat$ indexed by $\Sigma^\op$. Hence taking limit produces
    $$\lim_{\Sigma^\op}\Fun(\Theta(\sigma)^\op,\Sp)\overset{\simeq}{\longrightarrow}\lim_{\Sigma^\op}\QCoh([X_\sigma/\TT])\simeq\QCoh([X_\Sigma/\TT]).$$
\end{prop}
\begin{rem}[Compatibility with the torus]
\label{global combinatorial-coherent functor is compatible with torus action}
We will establish along the way an equivalence
$$\Phi_M:\Fun(M,\Sp)\simeq\QCoh(B\TT)$$
and will also provide compatibility of $\Phi_M$ with above equivalence $\Phi_\sigma$ (see \cref{construction with torus is compatible with toric quotient}), i.e., the following diagram commutes
\[\begin{tikzcd}
	{\lim_{\Sigma^\op}\Fun(\Theta(\sigma)^\op,\Sp)} & &{\QCoh([X_\Sigma/\TT])} \\
	{\Fun(M,\Sp)} & &{\QCoh(B\TT)}
	\arrow["{\lim_{\Sigma^\op}\Phi_\sigma}", from=1-1, to=1-3]
	\arrow["{\Phi_M}", from=2-1, to=2-3]
	\arrow["\lim_{\Sigma^\op}(p_\sigma)_!",from=2-1, to=1-1]
	\arrow["{\lim_{\Sigma^\op}\pi_\sigma^*}", from=2-3, to=1-3]
\end{tikzcd}.\]
\end{rem}
\begin{rem}[The geometry of filtrations]
    Take the pair $N=\ZZ$ and $\Sigma=\{0,\RR_{\geq0}\}$. The theorem above reads 
    $$\Fun(\mathbb{Z}_\leq,\Sp)\simeq\QCoh([\AA^1/\GG_m]),$$
    which is \cite[Theorem 1.1]{Moulinos}. The proof presented in this subsection actually follows closely the approach in \cite{Moulinos}.
\end{rem}
We begin by constructing the functor $\Phi_\sigma$, then explain its naturality along $\sigma\in \Sigma^\op$. 
\begin{construction}(Construction of the functor in the unstable case)
\label{construction of unstable model of the comb-qcoh comparision map}
Fix a cone $\sigma$ in a lattice $N$, we define a functor
$$\phi_\sigma:\Theta(\sigma)\rightarrow\Mod_{S_\sigma}(\Set)^M$$
as follows: for $V\in\Theta(\sigma)$, recall that $V$ is an integral translation of $\sigma^\svee$.
We define $\phi_\sigma(V)$ to be
\[(V\cap M,f|_{V\cap M\times_{S_\sigma}S_\sigma\times M})\in \Mod_{S_\sigma}(\Set)^M\]
where $V\cap M$ is a subobject of  $M\in\Mod_{S_\sigma}(\Set)$ and it inherits the structure map $f$ from \cref{M as an object in the unstable quasicoherent category} since $f$ preserves $V\cap M\times_{S_\sigma}S_\sigma\times M$.\par
Then we move on to  morphisms. Given an inclusion $i:V\subset W\in \Theta(\sigma)$, we define $\phi_\sigma(i)$ to be the inclusion map 
\[\phi(i):V\cap M\rightarrow W\cap M\in \Mod_{S_\sigma}(\Set).\] 
It remains to check that $\phi(i)$ is compatible with the structure maps of $\phi(V)$ and $\phi(W)$. This follows directly from that both structure maps are inherited from $(M,f)$. So we know that $\phi(i)$ lifts to a map in $\Mod_{S_\sigma}(\Set)^M$. The symmetric monoidal structure on the functor can be supplied and checked directly as it is a functor between $1$-categories.
The construction lands in $\Mod_{S_\sigma}(\Spc)^\free$ in each degree and hence lifts to a symmetric monoidal functor to $\Mod_{S_\sigma}(\Spc)^M$. To conclude, we have obtained a symmetric monoidal functor
$$\phi_\sigma:\Theta(\sigma)\longrightarrow\Mod_{S_\sigma}(\Spc)^M.$$
\end{construction}
\begin{rem}(Naturality along $\sigma\in\Sigma^\op$)
\label{naturality along sigma in Sigma}
The functors $\phi_\sigma$ as above assembles into a natural transformation between diagrams in $\SMCat$ indexed by $\Sigma^\op$:
$$\Theta(-)\rightarrow\Mod_{S_{-}}(\Spc)^M.$$
Since we are working within $1$-categories, the naturality could be inspected directly.
\end{rem}
\begin{defn}
\label{definition of the combinatorial-quasicoherent comparison functor}
    We define $\Phi_\sigma$ to be the left Kan extension of $\Oc[\phi_\sigma]$ along the stable Yoneda embedding:
    $$\Phi_\sigma\defeq \Lan_h(\Oc[\phi_\sigma]):\Fun(\Theta(\sigma)^\op,\Sp)\longrightarrow\QCoh([X_\sigma/\TT]),$$
    where we have used the linearization functor
    $$\Oc[-]:\Mod_{S_\sigma}(\Spc)^M\rightarrow\QCoh([X_\sigma/\TT])$$
    from the last paragraph of \cref{section on quasi-coherent sheaves}. Note that it is symmetric monoidal with the Day convolution product on the domain. From the discussion in \cref{naturality along sigma in Sigma} and functoriality of the Day convolution (see \cref{reminders on Day convolution}) we learn that $\Phi_\sigma$ is a natural transformation 
\[\begin{tikzcd}
	&& {} \\
	{\Sigma^\op} &&& {\SMCat}
	\arrow[""{name=0, anchor=center, inner sep=0}, "{\Fun(\Theta(\sigma)^\op,\Sp)}", curve={height=-12pt}, from=2-1, to=2-4]
	\arrow[""{name=1, anchor=center, inner sep=0}, "{\QCoh([X_\sigma/\TT])}"', curve={height=12pt}, from=2-1, to=2-4]
	\arrow["{\Phi_\sigma}"', shorten <=3pt, shorten >=3pt, Rightarrow, from=0, to=1]
\end{tikzcd}\]
    
    between diagrams in $\SMCat$ indexed by $\Sigma$.
\end{defn}
\begin{example}[Equivariant line bundles on the affine line]
Take the pair $N=\ZZ$ and $\Sigma=\{0,\RR_{\geq0}\}$. The construction above produces a family of line bundles from the  symmetric monoidal functor
$$\Phi_{\RR_{\geq0}}:\Fun(\ZZ_\leq^\op;\Sp)\longrightarrow\QCoh([\AA^1/\GG_m]).$$
After base changing to $\ZZ$, it recovers the universal line bundles $\phi(n)=\Oc(n)$, universal sections $\cdot x:\Oc(n)\rightarrow\Oc(n+1)$, and isomorphisms $\Oc(m)\otimes\Oc(n)\rightarrow\Oc(mn)$. In general, it is possible to globalize this construction and construct torus-equivariant line bundles on toric schemes.
\end{example}

Now we move on to proving the main theorem of this section: that each $\Phi_\sigma$ is an equivalence of categories. First we do some preparations.
\begin{var}[Compare with \cref{construction of unstable model of the comb-qcoh comparision map}]
\label{construction of comparison from Fun(M) to qcoh of BT}
    We can define a symmetric monoidal functor $$\phi_M:M\rightarrow\Mod_{*}(\Set)^M$$ (where the point $*$ is the initial monoid) as follows. On objects, $m\in M$ is taken to the pair $(*,f_m)$. Here $*\in\Set$ is the underlying object and $f_m:*\times M\rightarrow *\times M$ is the isomorphism given by translation by $m$. Again one checks this satisfies the cocycle condition as in \cref{M as an object in the unstable quasicoherent category} so $(*,f_m)$ defines an object in $\Mod_*(\Set)^M$. This assignment lifts to a symmetric monoidal functor by direct inspection. Hence we get a symmetric monoidal functor 
    $$\Phi_M\defeq \Lan_h\Oc[\phi_M]:\Fun(M,\Sp)\rightarrow\QCoh(B\TT).$$
\end{var}
\begin{rem}
\label{construction with torus is compatible with toric quotient}
By the very explicit construction, the equivalence $\Phi_M$ above 
is compatible with \cref{definition of the combinatorial-quasicoherent comparison functor}: there is a symmetric monoidal functor $p_\sigma:M\rightarrow\Theta(\sigma)$ that sends $m$ to $m+\sigma^{\svee}$ (see \cref{alternative way to see symmetric monoidal structure on Theta}) making the diagram
\[\begin{tikzcd}
	{\Fun(\Theta(\sigma)^\op,\Sp)} & {\QCoh([X_\sigma/\TT])} \\
	{\Fun(M,\Sp)} & {\QCoh(B\TT)}
	\arrow["{\Phi_\sigma}", from=1-1, to=1-2]
	\arrow["{\Phi_M}", from=2-1, to=2-2]
	\arrow["(p_\sigma)_!",from=2-1, to=1-1]
	\arrow["{\pi_\sigma^*}", from=2-2, to=1-2]
\end{tikzcd}\]
commute, where $(p_\sigma)_!$ stands for left Kan extension of presheaves along $p_\sigma$ and $\pi_{\sigma}^*$ stands for $*$-pullback of quasi-coherent sheaves along $\pi_\sigma:[X_\sigma/\TT]\rightarrow B\TT$. The commutativity of the square comes from 1-categorical inspection before linearization. Moreover, the maps above are natural in $\sigma\in\Sigma^\op$ that one can interpret it as a square of natural transformations of diagrams in $\SMCat$ indexed by $\sigma\in\Sigma^\op$.
    
\end{rem}
We will follow the approach taken in \cite[\href{https://arxiv.org/pdf/1907.13562.pdf\#equation.4.1}{Theorem 4.1}]{Moulinos} to prove the following:
\begin{thm} 
\label{theorem of quasicoherent sheaves of torus compare to graded spectra}
There is an equivalence of symmetric monoidal categories
$$\Phi_M:\Fun(M,\Sp)\simeq\QCoh(B\TT),$$
where the left-hand side comes with the Day convolution tensor product and the right-hand side comes with the standard tensor product of quasi-coherent sheaves.    
\end{thm}
\begin{proof}
    We interpret $\Phi_M$ as an augmentation of the cosimplicial diagram presenting $\QCoh(B\TT)$:
\[\begin{tikzcd}
	\cdots & {\QCoh(*\times \TT \times\TT)} & {\QCoh(* \times \TT)} & {\QCoh(*)} & {\Fun(M,\Sp)}
	\arrow[shift left=2, from=1-3, to=1-2]
	\arrow[shift right=2, from=1-3, to=1-2]
	\arrow[shift right, from=1-4, to=1-3]
	\arrow[shift left, from=1-4, to=1-3]
	\arrow[shift right=3, from=1-2, to=1-1]
	\arrow[shift left=3, from=1-2, to=1-1]
	\arrow[shift left, from=1-2, to=1-1]
	\arrow[shift right, from=1-2, to=1-1]
	\arrow[from=1-3, to=1-2]
	\arrow[from=1-5, to=1-4]
\end{tikzcd}.\]
    Then the theorem follows from a direct application of \HA{Corollary}{4.7.5.3} in its comonadic form (as used in the proof of \SAG{Theorem}{5.6.6.1}). So we want to check the following: 
    \begin{enumerate}
        \item The functor $d^0:\Fun(M,Sp)\rightarrow \QCoh(*)=\Sp$ is comonadic.
        \item The Beck-Chevalley condition holds: for each $\alpha:[m]\rightarrow[n]$ in $\Delta_+$, the diagram
\[\begin{tikzcd}
	{\Cc^m} & {\Cc^{m+1}} \\
	{\Cc^{n}} & {\Cc^{n+1}}
	\arrow["\alpha"', from=1-1, to=2-1]
	\arrow["{d^0}", from=1-1, to=1-2]
	\arrow["{\alpha+1}", from=1-2, to=2-2]
	\arrow["{d^0}", from=2-1, to=2-2]
\end{tikzcd}\]
        is right adjointable (for horizontal maps).
    \end{enumerate}
    We first show $d^0:\Fun(M,\Sp)\rightarrow\Sp$ is comonadic.
    By construction, $d^0$ takes an $M$-family of spectra $\{X_m\}$ to the coproduct $\oplus X_m$.
    The crucial observation is that each $X_m$ is a retract of $\oplus X_m$.
    If $\oplus X_m\simeq0$, then each of $X_m$ is a retract of $0$,
    and hence we know that the family $\{X_m\}$ is $0$. Therefore, $d^0$ is conservative.
    It remains to show that $d^0$ preserves limits of cosimplicial diagrams in $\Fun(M,\Sp)$ with split images in $\Sp$.

    A cosimplicial diagram $X^\bullet$ in $\Fun(M,\Sp)$ is just an $M$-family of cosimplicial diagrams $\{X_m^\bullet\}$ in $\Sp$. Under this identification,  $d^0(X^\bullet) = \oplus X_m^\bullet$. Denote by $X^{-\infty} = \{X_m^{-\infty}\}$ a limit of $X^\bullet$.
    Note that each $X_m^\bullet$ is a retract of the split cosimplicial object $\oplus X_m^\bullet$,
    which itself must be split by \HA{Corollary}{4.7.2.13}.
    Therefore each of the augmented cosimplicial object $X_m^{-\infty} \rightarrow X^\bullet_m$ is split.
    It follows that the coproduct \[
    d^0(X^{-\infty}) \simeq \oplus_m X^{-\infty}_m \rightarrow \oplus_m X^\bullet_m \simeq d^0(X^\bullet_m)
    \] is also split, thus a limit diagram, as desired.

    Now we move on to checking the adjointability.
    For $\alpha:[m]\rightarrow[n]$, if $m\neq -1$, one can look at the corresponding groupoid object in $\Stk$:
\[\begin{tikzcd}
	{*\times\TT^{\times n+1}} & {*\times\TT^{\times m+1}} & {*\times\TT} \\
	{*\times\TT^{\times n}} & {*\times\TT^{\times m}} & {*}
	\arrow["\alpha", from=2-1, to=2-2]
	\arrow["{d^0}", from=1-2, to=2-2]
	\arrow["{\alpha+1}", from=1-1, to=1-2]
	\arrow["{d^0}", from=1-1, to=2-1]
	\arrow["{d^0}", from=1-3, to=2-3]
	\arrow["{\{0\}}", from=2-2, to=2-3]
	\arrow["{\{0,1\}}", from=1-2, to=1-3]
\end{tikzcd}.\]
    By the Segal condition \HTT{Proposition}{6.1.2.6}, both the right square and the outer rectangle are pullback squares, so the left square is also a pullback in $\Stk$. So we conclude that after applying $\QCoh(-)$ the left square is right adjointable by \SAG{Lemma}{D.3.5.6}. For $\alpha:[-1]\rightarrow[n]$, we first check that
\[\begin{tikzcd}
	{\Fun(M,\Sp)} & \Sp \\
	\Sp & {\QCoh(\TT)}
	\arrow["{d^0}", from=1-1, to=1-2]
	\arrow["{d^0}", from=2-1, to=2-2]
	\arrow["\alpha=d^0"', from=1-1, to=2-1]
	\arrow["{\alpha+1=d^1}"', from=1-2, to=2-2]
\end{tikzcd}\]
is right adjointable. We will use the following notational convention. Put $p:M\rightarrow*$ to be the projection of set $M$ to a point, and we write $p_!\dashv p^*$ for the adjunction between left Kan extension and restriction of presheaves. Put $\pi:\TT\rightarrow*$ to be the projection of stack $\TT$ to a point, and we write $\pi^*\dashv\pi_*$ for the adjunction between $*$-pullback and $*$-pushforward of quasi-coherent sheaves.
Under this notation, the diagram above reads:
\[\begin{tikzcd}
	{\Fun(M,\Sp)} & \Sp \\
	\Sp & {\QCoh(\TT)}
	\arrow["{p_!}", from=1-1, to=1-2]
	\arrow["{\pi^*}", from=2-1, to=2-2]
	\arrow["{p_!}"', from=1-1, to=2-1]
	\arrow["{\pi^*}"', from=1-2, to=2-2]
\end{tikzcd},\]
and the $2$-cell filling the square comes from the construction in \cref{construction of comparison from Fun(M) to qcoh of BT}. Be careful that the $2$-cell is not the trivial one (and the trivial one won't be right adjointable). We need to show that
$$p_!p^*\rightarrow\pi_*\pi^*p_!p^*\rightarrow\pi_*\pi^*p_!p^*\rightarrow \pi_*\pi^*$$
is an equivalence of functors, where the maps involed are given, in turn, by the unit for $\pi^*\dashv\pi_*$, the homotopy of the $2$-cell filling the diagram $\pi^*p_!\simeq\pi^*p_!$ and the counit for $p_!\dashv p^*$. Note that since both $p_!p^*$ and $\pi_*\pi^*$ are colimit-preserving, it suffices to check on $\SS\in\Sp$. Unwinding the definition, the map reads
\[\begin{tikzcd}
	{p_!p^*\SS} & {\pi_*\pi^*p_!p^*\SS} & {\pi_*\pi^*p_!p^*\SS} & {\pi_*\pi^*\SS} \\
	{\bigoplus_M\SS} & {\bigoplus_M\SS[M]} & {\bigoplus_M\SS[M]} & {\SS[M]}
	\arrow[from=1-1, to=1-2]
	\arrow["{=}"{marking, allow upside down}, draw=none, from=1-1, to=2-1]
	\arrow[from=1-2, to=1-3]
	\arrow["{=}"{marking, allow upside down}, draw=none, from=1-2, to=2-2]
	\arrow[from=1-3, to=1-4]
	\arrow["{=}"{marking, allow upside down}, draw=none, from=1-3, to=2-3]
	\arrow["{=}"{marking, allow upside down}, draw=none, from=1-4, to=2-4]
	\arrow["{\oplus\epsilon}", from=2-1, to=2-2]
	\arrow["{\oplus m\cdot}", from=2-2, to=2-3]
	\arrow["{\mathrm{sum}}", from=2-3, to=2-4]
\end{tikzcd}.\]
The first map is the coproduct of unit maps $\SS\rightarrow\SS[M]$ for the algebra $\SS[M]$. The second map is the coproduct of the maps $\cdot m:\SS[M]\rightarrow\SS[M]$ on each direct summand $m\in M$. The third map is induced by identity map $\id:\SS[M]\rightarrow\SS[M]$ on each summand, i.e., forming summation. The composition, which is $\cdot m: \SS\rightarrow\SS[M]$ on each summand, is readily an equivalence of spectra.\par 
Recall that we are showing diagrams involving $[-1]$ are right adjointable, and we have only checked one of them. Now for a general map $\alpha:[-1]\rightarrow[n]$, observe that there is a commutative diagram (note the different orientation of the diagram)
\[\begin{tikzcd}
	{[-1]} & {[0]} & {[n]} \\
	{[0]} & {[1]} & {[n+1]}
	\arrow["\alpha'"', from=1-1, to=1-2]
	\arrow["\beta"', from=1-2, to=1-3]
	\arrow["{d^0}", from=1-1, to=2-1]
	\arrow["{\alpha'+1=d^1}", from=2-1, to=2-2]
	\arrow["{\beta+1}", from=2-2, to=2-3]
	\arrow["{d^0}", from=1-2, to=2-2]
	\arrow["{d^0}"', from=1-3, to=2-3]
\end{tikzcd}\]
in $\Delta_+$ where $\beta\circ\alpha'=\alpha$. This is taken to a diagram of categories where both of the squares are right adjointable (now along the vertical edges). We hence conclude that the outer rectangle is also right adjointable, as desired.
\end{proof}
We now state a technical claim about adjointability of diagrams that we will use in proving \cref{theorem of combinatorics compares to quasicoherent}. The proof will be offered later.
\begin{lem} 
\label{One can take adjoint and end with comparison between monadic adjunction}
 The diagram in \cref{construction with torus is compatible with toric quotient} is right adjointable for taking right adjoints of $(p_\sigma)_!$ and $\pi_\sigma^*$. In other words, the diagram
\[\begin{tikzcd}
	{\Fun(\Theta(\sigma)^\op,\Sp)} & {\QCoh([X_\sigma/\TT])} \\
	{\Fun(M,\Sp)} & {\QCoh(B\TT)}
	\arrow["{\Phi_\sigma}", from=1-1, to=1-2]
	\arrow["{\Phi_M}", from=2-1, to=2-2]
	\arrow["{(p_\sigma)^*}"', from=1-1, to=2-1]
	\arrow["{\pi_{\sigma*}}"', from=1-2, to=2-2]
\end{tikzcd}\]
commutes, with the homotopy specified by
$$\Phi_Mp_\sigma^*\rightarrow\pi_{\sigma*}\pi_\sigma^*\Phi_Mp_\sigma^*\rightarrow\pi_{\sigma*}\Phi_\sigma p_{\sigma!}p_\sigma^*\rightarrow\pi_{\sigma*}\Phi_\sigma$$
where the maps involved are given in turn by the unit for $\pi_\sigma^*\dashv\pi_{\sigma*}$, the homotopy of the $2$-cell in the diagram $\pi_\sigma^*\Phi_M\simeq\Phi_\sigma p_{\sigma!}$ and the counit for $p_{\sigma!}\dashv p_\sigma^*$. 
\end{lem}
We are now ready to prove the main theorem of the section.
\begin{proof}[Proof of \cref{theorem of combinatorics compares to quasicoherent}]
    Naturality of the functors has been  explained in \cref{construction with torus is compatible with toric quotient}.  What's left to check is that for each $\sigma$, $\Phi_\sigma$ is an equivalence of categories. Given \cref{One can take adjoint and end with comparison between monadic adjunction} we are in the situation of comparing monadic adjunctions \HA{Proposition}{4.7.3.16}: each of the category is monadic over another category. We claim that the assumptions in \HA{Proposition}{4.7.3.16} are readily true in our case: (1) is true as our diagrams are obtained by taking right adjoints of  right adjointable diagrams; (2) and (3) follow from that both $p_\sigma^*$ and $\pi_{\sigma*}$ are colimit-preserving functors; (4) is true because $\pi$ is affine so $*$-pushforward along $\pi$ is conservative. Note that the functor $p^*$ (which is restriction of presheaves) is also conservative since $p$ is an essentially surjective functor; and (5) requires essentially to check if the diagram is left adjointable, which again follows from the fact that the diagram itself comes from taking right adjoints of a right adjointable diagram, see \HTT{Remark}{7.3.1.3}.
\end{proof}
\begin{proof}[Proof of \cref{One can take adjoint and end with comparison between monadic adjunction}]
    We look at the map between augmented action diagrams which presents the map $[X_\sigma/\TT]\rightarrow B\TT$
\[\begin{tikzcd}
	\cdots & {X_\sigma\times \TT \times\TT} & {X_\sigma \times \TT} & {X_\sigma} & {[X_\sigma/\TT]} \\
	\cdots & { \TT \times\TT} & { \TT} & {*} & B\TT
	\arrow[shift right=2, from=1-2, to=1-3]
	\arrow[shift left=2, from=1-2, to=1-3]
	\arrow[from=1-2, to=1-3]
	\arrow[shift left, from=1-3, to=1-4]
	\arrow[shift right, from=1-3, to=1-4]
	\arrow[shift left=3, from=1-1, to=1-2]
	\arrow[shift right=3, from=1-1, to=1-2]
	\arrow[shift right, from=1-1, to=1-2]
	\arrow[shift left, from=1-1, to=1-2]
	\arrow[shift left, from=2-1, to=2-2]
	\arrow[shift right, from=2-1, to=2-2]
	\arrow[shift left=3, from=2-1, to=2-2]
	\arrow[shift right=3, from=2-1, to=2-2]
	\arrow[from=2-2, to=2-3]
	\arrow[shift left=2, from=2-2, to=2-3]
	\arrow[shift right=2, from=2-2, to=2-3]
	\arrow[shift left, from=2-3, to=2-4]
	\arrow[shift right, from=2-3, to=2-4]
	\arrow[from=2-4, to=2-5]
	\arrow[from=1-4, to=1-5]
	\arrow[from=1-2, to=2-2]
	\arrow[from=1-3, to=2-3]
	\arrow[from=1-4, to=2-4]
	\arrow[from=1-5, to=2-5]
\end{tikzcd}.\]
For each $\alpha:[m]\rightarrow[n]\in\Delta$, we have the diagram
\[\begin{tikzcd}
	{X_\sigma\times\TT^{\times n}} & {X_\sigma\times\TT^{\times m}} & {[X_\sigma/\TT]} \\
	{\TT^{\times n}} & {\TT^{\times m}} & B\TT
	\arrow["\alpha", from=1-1, to=1-2]
	\arrow["\alpha", from=2-1, to=2-2]
	\arrow[from=1-2, to=1-3]   
	\arrow[from=1-1, to=2-1]
	\arrow[from=1-2, to=2-2]
	\arrow[from=2-2, to=2-3]
	\arrow[from=1-3, to=2-3]
\end{tikzcd},\]
where both the outer rectangle and the inner right square are pullbacks, so the left square is also a pullback (see \cref{Grouplike action gives groupoid object}). Hence from \SAG{Lemma}{D.3.5.6}, we learn that after taking $\QCoh$, the left square becomes
\[\begin{tikzcd}
	{\QCoh(X_\sigma\times\TT^{\times n})} & {\QCoh(X_\sigma\times\TT^{\times m})} \\
	{\QCoh(\TT^{\times n})} & {\QCoh(\TT^{\times m})}
	\arrow["\alpha"', from=1-2, to=1-1]
	\arrow["\alpha"', from=2-2, to=2-1]
	\arrow[from=2-1, to=1-1]
	\arrow[from=2-2, to=1-2]
\end{tikzcd}\]
which is right adjointable (for vertical maps). By \HA{Corollary}{4.7.4.18} this implies that $\QCoh(-)$ of the action diagram, viewed as $[n]\mapsto[\QCoh(\TT^{\times n})\rightarrow\QCoh(X_\sigma\times\TT^{\times n})]$, lifts to a simplicial object in $\Fun^\RAd(\Delta^1,\Cat)$, and $\QCoh(-)$ of the augmented action diagram is a limit diagram in $\Fun^\RAd(\Delta^1,\Cat)$. Now one can similarly view the diagram 
\[\begin{tikzcd}
	{\Fun(\Theta(\sigma)^\op,\Sp)} & {\QCoh([X_\sigma/\TT])} \\
	{\Fun(M,\Sp)} & {\QCoh(B\TT)}
	\arrow["{\Phi_\sigma}", from=1-1, to=1-2]
	\arrow["{\Phi_M}", from=2-1, to=2-2]
	\arrow["(p_\sigma)_!",from=2-1, to=1-1]
	\arrow["{\pi_\sigma^*}", from=2-2, to=1-2]
\end{tikzcd}\]
as an augmentation to the simplicial object $[n]\mapsto[\QCoh(\TT^{\times n})\rightarrow\QCoh(X_\sigma\times\TT^{\times n})]$ in $\Fun(\Delta^1,\Cat)$. Determining its right adjointability reduces to asking if this augmentation lifts to $\Fun^\RAd(\Delta^1,\Cat)$. The only thing left to check is right adjointability of the diagram (for taking right adjoints of the vertical arrows)
\[\begin{tikzcd}
	{\Fun(\Theta(\sigma)^\op,\Sp)} & {\QCoh(X_\sigma)} \\
	{\Fun(M,\Sp)} & {\QCoh(*)}
	\arrow["{\Phi_\sigma}", from=1-1, to=1-2]
	\arrow["{\Phi_M}", from=2-1, to=2-2]
	\arrow["{p_{\sigma!}}",from=2-1, to=1-1]
	\arrow["{\pi^*}", from=2-2, to=1-2]
\end{tikzcd}.\]
This is readily true once one unwinds the definition as in the proof of \cref{theorem of quasicoherent sheaves of torus compare to graded spectra}.


\end{proof}

\section{Constructible sheaves}
In the seminal book of \cite{HA}, Jacob Lurie sets up a general theory of constructible sheaves of spaces on a stratified topological space. In particular, this theory has been worked out with no finiteness assumptions on the sheaves involved. For a compactly generated presentable coefficient category $\Cc$, the theory of constructible sheaves on a stratified topological space valued in $\Cc$ could be worked out in a similar way (for example, as in \cite{lurie2018associativealgebrasbrokenlines}). We will follow this convention and setup various functors involved in the coherent-constructible correspondence. This approach makes several constructions easier. First of all, the yoga of six-functor provides a neat way to write down convolution products defined on the category of sheaves of spectra on a (locally compact Hausdorff) topological group and related functors. Secondly  recent advances in exodromy \cite{haine2024exodromyconicality,clausenJansen2023reductiveborelserre} make it gracefully simple to work with large categories of constructible sheaves.\par
The main goal of this section is to write down a symmetric monoidal functor from the combinatorial model to the category of sheaves on a real vector space. To do so, we first recall some generalities on convolution products for sheaves on real vector spaces. Then we move onto a digression on the lax symmetric monoidal structure on the relative homology functor. This is used in the next part to provide a combinatorial-constructible comparison functor along with its lax symmetric monoidal structure. After that we take a turn to recall some generalities on constructible sheaves and pin down a stratification following \cite{FLTZ}. As a consequence, we show that the comparison functor is fully faithful for a smooth fan and its image consists of sheaves constructible for the stratification we introduced. Finally we take a detour to prove a technical fact about descent along idempotent algebras in $\Shv(\MR;\Sp)$. Putting everything together, we conclude that for a smooth projective fan, the combinatorial-constructible comparison functor we constructed is fully faithful and symmetric monoidal. We leave the characterization of the image to the next section.
\subsection{Convolution product for sheaves on real vector spaces}
\begin{rem}[Hypercompleteness] One needs not to worry about hypercompleteness in our situation, as we will only deal with sheaves on finite dimensional manifolds. In particular, equivalence of sheaves can be detected stalk-wise.
\end{rem}
Take a finite dimensional real vector space $V\simeq\RR^{\oplus n}$.
It acquires the structure of a commutative algebra in $(\LCH,\times)$ via the addition of vectors
$$+:V\times V\rightarrow V.$$
This equips $\Shv(V;\Sp)$ with a binary operation $$*:\Shv(V;\Sp)\times\Shv(V;\Sp)\rightarrow\Shv(V;\Sp)$$ defined by
$$\Fc*\Gc\defeq+_!(\pr_1^*\Fc\otimes\pr_2^*\Gc).$$ This operation could be made coherently into a symmetric monoidal structure as follows.
\begin{construction}[Convolution product]\label{construction:convolution} Recall that the `six-functor formalism' on $\LCH$ is a lax symmetric monoidal functor
$$\Dc:\Corr(\LCH,\mathrm{all})\longrightarrow\Cat$$
and we have another symmetric monoidal functor (`Reg' for right leg)
$$\mathrm{Reg}:\LCH\rightarrow\Corr(\LCH,\mathrm{all})$$
which on objects acts as $X\mapsto X$ and on morphisms acts as
$$[X\overset{f}{\rightarrow} Y]\mapsto
\Biggr[\begin{tikzcd}
	& X \\
	X && Y
	\arrow["{\id_X}"', from=1-2, to=2-1]
	\arrow["f", from=1-2, to=2-3]
\end{tikzcd}\Biggr].
$$
We define the composition $$D_!(-)\defeq\Dc\circ\mathrm{Reg}:\LCH\rightarrow\Cat$$ which is again a lax symmetric monoidal functor.
This implies that for every commutative algebra $A\in\CAlg(\LCH)$, the category $D_!(A)=\Shv(A;\Sp)$ acquires a symmetric monoidal structure through the functoriality of $D_!$.
We name the resulting monoidal product \stress{convolution} and write it as $*$.
\end{construction}
\begin{prop}\label{proposition with direct computation of convolution}
    Let $V$ be a real vector space and $*$ be the convolution product operation on sheaves.
    \begin{enumerate}
        \item The convolution product $*$ is cocontinuous in each variable.
        \item Let $X,Y\subseteq V$ be polyhedral open subsets of a real vector space. We can compute very explicitly
        $$\underline\SS_X*\underline\SS_Y\simeq \underline{\SS}_{X+Y}[-\dim(V)]$$
        where
        $$X+Y\defeq\{x+y:x\in X,y\in Y\}$$ is the \stress{ Minkowski sum} of the subsets.
    \end{enumerate}
\end{prop}
\begin{proof}
Point 1 follows from the fact that $*$-pullback, $\otimes$ of sheaves and $!$-pushforward all preserve colimits. For point 2, we apply proper base change and learn that 
        $$\underline\SS_X*\underline\SS_Y\simeq +_{|_{X\times Y}!}\underline{\SS}_{X\times Y}$$
        where $+$ is restricted to a map $X\times Y\rightarrow\RR\times\RR\rightarrow \RR$.
        By the fact that $X$ and $Y$ are polyhedral open subsets, one can prove this map $+$ is a smooth $\RR^n$ bundle over its image $X+Y\subseteq\RR^n$.
        It follows that the !-pushforward of $\underline{\SS}_{X\times Y}$ along
        the addition map is locally constant.
        The computation reduces to the fact that for the projection $p:Z\times\RR^n\rightarrow Z$, one has \[p_!\underline\SS=\underline\SS[-n].\]
        Keep in mind that $X + Y$ is contractible.\qedhere
\end{proof}
\begin{rem}
As a side remark, polyhedral opens form a basis for the topology. In principle, one can compute the convolution of any two sheaves using the above facts.
\end{rem}
\subsection{Digression: multiplicative structures on Betti homology}
As we have seen earlier, the addition operation on the finite dimensional real vector space $M_\RR$ makes it into a commutative monoid in the 1-category $\LCH$. Thus the slice category $\LCH_{/M_\RR}$ acquires a symmetric monoidal structure which can be informally defined as follows: $$(X,f)\otimes(Y,g)\defeq(X\times Y,f+g)$$ (see \HA{Theorem}{2.2.2.4} for the general construction). We denote by $(\LCH_{/M_\RR},\otimes)$ this symmetric monoidal category. The structure of a commutative monoid on $M_\RR$  also has been used to provide a convolution product on the category of sheaves on $M_\RR$, and these two categories are indeed related. The goal of this digression is to explain the following construction.
\begin{prop}[Taking homology is symmetric monoidal]\label{construction:symmonshriek}
    There is a lax symmetric monoidal functor 
    $$\Gamma_{M_\RR}:(\LCH_{/M_\RR},\otimes)\longrightarrow(\Shv(M_\RR;\Sp),*)$$
    which on objects acts by
    $$(X,f)\longmapsto f_!f^!\omega_{M_\RR},$$
    where $\omega_{M_\RR}$ is the dualizing sheaf on $M_\RR$.
\end{prop}
\begin{rem}[A similar construction in the literature]
Let us immediately point out that, a very similar construction has been carried out (in the $\ell$-adic context) by Gaitsgory-Lurie in \cite[\href{https://www.math.ias.edu/~lurie/papers/tamagawa-abridged.pdf\#chapter.3}{Chapter 3}]{GL}. An elaboration (in the Betti context) of the ideas in that paper would produce a more general construction that easily provides the functor as above (for example, one could allow the base groups $G$ to vary). We have, however, decided to give an ad-hoc construction of the functor that we need in this note to simplify our exposition (also because the situation we are dealing with here is extremely simple). We will return to this construction elsewhere.    
\end{rem}
The construction is technical in contrast to the simple application we have in mind. The reader is advised to skip the rest of this section and come back later. Before we go into the construction, here is a rough plan. 
\begin{rem}[Preview of strategy]

We will define a symmetric monoidal category $\Shv_!$ which comes with a symmetric monoidal functor $$p:\Shv_!\rightarrow\LCH_{/M_\RR}.$$ We will then produce a lax symmetric monoidal functor as a section of $p$: $$s:\LCH_{/M_\RR}\rightarrow\Shv_!,$$ and another symmetric monoidal functor $$t:\Shv_!\rightarrow\Shv(M_\RR;\Sp),$$ so that the composition $$t\circ s:\LCH_{/M_\RR}\rightarrow\Shv(M_\RR;\Sp)$$ is what we want. 

\end{rem} 
\begin{rem}[A rough description of the players] We give an informal description of the categories and functors appearing in the previous remark. One can describe the category $\Shv_!$ as follows. An object in $\Shv_!$ is a pair $(X,f,\Fc)$ where $(X,f)$ is an object of $\LCH_{/M_\RR}$ and $\Fc\in\Shv(X;\Sp)$. A map $(h,\phi)$ from $(X,f,\Fc)$ to $(Y,g,\Gc)$ consists of a map $h:(X,f)\rightarrow (Y,g)$ in $\LCH_{/M_\RR}$ and a map $\phi:h_!\Fc\rightarrow\Gc$ in $\Shv(Y;\Sp)$. The symmetric monoidal structure is a mixture of tensor product in $\LCH_{/M_\RR}$ and exterior product of sheaves: $(X,f,\Fc)\otimes(Y,g,\Gc)=(X\times Y,f+g,\Fc\boxtimes\Gc)$. With these we can also roughly describe the functors. The functor $$p:\Shv_!\rightarrow\LCH_{/M_\RR}$$ is the forgetful functor taking $(X,f,\Fc)$ to $(X,f)$. The functor $$s:\LCH_{/M_\RR}\rightarrow\Shv_!$$ takes $(X,f)$ to $(X,f,f^!\omega_{M_\RR})\in\Shv_!$. The functor $$t:\Shv_!\rightarrow\Shv(M_\RR;\Sp)$$ takes $(X,f,\Fc)$ to $f_!\Fc\in\Shv(M_\RR;\Sp)$. This casual description suggests that $t\circ s$ supplies the construction we need. Note that we are not even mentioning what these functor does to maps or higher coherences, nor multiplicative structure. This is what makes the construction technical.
\end{rem}
We start by constructing $\Shv_!$.
\begin{notation}
The forgetful functor $\mathrm{forgetful}:\LCH_{/M_\RR}\longrightarrow\LCH$ is symmetric monoidal and we have a composition of functors
    $$\LCH_{/M_\RR}\overset{\mathrm{forgetful}}{\longrightarrow}\LCH\overset{D_!}{\longrightarrow}\Cat$$
    where the latter functor comes from \cref{construction:convolution}. We abuse notation and again write the composition as
    $$D_!:\LCH_{/M_\RR}\longrightarrow\Cat$$
    when there is no danger of confusion. Note that this composition is also a lax symmetric monoidal functor.
\end{notation}
The category $\Shv_!$ is just the unstraightening (i.e. Grothendieck construction) of the functor $D_!:\LCH_{/M_\RR}\rightarrow\Cat$, and the symmetric monoidal structure actually comes with unstraightening - using a symmetric monoidal version of the Grothendieck construction that we recall as follows.
\begin{thm}[Symmetric monoidal Grothendieck construction]
    (See \cite[A.2.1]{hinich2015rectification}\cite[Proposition 3.3.4.11]{GL} \cite[Theorem 2.1]{ramzi2022monoidal} for a history of the theorem.)   Let $(\Cc,\otimes)$ be a symmetric monoidal category. There is an equivalence of categories
    $$\cocart_{\Cc}^{\EE_\infty}\simeq\Fun^{\mathrm{lax}\otimes}(\Cc,\Cat)$$
    which is compatible with the straightening-unstraightening equivalence
    $$\cocart_{\Cc}\simeq\Fun(\Cc,\Cat).$$
\end{thm}
Let's immediately recall the definition of the objects appearing in the theorem. 

\begin{enumerate}
    \item For a category $\Cc$, the category $\cocart_\Cc$ is the category of \stress{coCartesian fibrations} over $\Cc$ with coCartesian edge preserving functors over $\Cc$ as morphisms. 
    \item If $(\Cc,\otimes)$ is a symmetric monoidal category with $\Cc^\otimes\rightarrow\einfty$\footnote{Note that $\einfty$ is just a fancy name for $\mathrm{Fin}_*$.} being the underlying operad, the category $\cocart_\Cc^{\einfty}$ is the category of \stress{$\einfty$-monoidal coCartesian fibrations} over $\Cc$ of \cite[Definition 1.11]{ramzi2022monoidal}. It is defined to be the full subcategory of $\cocart_{\Cc^\otimes}$ spanned by those coCartesian fibrations $\Dc^\otimes\rightarrow\Cc^\otimes$ such that the underlying $\Dc\rightarrow\Cc$ is a coCartesian fibration and that the $\einfty$-monoidal operations preserve coCartesian edge.
\end{enumerate}

\begin{defn}
    Applying the symmetric monoidal Grothendieck construction to the lax symmetric monoidal functor $D_!:\LCH_{/M_\RR}\rightarrow\Cat$ produces an $\einfty$-monoidal coCartesian fibration
    $$p^\otimes:\Shv_!^\otimes\longrightarrow\LCH_{/M_\RR}^\otimes.$$
    We write 
    $$p:\Shv_!\longrightarrow\LCH_{/M_\RR}$$
    for the underlying map making $\Shv_!$ a coCartesian fibration over $\LCH_{/M_\RR}$.
\end{defn}
In view of \HA{Remark}{2.1.2.14} and \cref{lemma:cocartesian}, the structure map $p^\otimes$ is a map of $\einfty$-monoidal category. In other words, it presents $p$ as a symmetric monoidal functor. This functor $p$ won't appear in the final construction, but we will introduce other players that revolve around $\Shv_!$ and $p$. We start with introducing the following diagram
\[\begin{tikzcd}
	{\LCH_{/M_\RR}} & {\LCH_{/M_\RR}} & \Cat
	\arrow[""{name=0, anchor=center, inner sep=0}, "\mathrm{id}", shift left, curve={height=-6pt}, from=1-1, to=1-2]
	\arrow["{D_!}", from=1-2, to=1-3]
	\arrow[""{name=1, anchor=center, inner sep=0}, "{\underline{M_\RR}}"', shift right, curve={height=6pt}, from=1-1, to=1-2]
	\arrow["h", shorten <=2pt, shorten >=2pt, Rightarrow, from=0, to=1]
\end{tikzcd},\]
where $\underline{M_\RR}$ is the constant functor at $(M_\RR,\id)\in\LCH_{/M_\RR}$ and $h$ is the natural transformation to the constant functor on the terminal object. Note that $h$ is actually a natural transformation between lax symmetric monoidal functors.
Now we apply the Grothendieck construction to $D_!(h):D_!\circ\id\rightarrow D_!\circ\underline{M_\RR}$ and get the following diagram
\[\begin{tikzcd}
	{\Shv_!} &&& {\LCH_{/M_\RR}\times\Shv(M_\RR;\Sp)} \\
	&& {\LCH_{/M_\RR}}
	\arrow["p"', from=1-1, to=2-3]
	\arrow["q",from=1-4, to=2-3]
	\arrow["\unstr({D_!(h)})", from=1-1, to=1-4]
\end{tikzcd}\]
underlying the diagram of operads supplied by
the symmetric monoidal Grothendieck construction
\[\begin{tikzcd}
	{\Shv_!^\otimes} &&& {(\LCH_{/M_\RR}\times\Shv(M_\RR;\Sp))^\otimes} \\
	&& {\LCH_{/M_\RR}^\otimes} \\
	&& \einfty
	\arrow["{p^\otimes}"', from=1-1, to=2-3]
	\arrow["{q^\otimes}", from=1-4, to=2-3]
	\arrow["{\unstr({D_!(h)})^\otimes}", from=1-1, to=1-4]
	\arrow["{\pi_1^\otimes}", from=2-3, to=3-3]
	\arrow["{\pi_2^\otimes}"', curve={height=12pt}, from=1-1, to=3-3]
	\arrow["{\pi_3^\otimes}", curve={height=-12pt}, from=1-4, to=3-3]
\end{tikzcd}.\]
In the diagram, $\pi_i^\otimes$ are the structure maps of the operads. Our first goal is to produce the right adjoint $r$ of $\unstr({D_!(h)})$ along with the lax symmetric monoidal structure on it.
\begin{prop}
    The functor $\unstr({D_!(h)}):\Shv_!\rightarrow\LCH_{/M_\RR}\times\Shv(M_\RR;\Sp)$ admits a right adjoint $r$. Moreover, $r$ admits a lax symmetric monoidal structure.
\end{prop}
\begin{proof}
To begin with, we want to show that $\unstr({D_!(h)})$ has a right adjoint functor $r$. We know the following facts about $\unstr({D_!(h)})$: that the restriction of $\unstr({D_!(h)})$ to each fiber over $\LCH_{/M_\RR}$ has a right adjoint and that $\unstr({D_!(h)})$ preserves coCartesian edges since it is unstraightened from a natural transformation. With these one can apply \HA{Proposition}{7.3.2.6} and learn that it has a right adjoint (even relative to $\LCH_{/M_\RR}$). By  construction, $r$ restricts to fiberwise right adjoints.\par
Now we explain the lax symmetric monoidal structure on $r$. 
From \cref{lemma:cocartesian} we learn that $\unstr({D_!(h)}^\otimes)$ is a map of $\einfty$-monoidal categories, i.e. $\unstr({D_!(h)})$ is a symmetric monoidal functor. Now one can invoke \HA{Corollary}{7.3.2.7} and learn that $r$ has a structure of lax symmetirc monoidal functor.
\end{proof}
We have achieved our first goal.
Our next player is the functor  $$\id\times\underline{\omega_\MR}:\LCH_{/M_\RR}\rightarrow\LCH_{/M_\RR}\times\Shv(M_\RR;\Sp).$$
As the name suggests, it is induced by $\id:\LCH_{/M_\RR}\rightarrow\LCH_{/M_\RR}$ and the constant functor
$\underline{\omega_{M_\RR}}:\LCH_{/M_\RR}\rightarrow\Shv(M_\RR;\Sp)$.
Recall that we have the \stress{dualizing sheaf}
$\omega_{M_\RR}$ defined by
$$\omega_{M_\RR}:=\pi^!\mathbb{1}_{\Shv(*;\Sp)}\in\Shv(M_\RR;\Sp),$$
where $\pi:M_\RR\rightarrow*$ is the map from $M_\RR$ to the terminal object $*$.
Let's make an observation on $\omega_{M_\RR}$.
\begin{prop}
    The dualizing sheaf $\omega_{M_\RR}$ acquires the structure of a commutative algebra for the convolution product.
\end{prop}
\begin{proof}
This follows from the fact that $\pi^!:\Shv(*;\Sp)\rightarrow\Shv(M_\RR;\Sp)$ has the structure of a lax symmetric monoidal functor where both side has the convolution symmetric monoidal structure.
In addition, the convolution product on $\Shv(*;\Sp)$ is the same as the point-wise tensor product on $\Shv(*;\Sp)\simeq\Sp$ that is usually used.
The lax symmetric monoidal structure on $\pi^!$ is given by the (strong) symmetric monoidal structure on its left adjoint $\pi_!$.
To be more precise: the map $\pi$ is actually a map of commutative monoids in $\LCH$.
Hence by construction of the convolution tensor product, $\pi$ induces a symmetric monoidal functor 
\[\pi_!:\Shv(M_\RR;\Sp)\longrightarrow\Shv(*;\Sp).\]
We again take advantage of \HA{Corollary}{7.3.2.7} and get a lax symmetric monoidal structure on its right adjoint 
$$\pi^!:\Shv(*;\Sp)\longrightarrow\Shv(M_\RR;\Sp).$$
In particular it takes $\mathbb{1}_{\Shv(*;\Sp)}$ to a commutative algebra, as desired.
\end{proof}

The commutative algebra structure on $\omega_{M_\RR}$ furnishes the constant functor 
$$\underline{\omega_{M_\RR}}:\LCH_{M_\RR}\longrightarrow\Shv(M_\RR;\Sp)$$
with a lax symmetric monoidal structure. From this discussion, one learns that
\begin{prop}
    The functor $\id\times\underline{\omega_\RR}$ admits a lax symmetric monoidal structure.
\end{prop}
\begin{proof}
    By previous discussion, it is a product of two lax symmetric monoidal functors, hence has a lax symmetric monoidal structure.
\end{proof}
We arrive at the following diagram
\[\begin{tikzcd}
	{\Shv_!} &&& {\LCH_{/M_\RR}\times\Shv(M_\RR;\Sp)} & {\Shv(M_\RR;\Sp)} \\
	&& {\LCH_{/M_\RR}}
	\arrow["p"', from=1-1, to=2-3]
	\arrow["q"', from=1-4, to=2-3]
	\arrow["{\unstr(D_!(h))}"', color={rgb,255:red,204;green,51;blue,51}, from=1-1, to=1-4]
	\arrow["{p_2}", color={rgb,255:red,204;green,51;blue,51}, from=1-4, to=1-5]
	\arrow["{\id\times\underline{\omega_{M_\RR}}}"', color={rgb,255:red,204;green,51;blue,51}, curve={height=12pt}, from=2-3, to=1-4]
	\arrow["r"', color={rgb,255:red,204;green,51;blue,51}, curve={height=18pt}, from=1-4, to=1-1]
\end{tikzcd}\]
where we are going to make use of the red-colored functors, which are lax symmetric monoidal. We conclude the construction by a composition of these four functors: according to the plan, we have constructed the following lax symmetric monoidal functors
$$s=r\circ(\id\times\underline{\omega_{M_\RR}}):\LCH_{/M_\RR}\rightarrow\Shv_!$$
and
$$t=p_2\circ\unstr(D_!(h)):\Shv_!\rightarrow\Shv(M_\RR;\Sp)$$
so that the composition
$$t\circ s:\LCH_{/M_\RR}\rightarrow\Shv(M_\RR;\Sp)$$
is what we aimed for.
\begin{defn}[Sheaf of relative homology]
\label{defintion of a functor taking a space ot its sheaf of homology}
We define the lax symmetric monoidal functor 
$$\Gamma_{M_\RR} = t\circ s:\LCH_{/M_\RR}\rightarrow\Shv(M_\RR;\Sp)$$
 as the output of the construction. And we call $\Gamma_{M_\RR}(X,f)$ the \stress{sheaf of homology of $X$ relative to $M_\RR$}. The naming follows from the fact that after further !-pushfoward to a point, the sheaf $\Gamma_{M_\RR}(X,f)$ is taken to the homology of $X$:
 \[\pi_!\Gamma_\MR(X,f)\simeq C_c^*(X,\omega_X).\]
 Note, however, the name might be misleading since the stalk of the sheaf $\Gamma_{\MR}(X,f)$ needs not to be the homology of the fiber.
\end{defn}
\begin{var}
For later purposes, we also by abuse of notation write the restriction of the functor $\Gamma_\MR$ to the full subcategory of polyhedral subsets as 
$$\Gamma_{M_\RR}:\Closed(M_\RR)\rightarrow\Shv(M_\RR;\Sp).$$
Moreover, the category $\Closed(M_\RR)$ carries a symmetric monoidal structure given by Minkowski sum that makes the inclusion functor
$$\Closed(\MR)\longrightarrow\LCH_{/\MR}$$
lax symmetric monoidal. We hence conclude that the functor
$$\Gamma_{M_\RR}:\Closed(M_\RR)\rightarrow\Shv(M_\RR;\Sp)$$
is also lax symmetric monoidal.
\end{var}
We end the section by recording the following elaboration of the argument in \HA{Proposition}{2.1.2.12}. See also \cite[\href{https://kerodon.net/tag/01UL}{01UL}]{kerodon}.
\begin{lem}\label{lemma:cocartesian}
    We have the following facts concerning coCartesian fibrations:
    \begin{enumerate}
        \item Consider the following commuting diagram of categories:
\[\begin{tikzcd}
	\Cc && \Dc \\
	& \Ec
	\arrow["{q\circ p}", from=1-1, to=2-2]
	\arrow["p", from=1-1, to=1-3]
	\arrow["q"', from=1-3, to=2-2]
\end{tikzcd}.\]
        If both $q$ and $p$ are coCartesian fibraitons, then so is $q\circ p$. Moreover, given an edge $f\in \Ec$ and a $q\circ p$-coCartesian lift $f'\in\Cc$ of $f$, there exists an edge $f''\in\Dc$ which is a $q$-coCartesian lift of $f$ and $p(f')$ is equivalent to $f''$. Consequently, $p$ preserves coCartesian lifts from $\Ec$.
        \item Consider the following commuting diagram of categories:
    \[\begin{tikzcd}
	\Cc && \Dc \\
	& \Ec \\
	& \Oc
	\arrow["{q\circ p}", from=1-1, to=2-2]
	\arrow["p", from=1-1, to=1-3]
	\arrow["q"', from=1-3, to=2-2]
	\arrow["{\pi_1}"{description}, from=2-2, to=3-2]
	\arrow["{\pi_2}"', from=1-1, to=3-2]
	\arrow["{\pi_3}", from=1-3, to=3-2]
    \end{tikzcd}.\]
    Assume that $q$, $q\circ p$ and $\pi_1$ are coCartesian fibrations. Assume further that $p$ preserves coCartesian lifts from $\Ec$. Then $p$ preserves coCartesian lifts from $\Oc$.
    \end{enumerate}
\end{lem}
\begin{proof}
    \begin{enumerate}
        \item That a composition of coCartesian fibration is coCartesian fibration is proved in \HTT{Proposition}{2.4.2.3}. For the second part, given $f\in\Ec$ and a $q\circ p$ coCartesian lift $f'\in\Cc$ of $f$, one can choose $f''\in\Dc$ to be a $q$-coCartesian lift of $f$.  Let $\Bar{f'}\in\Cc$ be a $p$-coCartesian lift of $f''$, then $\Bar{f'}$ would also be a $q\circ p$-coCartesian lift of $f$ using \HTT{Proposition}{2.4.1.3}. We conclude that $\Bar{f'}$ is equivalent to $f'$ and hence $p(f')$ is equivalent to $p(\Bar{f'})=f''$. The last claim about $p$ preserving coCartesian lifts from $\Ec$ then follows.
    \item Let $f'\in \Cc$ be a $\pi_2$-coCartesian lift of $f\in\Oc$. By the previous item, we might assume $f'$ is a $q\circ p$-coCartesian lift of $q\circ p(f')$. Then by assumption on $p$, the image $p(f')\in\Dc$ is a $q$-coCartesian lift of $q\circ p(f')$, hence is a $\pi_3$-coCartesian lift of $f\in\Oc$ as desired.
    \end{enumerate}
\end{proof}
\subsection{Combinatorial v.s. constructible}
Now we take advantage of the functor $\Gamma_\MR$ from the previous section to write down the combinatorial-constructible comparison functor. First, we give a quick idea of the construction.\par Fix toric data $(N,\Sigma)$ and pick a cone $\sigma\in\Sigma$. Recall that we have defined the combinatorial category $\Theta(\sigma)$ to be a full subcategory of $\Closed(M_\RR)$. The category $\Closed(M_\RR)$ has a symmetric monoidal structure given by Minkowski sum and one can think of the symmetric monoidal structure on $\Theta(\sigma)$ as inherited from the inclusion (to be very precise, $\Theta(\sigma)$ includes into the full subcategory  $\Mod_{\sigma^{\svee}}\Closed(M_\RR)$ over the idempotent algebra $\sigma^{\svee}\in\Closed(\MR)$ and this inclusion is symmetric monoidal). Post-composing this inclusion with $\Gamma_\MR$ that we have defined earlier, we get a combinatorial-to-constructible comparison functor. The goal of this section is to construct this functor and present its functoriality along $\Sigma$.\par
 We start with constructing a family of idempotent algebras in $\Shv(\MR;\Sp)$. Here is a technical observation of the interaction of $\Gamma_\MR$ with  polytopes which is conceptually helpful, albeit not necessarily needed.
\begin{lem}
\label{comparison between homology sheaf of open and closed polygon}
    For a closed  polyhedral subset (of top dimension) $\overline U\subseteq \MR$ and its interior $U$, the map of sheaves
    $$\Gamma_\MR(U)\rightarrow\Gamma_\MR(\overline{U})$$
    induced from $U\rightarrow\overline{U}$ is an equivalence. Note that left hand side is a more familiar object: the extension-by-zero of a shift of constant sheaf on an open subset.
\end{lem}
\begin{proof}
    This can be proved by comparing the recollement sequences for $U$ and $\overline{U}$. Here we give a direct proof. In this case, one can check equivalence on stalks. By proper base change, it is easy to check for $x\notin\partial\overline U$ the map is an equivalence on stalk at $x$. It remains to check that at $x\in\partial\overline U$ the stalk of the right-hand side vanishes (again by proper base change it vanishes on the left-hand side). To compute the stalk, one can pick a family of open balls $D_i$ of shrinking radii centered at $x$ and compute 
    $$\Gamma_\MR(\overline{U})_x\simeq\colim\Gamma_\MR(\overline{U})(D_i).$$
    To compute the right-hand side, use the identification $\omega_\MR\simeq\underline\SS[n]$ and apply proper base change to get
        $$\Gamma_\MR(\overline{U})(D_i)\simeq (i_{\overline{U}!}i_{\overline{U}}^!\underline{\SS}[n])(D_i)\simeq \fib[(\SSconst(D_i)\rightarrow\SSconst(D_i\setminus \overline{U})][n].
        $$
        Since $\overline{U}$ is polyhedral, for a sufficiently small ball $D_i\rightarrow D_i\setminus \overline{U}$ is a homotopy equivalence and hence the stalk vanishes, as desired. 
\end{proof}
\begin{prop}[Dualizing sheaf of a cone is an idempotent algebra]
\label{idempotent algebra in sheaf category indexed by the fan}
    For each $\sigma\in\Sigma$, the object $\sigma^\svee\in\Closed(\MR)$ has the structure of an idempotent algebra. Thus, we might think of $\sigma^\svee$ as a diagram of idempotent algebras in $\Closed(\MR)$ indexed by $\sigma\in\Sigma^\op$. Moreover, the image of each $\sigma^\svee$ under $\Gamma_\MR$ is also an idempotent algebra. Therefore, we get $\Gamma_\MR(\sigma^\svee)=\omega_{\sigma^\svee}$ as a diagram of idempotent algebras in $\Shv(\MR;\Sp)$ indexed by $\Sigma^\op$.
\end{prop}
\begin{proof}
    The first observation is direct, using that $\sigma^\svee+\sigma^\svee=\sigma^\svee$ for any cone $\sigma$. For the second assertion, one needs to compute that the multiplication map of the algebra $\Gamma_\MR(\sigma^\svee)$ is an isomorphism
    $$\Gamma_\MR(\sigma^\svee)*\Gamma_\MR(\sigma^\svee)\overset{\simeq}{\longrightarrow}\Gamma_\MR(\sigma^\svee).$$
    By the previous lemma, it is equivalent to showing that $\Gamma_\MR(\sigma^{\svee,\circ})$ is an idempotent algebra. Now that we are working with a polyhedral open subset, we can unpack the definition of multiplication maps and this reduces to the same computation as in \cref{proposition with direct computation of convolution} up to a shift.
\end{proof}
\begin{cor}
    \label{taking module category over the family of idempotent algebras in the sheaf category}
    There is a diagram in $\SMCat$ indexed by $\Sigma^\op$ given by 
    \[\sigma\mapsto\Mod_{\omega_{\sigma^\svee}}\Shv(\MR;\Sp)\in\SMCat.\]
    Furthermore, there is a symmetric monoidal left adjoint functor
    \[L:\Shv(\MR;\Sp)\longrightarrow\lim_{\Sigma^\op}\Mod_{\omega_{\sigma^\svee}}\Shv(\MR;\Sp)\]
    given by tensoring with $\omega_{\sigma^\svee}$ in each component. In particular, this functor has a lax symmetric monoidal right adjoint
    \[R:\lim_{\Sigma^\op}\Mod_{\omega_{\sigma^\svee}}\Shv(\MR;\Sp)\longrightarrow\Shv(\MR;\Sp).\]
    Note that since each $\omega_{\sigma^\svee}$ is an idempotent algebra, the forgetful functor
    \[\Mod_{\omega_{\sigma^\svee}}\Shv(\MR;\Sp)\longrightarrow\Shv(\MR;\Sp)\]
    is a fully faithful functor. Hence $R$ is also fully faithful. One can describe the functor $R$ explicitly as follows: given an object in the limit, one applies forgetful functor  to $\Shv(\MR;\Sp)$ pointwise to get a diagram in $\Shv(\MR;\Sp)$ and then take the limit. See \cref{the idempotent algebras for a smooth projective fan glues to the unit} for more on this functor $R$ and that it is always an equivalence for a smooth projective fan, hence in particular symmetric monoidal.
\end{cor}
We move on to the main construction. The following proposition sketches our goal and the construction will be provided right after.
\begin{prop}
\label{construction of combinatorial to constructible comparison functor}
    There is a symmetric monoidal functor
    $$\Psi_\sigma:\Fun(\Theta(\sigma)^\op,\Sp){\longrightarrow}\Mod_{\omega_{\sigma^\svee}}\Shv(M_\RR;\Sp)$$
    where the left-hand side has the Day convolution tensor product and right-hand side has the convolution product of sheaves. Moreover, these functors are natural in $\sigma\in\Sigma^\op$ that they assemble into a natural transformation of diagrams in $\SMCat$ indexed by $\sigma\in\Sigma^\op$. Hence taking limit produces
        $$\lim_{\Sigma^\op}\Fun(\Theta(\sigma)^\op,\Sp)\overset{\lim\Psi_\sigma}{\longrightarrow}\lim_{\Sigma^\op}\Mod_{\omega_{\sigma^\svee}}\Shv(\MR;\Sp)\overset{R}{\longrightarrow}\Shv(\MR;\Sp).$$
    The first functor is symmetric monoidal. It is fully faithful when the fan is smooth, as shown in \cref{comparison functor is fully faithful when the fan is smooth}.
    The latter functor is the right adjoint functor $R$ in \cref{taking module category over the family of idempotent algebras in the sheaf category} which is lax symmetric monoidal and fully faithful. It is symmetric monoidal when the fan is smooth and projective, as shown in \cref{the idempotent algebras for a smooth projective fan glues to the unit}. In conclusion, when the fan $\Sigma$ is smooth and projective we have a symmetric monoidal fully faithful functor
    \[\Psi_\Sigma:\lim_{\Sigma^\op}\Fun(\Theta(\sigma)^\op,\Sp)\overset{\lim\Psi_\sigma}{\longrightarrow}\lim_{\Sigma^\op}\Mod_{\omega_{\sigma^\svee}}\Shv(\MR;\Sp)\longrightarrow\Shv(\MR;\Sp).\]
\end{prop}
We first construct $\Psi_\sigma$ pointwise. 
\begin{construction}
    Fix $\sigma\in\Sigma$, consider the composition of lax symmetric monoidal functors
    $$\Theta(\sigma)\longrightarrow\Closed(\MR)\overset{\Gamma_\MR}{\longrightarrow}\Shv(\MR;\Sp),$$
    where the first functor is the canonical inclusion (recall that $\Theta(\sigma)$ is {by definition} a full subcategory of $\Closed(\MR)$) and the second functor is $\Gamma_\MR$ constructed in \cref{defintion of a functor taking a space ot its sheaf of homology}. Now we observe that for each $\sigma$, the image of $\Theta(\sigma)$ lies in the full subcategory  $\Mod_{\omega_{\sigma^\svee}}\Shv(\MR;\Sp)$ of $\Shv(\MR;\Sp)$. It follows that we have a lax symmetric monoidal functor 
    $$\psi_\sigma:\Theta(\sigma)\rightarrow\Mod_{\omega_{\sigma^\svee}}\Shv(\MR;\Sp)$$
    which is readily symmetric monoidal by \cref{proposition with direct computation of convolution}.
    Now one can left Kan extend this to a symmetric monoidal functor
    $$\Psi_\sigma:\Fun(\Theta(\sigma)^\op,\Sp)\rightarrow\Mod_{\omega_{\sigma^\svee}}\Shv(\MR;\Sp)$$
    which is what we want. 
\end{construction}
Now we construct $\Psi_\sigma$ with functoriality along $\sigma$.
\begin{rem}[Functoriality of $\Psi_\sigma$ along $\sigma$]
    To provide the functoriality of symmetric monoidal functors 
    \[\Psi_\sigma:\Fun(\Theta(\sigma)^\op,\Sp)\rightarrow\Mod_{\omega_{\sigma^\svee}}\Shv(\MR;\Sp)\]
    along $\sigma\in\Sigma^\op$, we make the following constructions. Consider the subcategory $\Closed^*(\MR)\subseteq\Closed(\MR)$ spanned by the origin and top dimensional polyhedral subsets. This category inherits a symmetric monoidal structure and \cref{proposition with direct computation of convolution} implies that the restriction
    \[\Gamma_{\MR}:\Closed^*(\MR)\longrightarrow\Shv(\MR;\Sp)\]
    is symmetric monoidal\footnote{It is true that $\Gamma_\MR$ is symmetric monoidal on $\Closed(\MR)$, but it takes more effort to show. We restrict to $\Closed^*(\MR)$ as it suffices for our purpose here.}. Now we consider the presheaf category \footnote{Note that we have played the same trick of passing to presheaf category in \cref{definition of the combinatorial-quasicoherent comparison functor}.}
    \[\Fun(\Closed^*(\MR)^\op,\Spc)\]
    equipped with the Day convolution product. We have a family of idempotent  algebras in $\Fun(\Closed^*(\MR)^\op,\Spc)$ given by 
    \[\sigma^\svee\in\CAlg(\Fun(\Closed^*(\MR)^\op,\Spc))\]
    coming from \cref{taking module category over the family of idempotent algebras in the sheaf category}. We have by abuse of notation, identified $\sigma^\svee$ with its Yoneda image. Now $\Gamma_{\MR}$ induces a symmetric monoidal colimit-preserving functor
    \[\Fun(\Closed^*(\MR)^\op,\Spc)\longrightarrow\Shv(\MR;\Sp)\]
    so we can apply \cref{Taking module category} and obtain a diagram in $\SMCat$ indexed by $\sigma\in\Sigma^\op$:
    \[\Mod_{{\sigma^\svee}}\Fun(\Closed^*(\MR)^\op,\Spc)\longrightarrow\Mod_{\omega_{\sigma^\svee}}\Shv(\MR;\Sp).\]
    It remains to write down a natural transformation in $\SMCat$ of diagrams indexed by $\Sigma^\op$
    \[\psi_\sigma:\Theta(\sigma)\longrightarrow\Mod_{{\sigma^\svee}}\Fun(\Closed^*(\MR)^\op,\Spc).\]
    This reduces to $1$-categorical considerations: for example, one way to do this is to identify (as a symmetric monoidal category) $\Theta(\sigma)$ with the full subcategory of $\Mod_{{\sigma^\svee}}\Fun(\Closed^*(\MR)^\op,\Spc)$ spanned by Yoneda image of integral translations of $\sigma^\svee$ as in \cref{alternative way to see symmetric monoidal structure on Theta}. Given $\tau\subseteq\sigma$ in $\Sigma$, the symmetric monoidal functors given by base change
    \[\Mod_{{\sigma^\svee}}\Fun(\Closed^*(\MR)^\op,\Spc)\longrightarrow\Mod_{{\tau^\svee}}\Fun(\Closed^*(\MR)^\op,\Spc)\]
    restricts to structure maps
    \[\Theta(\sigma)\longrightarrow\Theta(\tau).\]
    Consequently we have a natural transformation between diagrams in $\SMCat$ indexed by $\Sigma$:
    \[\Theta(\sigma)\longrightarrow\Mod_{{\sigma^\svee}}\Fun(\Closed(\MR)^\op,\Spc)\longrightarrow\Mod_{\omega_{\sigma^\svee}}\Shv(\MR;\Sp).\]
    Now one can left Kan extend as in \cref{reminders on Day convolution} and obtain the symmetric monoidal functors naturally along $\Sigma^\op$
    \[\Psi_\sigma:\Fun(\Theta(\sigma)^\op,\Sp)\longrightarrow\Mod_{\omega_{\sigma^\svee}}\Shv(\MR;\Sp).\]
\end{rem}
For future use, we record here the interaction of $\Psi_\sigma$ with translation action by lattice $M$.
\begin{construction}[Compatibility with the lattice]
\label{comb-constructible functor is compatible with lattice}
    For convenience, we assume that the fan is smooth and projective. For each $\sigma$, recall that we have a symmetric monoidal inclusion $p_\sigma:M\rightarrow\Theta(\sigma):m\mapsto m+\sigma^\svee$ and one obtains its left Kan extension as a symmetric monoidal functor:
    $$p_{\sigma!}:\Fun(M,\Sp)\rightarrow\Fun(\Theta(\sigma)^\op,\Sp).$$ This functor is natural in $\sigma$ when we view $\Fun(M,\Sp)$ as a constant diagram indexed by $\sigma\in\Sigma^\op$. It follows that we have the following diagram in $\SMCat$ after taking limit:
\[\begin{tikzcd}
	{\lim_{\sigma}\Fun(\Theta(\sigma)^\op,\Sp)} & {\Shv(\MR;\Sp)} \\
	{\Fun(M,\Sp)}
	\arrow["{\Psi_\Sigma}", from=1-1, to=1-2]
	\arrow["{\lim_{\sigma}p_{\sigma!}}", from=2-1, to=1-1]
	\arrow["{\Psi_\Sigma\circ\lim_{\sigma}p_{\sigma!}}"', from=2-1, to=1-2]
\end{tikzcd}.\]
    It turns out that one can identify the diagonal functor with a more familiar one when working with a smooth projective fan $\Sigma$: (on the bottom we view $M$ as a discrete topological group, and $i_!$ is the !-pushforward along the inclusion of topological groups which is symmetric monoidal for sheaf categories with convolution products)
\[\begin{tikzcd}
	{\lim_{\sigma}\Fun(\Theta(\sigma)^\op,\Sp)} & {\Shv(\MR;\Sp)} \\
	{\Fun(M,\Sp)} & {\Shv(M;\Sp)}
	\arrow["{\Psi_\Sigma}", from=1-1, to=1-2]
	\arrow["{\lim_{\sigma}p_{\sigma!}}", from=2-1, to=1-1]
	\arrow["\simeq"', from=2-1, to=2-2]
	\arrow["{i_!}"', from=2-2, to=1-2]
\end{tikzcd}.\]
    As we don't need it for now, we defer its proof to \cref{equivariantization for large sheaf category}.

\end{construction}
We summarize the constructions so far by making the following definition.
\begin{defn}
\label{definition of the CCC functor}
Let $\Sigma$ be a smooth projective fan. Combining \cref{theorem of combinatorics compares to quasicoherent} and \cref{construction of combinatorial to constructible comparison functor}, we have the following functors
\[\QCoh([X_\Sigma/\TT]\overset\simeq\longleftarrow\lim_{\sigma\in\Sigma^\op}\Fun(\Theta(\sigma)^\op;\Sp)\longrightarrow\lim_{\sigma\in\Sigma^\op}\Mod_{\omega_{\sigma^\svee}}\Shv(\MR;\Sp)\longrightarrow\Shv(\MR;\Sp),\]
where the first two functors are symmetric monoidal functors supplied by $\lim\Phi_\sigma$ and $\lim\Psi_\sigma$. We take the inverse of the first functor and obtain the \stress{coherent-constructible correspondence} functor 
\[\kappa:\QCoh([X_\Sigma/\TT])\longrightarrow\Shv(\MR;\Sp),\]
which is symmetric monoidal and fully faithful when $\Sigma$ is smooth projective, in view of \cref{construction of combinatorial to constructible comparison functor}. Given \cref{global combinatorial-coherent functor is compatible with torus action} and \cref{comb-constructible functor is compatible with lattice}, we also have the following diagram in $\SMCat$:
\[\begin{tikzcd}
	{\QCoh(X_\Sigma/\TT)} & {\Shv(M_\RR;\Sp)} \\
	{\QCoh(B\TT)} & {\Shv(M;\Sp)} \\
	\arrow["\kappa"', from=1-1, to=1-2]
	\arrow["{\pi^*}", from=2-1, to=1-1]
	\arrow["\simeq"', from=2-1, to=2-2]
	\arrow["{i_!}", from=2-2, to=1-2]
\end{tikzcd}.\]
\end{defn}
\subsection{Polyhedral stratification}
The goal of this subsection is twofold: on the one hand, we show that the functors $\Psi_\sigma$ previously constructed are fully-faithful;
on the other hand, we pin down a first-order approximation of the characterization of the image of $\kappa$.
That is to say, we will not actually work with the whole (gigantic) category of sheaves, but only a subcategory: those constructible for some fixed stratification. Moreover, the stratification has an elementary description in terms of the fan data.
We start with a quick review on constructible sheaves following \cite{clausenJansen2023reductiveborelserre}.
\begin{defn}
    A poset $P$ satisfies the \stress{ascending chain condition} if every strictly increasing chain in $P$ stops after finitely many steps.
    A poset $P$ is \stress{locally finite} if each $P_{\geq q}\defeq\{p:p\geq q\}$ is finite.
\end{defn}
\begin{rem}
    Note that locally finite implies the ascending chain condition but not vice versa.
\end{rem}
\begin{defn}
    A \stress{stratified topological space} is a continuous map $\pi:X\rightarrow P$ where $X$ is a topological space and $P$ is a poset equipped with the Alexandroff topology
    \footnote{
    Recall that a subset $U \subseteq P$ is open in the Alexandroff topology if and only if for $p\in U$, $p\leq q$ implies $q\in U$. In other words, $U$ is a `cosieve': a subset that is upward closed for the partial order of $P$.}. We often write $(X,P)$ for a stratified topological space and omit the map $\pi$.
    For each $p\in P$, the preimage $\pi^{-1}(p)\subseteq X$ is called its \stress{$p$-stratum} $X_p$. The stratum $X_p$ is a closed subspace of $\openstar_p\defeq\pi^{-1}
    \{q:p\leq q\}\subseteq X$, the \stress{open star around $p$}.
\end{defn}
\begin{defn}
    Given two maps $f,g:(X,P)\rightarrow(Y,Q)$ between stratified topological spaces. A \stress{stratified homotopy} between $f$ and $g$ is a map $H:X\times[0,1]\rightarrow Y$ between stratified topological spaces that restricts to $f$ on $X\times{0}$ and $g$ on $X\times 1$. Here $X\times[0,1]$ inherits the stratification from $X$.
\end{defn}
\begin{defn}
    A map of stratified topological space $f:(X,P)\rightarrow(Y,Q)$ is a \stress{stratified homotopy equivalence} if there is a map $g:(Y,Q)\rightarrow(X,P)$ going in the other direction, such that both of $f\circ g$ and $g\circ f$ are stratified homotopic to identity. 
\end{defn}
\begin{defn}
    Fix a compactly generated category $\Cc$ (we will only care about $\Spc$ or $\Sp$) as coefficient and a stratified topological space $\pi:X\rightarrow P$.
    A sheaf on $X$ valued in $\Cc$ is \stress{$P$-constructible}
    \footnote{These are sometimes called quasi-constructible in the literature, where the word constructible is reserved for objects also satisfying a finiteness condition which we don't impose here.}
    if its restriction to each stratum $X_p$ is locally constant.
    We write $\Cons_P(X;\Cc)$ for the full subcategory of \stress{$P$-construcitible sheaves}.
\end{defn}
We want to take advantage of the exodromy equivalence to identify a family of compact generators for the category of constructible sheaves.
We start by importing the following theorem which realizes the exodromy equivalence for a class of particularly simple stratified topological spaces.

\begin{thm}\label{theorem:exodromysimple}
\cite[\href{https://arxiv.org/pdf/2108.01924.pdf\#theorem.3.4}{Theorem 3.4}]{clausenJansen2023reductiveborelserre}
Let $\pi:X\rightarrow P$ be a stratified topological space with $\pi$ surjective and $P$ satisfying the ascending chain condition. Suppose there is a collection $\Bc$ of open subsets of $X$ such that
\begin{enumerate}
    \item the representable sheaves $h_U$ for $U\in\Bc$ generate the topos $\Shv(X;\Spc)$.
    \item for all $U\in\Bc$, there is a $p\in P$ such that $U$ includes into $\openstar_p$ by a stratified homotopy equivalence.
\end{enumerate}
Then the pullback map $$\pi^*:\Fun(P,\Spc)\rightarrow\Shv(X;\Spc)$$
preserves all limits and colimits and is fully faithful with essential image $\Cons_P(X;\Spc)$.
\end{thm}
\begin{rem}
    The theorem in \cite{clausenJansen2023reductiveborelserre} was stated and proved for sheaves valued in $\Spc$. The proof works verbatim for $\Sp$ coefficient.
    It is also true for other compactly generated coefficient categories, which isn't needed for our exposition. 
\end{rem}
This gives, for a locally finite poset $P$ and stratification $X\rightarrow P$ as above, an explicit realization of the exodromy equivalence
$$\pi^*:\Fun(P,\Sp)\rightarrow\Shv(X;\Sp)$$
which is the left adjoint of $\Shv(X;\Sp)\rightarrow\Fun(P,\Sp)$ sending $\Fc$ to $[q\mapsto\Fc(\openstar_q)]$.
Tracing through the equivalences, one sees that for $q\in P$, the image of $q$ under stable Yoneda embedding (i.e. $\SS[\Map_P(q,-)]$) is taken to $i_{\openstar_q!}(\underline\SS)$ where $i_{\openstar_q}$ is the inclusion of $\openstar_q$ into $X$.
\begin{cor}
    Let $\pi:X\rightarrow P$ be as in \cref{theorem:exodromysimple}. The category $\Cons_P(X;\Sp)$ is {generated} by the compact objects $\{i_{\openstar_q!}(\underline\SS)\}_{q\in P}$ in the following sense: the smallest cocomplete stable subcategory 
    of $\Cons_P(X;\Sp)$ that contains these objects is itself.
\end{cor}
For future use we record a description of compact objects in functor categories here:
\begin{lem}
\label{characterization of compact in functor category on poset}
    Let $P$ be a locally finite poset. A functor $X\in\Fun(P,\Sp)$ is compact if and only if its value is none-zero on finitely many objects and each of its value is a finite spectrum. 
\end{lem}
\begin{proof}
    This is the same as \cite[Proposition 2.2.6]{LurieRotation} which proves the case where $P$ is finite.
    If a functor $F$ is compact, then as in \cite[Proposition 2.2.5]{LurieRotation}, it is a retract of an object in the smallest stable subcategory of $\Fun(P,\Sp)$ containing the image of stable Yoneda embedding.
    By the assumption that $P$ is locally finite, each of the Yoneda functor $\SS[\Map(x,-)]$ is non-zero on finitely many objects and each of its value is a finite spectrum. Such condition cuts out a stable subcategory $\Dc$ of $\Fun(P,\Sp)$ which is closed under retract, so we know that $F$ is in $\Dc$. On the other hand, if $F$ is non-zero on finitely many objects and each of its value is a finite spectrum, we may take  the subposet $\supp^\star(F):=\{y\in P:\exists x\leq y,F(x)\neq 0\}$. By assumption, this is a finite poset and the restriction of $F$ to $\supp^\star(F)$ is thus compact. Now note that $F$ is the left Kan extension of its restriction to $\supp^\star(F)$ and that left Kan extension preserves compact objects.
\end{proof}
Now we specialize to the case of our interest:
\begin{defn}[FLTZ stratification] 

\label{definition of FLTZ stratification} (See also \cite[Definition 4.3]{treumann2010remarksnonequivariantcoherentconstructiblecorrespondence})
    Fix a pair $(N,\Sigma)$ of lattice and fan, and assume further that $\Sigma$ spans $N_\RR$ as an $\RR$-vector space. We define a stratification $\Sc_\Sigma$ on $M_\RR$ as follows. To start with, one has an affine hyperplane arrangement in $M_\RR$ given by 
    $$H_\Sigma\defeq\{m+\sigma^\perp:m\in M,\sigma\in\Sigma(1)\},$$ 
    where $\sigma^\perp\defeq\{m\in M:(m,n)=0\,\forall\, n\in \sigma\}$. One has the following induction procedure to specify strata of a stratification: first look at the complement
    $$V\defeq M_\RR\setminus\bigcup_{h\in H_\Sigma}h,$$
    where each of the connected component of $V$ should be considered as a single stratum. For each $h\in H_\Sigma$, intersecting $h'\in H_\Sigma$ with $h$ produces an affine hyperplane arrangement on $h$. Thus one can work inductively and define a poset of strata $\Sc_\Sigma$ of $M_\RR$ (note they are locally closed). The closure of each stratum is a union of strata and one specify a poset structure by closure-inclusion. The map sending each point in $M_\RR$ to the stratum it belongs to in $\Sc_\Sigma$ will be a continuous map and this gives a stratification on $M_\RR$. We refer to this stratification $\Sc_\Sigma$ as the \stress{FLTZ stratification} for $\Sigma$ and we will often omit mentioning $\Sigma$ when it is clear from the context.
\end{defn}
\begin{rem}
    Note that the FLTZ stratification only depends on the collection of $1$-cones in $\Sigma$.
\end{rem}
\begin{rem}
    The stratification $\Sc_\Sigma$ on $\MR$ is apparently sub-analytic in the sense of \cite[Definition 5.3.7]{haine2024exodromyconicality} so the exodromy equivalence applies directly. We choose to apply exodromy equivalence from \cite{clausenJansen2023reductiveborelserre} since it's more elementary.
\end{rem}
We wish to use exodromy equivalence \cref{theorem:exodromysimple} to understand $\Sc_\Sigma$-constructible sheaves. For that we need:
\begin{prop}
    The stratification $\Sc_\Sigma$ on $M_\RR$ meets the assumptions of \cref{theorem:exodromysimple} above.
\end{prop}
\begin{proof}
    We need to provide a basis of opens for $M_\RR$ with desired properties. Consider the standard basis
    $$\Bc\defeq\{D(x,r):\text{open ball of radius $r$ centered at }x\in M_\RR \}$$
    and a subset of it.
    $$\Bc(\Sc_\Sigma)\defeq\{D(x,r)\in\Bc:D(x,r) \text{ is stratified homotopy equivalent to the open star at $x$}\}$$
    By definition each $D(x,r)\in\Bc(\Sc_\Sigma)$ would go through point 2.
    It suffices to check point 1, that it is a basis (or at the very least, nonempty). We claim that: for each $x\in M_\RR$ there exists $r_x>0$ such that $r<r_x$ implies $D(x,r)\in\Bc(\Sc_\Sigma)$. This directly implies that $\Bc(\Sc_\Sigma)$ is a basis of opens for $M_\RR$.
    To prove the claim, a first observation is that for sufficiently small $r$, $D(x,r)$ with restricted stratification of $\Sc_\Sigma$ is (stratified) isomorphic to a real vector space with stratification given by a family of hyperplane arrangements.
    There is no other stratum coming into the picture than those passing through $x$. Fix such small $r_x$, then for all $r\leq r_x$, each $D(x,r)$ includes into each other as a stratified homotopy equivalence.
    It remains to prove that $D(x,r_x)$ is stratified homotopy equivalent to the open star at $x$.
    For this a straight-line linear homotopy should do the work.
    Note this works because the open star is convex and the linear scaling towards $x$ respects the stratification.
\end{proof}
\begin{rem}
    Note also that for a smooth projective fan $\Sigma$ the poset $P$ underlying the stratification $\Sc_\Sigma$ is locally finite, as each stratum is only in the closure of finitely many other strata, and each exit-path will enter a higher dimensional stratum, so must stop after finitely many steps.
\end{rem}
The reason to introduce $\Sc_\Sigma$ is the following:
\begin{prop}
    Fix a pair $(N,\Sigma)$ of lattice and fan.
    One might post-compose the functor \[\Psi_\sigma:\Fun(\Theta(\sigma)^\op,\Sp)\longrightarrow\Mod_{\omega_{\sigma^\svee}}\Shv(\MR;\Sp)\]  in \cref{construction of combinatorial to constructible comparison functor} with forgetful into $\Shv(\MR;\Sp)$, then its image all lands into the subcategory $\Cons_{\Sc_\Sigma}(M_\RR;\Sp)$ of sheaves constructible for the FLTZ stratification. As a consequence, the functor 
    \[\lim_{\Sigma^\op}\Fun(\Theta(\sigma)^\op,\Sp)\longrightarrow\Shv(\MR;\Sp)
    \]
    of \cref{construction of combinatorial to constructible comparison functor} also lands into $\Cons_{\Sc_\Sigma}(M_\RR;\Sp)$.
\end{prop}
\begin{proof}
    It suffices to note that each $U\in\Theta(\sigma)$ is given by a cone bound by the hyperplane arrangement $H_\Sigma$. Any stratum of the stratification would be either contained in it or be disjoint from it. Using \cref{comparison between homology sheaf of open and closed polygon} and proper base change, it follows that $\Gamma_\MR(U)$ is constructible for the FLTZ stratification. Now the image of $\Theta(\sigma)$ is colimit generated by these objects as a stable category, and constructible sheaf category is also closed under colimits, so we are done.
\end{proof}
We give a standard example to illustrate the ideas of the definitions so far.
\begin{example}
Take the fan spanned by $\{e_1,\dots,e_n\}\subset\ZZ^n=N$. To be more precise, $\Sigma=\{\mathrm{span}(S):S\subseteq\{e_1,\dots,e_n\}\}$. This is the fan corresponding to $\AA^n$ in toric geometry. It specifies the standard grid in $M_\RR\simeq \RR^n$ as the FLTZ stratification. The strata of $M_\RR\rightarrow\Sc_\Sigma$ are faces of the unit hypercubes whose vertices have integer coordinates. More precisely, each stratum is cut out by equalities $\{x_i=n_i:i\in I\}$ and inequalities $\{x_j\in(n_j,n_j+1):j\in J\}$ where $n_i$ and $n_j$ are integers and the pair $(I,J)$ is a decomposition of $\{1,\dots,n\}$. The open stars in this case are also very explicit: they are certain hyper-rectangles whose vertices have integer coordinates. Using \cref{proposition with direct computation of convolution} one can compute the convolution product of representable sheaves on these open stars and it turns out to be again $\Sc_\Sigma$-constructible. It follows that in this case $\Cons_{\Sc_\Sigma}(M_\RR;\Sp)$ is closed under convolution product.
\end{example}
\begin{war}
    The convolution product usually doesn't interact well with the FLTZ stratification $\Sc_\Sigma$. More precisely, for a fixed $\Sigma$, the convolution product of two $\Sc_\Sigma$-constructible sheaves needs not to stay $\Sc_\Sigma$-constructible. We will see later how to fix this.
\end{war}

\begin{cor}\label{comparison functor is fully faithful when the fan is smooth}
    For a pair $(N,\Sigma)$ with the fan $\Sigma$ being smooth, the functor $\Psi_\Sigma$ constructed in \cref{construction of combinatorial to constructible comparison functor} is fully faithful. Moreover, for each $\sigma\in\Sigma$, the functor $\Psi_\sigma$ is fully faithful.
\end{cor}
\begin{proof}
    Fix such $\sigma$, by the assumption on the smoothness, one can perform a linear transform in $\mathrm{SL}(n,\ZZ)$ which takes  $\sigma$ to the cone $\{e_1,\dots,e_k\}$ in the standard fan $\{e_1,\dots,e_n\}\subset N=\ZZ^n$ as in the previous example. So without loss of generality, we will prove for this standard case the functor $\Psi_\sigma$ is fully faithful. Recall that $\Psi_\sigma$ is of the form 
    $$\Psi_\sigma:\Fun(\Theta(\sigma)^\op,\Sp){\longrightarrow}\Mod_{\omega_{\sigma^\svee}}\Shv(M_\RR;\Sp)$$
    and we note that it first of all factors through the full subcategory $\Cons_{\Sc_\Sigma}(M_\RR;\Sp)\cap \Mod_{\omega_{\sigma^\svee}}\Shv(M_\RR;\Sp)$ of FLTZ constructible sheaves \stress{inside} $\Mod_{\omega_{\sigma^\svee}}\Shv(M_\RR;\Sp)$ (for the standard fan $\Sigma$ spanned by $\{e_1,\dots,e_n\}$ as above). The domain category is a compactly generated presentable stable category, with a set of compact generators supplied by the stable Yoneda image of representables. By construction of the functor $\Psi_\sigma$, it is fully faithful on this set of compact generators. Let's try to give an explicit description of the intersection $\Cons_{\Sc_\Sigma}(M_\RR;\Sp)\cap \Mod_{\omega_{\sigma^\svee}}\Shv(M_\RR;\Sp)$. We make the following observations:
    \begin{enumerate}
        \item The image of $\Psi_\sigma(\sigma^{\svee})$ is an idempotent algebra for the constructible sheaf category $\Cons_{\Sc_\Sigma}(M_\RR;\Sp)$ equipped with convolution monoidal structure. As before, we denote $\omega_{\sigma^\svee}$ for this algebra and consider the category $\Mod_{\omega_{\sigma^\svee}}\Cons_{\Sc_\Sigma}(M_\RR;\Sp)$. This is a category compactly generated by the convolution of representable sheaves on open stars with $\omega_{\sigma^\svee}$. From the previous example we know explicitly these open stars are integral hyper-rectangles, and the convolution products are (shifts of) representable sheaves on $\sigma^{\svee,\circ}+m$ for $m\in M$. Note that these are precisely the image of $\sigma^{\svee,\circ}+m$ under $\Psi_\sigma$.
        \item It follows that wehave  \[\Mod_{\omega_{\sigma^\svee}}\Shv(M_\RR;\Sp)\cap\Cons_{\Sc_\Sigma}(M_\RR;\Sp)=\Mod_{{\omega_{\sigma^\svee}}}\Cons_{\Sc_\Sigma}(M_\RR;\Sp)\] and the functor $\Psi_\sigma$ lands in this full subcategory. Moreover $\Psi_\sigma$ takes a set of compact generators (representable presheaves on $\sigma^{\svee,\circ}+m$) to compact objects in the target $\Mod_{{\omega_{\sigma^\svee}}}\Cons_{\Sc_\Sigma}(M_\RR;\Sp)$, and is fully faithful on these compact generators.
    \end{enumerate}
    We apply the following \cref{lemma:compactfullyfaithful} and learn that $\Psi_\sigma$ is fully faithful. It follows that $\lim_{\Sigma^\op}\Psi_\sigma$ is also fully faithful. Now $\Psi_\Sigma$ is a composition of two fully faithful functors, and hence is itself fully faithful.
\end{proof}
\begin{lem}\label{lemma:compactfullyfaithful}
    Let $\Cc$ be a compactly generated presentable stable category, with a chosen set of compact generators $S$ (in other words, the smallest stable cocomplete full subcategory of $\Cc$ that contains $S$ is $\Cc$ itself). Given a cocontinuous functor $F:\Cc\longrightarrow\Dc$ with $\Dc$ a presentable stable category. Assume that $F$ is fully faithful on $S$, and it takes $S$ to compact objects in $\Dc$. Then $F$ is fully faithful on all of $\Cc$.
\end{lem}
\subsection{Digression: Gluing of idempotents in sheaf category}
\label{subsection on idempotent descent}
This subsection is meant to answer the following question: can one give a description of $\Shv(\MR;\Sp)$ similar to the limit diagram provided by Zariski descent for $\QCoh(-)$? In order to answer this question, we first recall how descent works in a presentably symmetric monoidal category with idempotent algebras. The following material is taken from \cite[\href{https://drive.google.com/file/d/1xXzBAFnRwooCVrnLS4YyDlZKVQydqJNy/view?usp=sharing}{Lecture 8}]{ClausendeRham}.
\begin{defn}\HA{Definition}{4.8.2.1}
    Let $\Cc$ be a presentably symmetric monoidal category. The category of idempotent objects $\Cc^\idem\subset\Fun([1], \Cc)$ is the full subcategory of pairs $(A,f:1_\Cc\rightarrow A)$ such that $f\otimes A:A\rightarrow A\otimes A$ is an equivalence.
\end{defn}
We also recall the following facts:
\begin{enumerate}
    \item \HA{Proposition}{4.8.2.9} Take $\CAlg(\Cc)^\idem$ to be the full subcategory of $\CAlg(\Cc)$ spanned by $A\in\CAlg(\Cc)$ such that the unit map makes $A$ into an idempotent object of $\Cc$. The forgetful functor $\CAlg(\Cc)^\idem\rightarrow\Cc^\idem$ is an equivalence. In particular every idempotent object acquires uniquely a commutative algebra structure.
    \item \HA{Proposition}{4.8.2.4} Take $A\in\Cc^\idem$. The functor $\Cc\rightarrow\Mod_A(\Cc)$ is a localization. In particular the forgetful $\Mod_A(\Cc)\rightarrow\Cc$ is fully faithful, with image those $X\in\Cc$ such that $X\rightarrow X\otimes A$ is an equivalence.
    \item 
    \cite[Lemma 5 of Lecture 8]{ClausendeRham} The category $\Cc^\idem$ is a poset.
    \item \cite[Lemma 5 of Lecture 8]{ClausendeRham} As a poset $\Cc^\idem$ has all joins (unions) and finite meets (intersections). The join of $A$ and $B$ is computed as $A\otimes B$, and the join of an infinite family $\{A_i:i\in I\}$ is computed as the filtered colimit over the join of finite subsets (in the underlying category).  
    $$A\vee B=A\otimes B$$
    $$\vee_{i\in I} A_i = \colim_{J\subset I \text{, finite}} \bigotimes_{j\in J}A_j$$
    The meet of $A$ and $B$ is computed as fiber of $A\times B\rightarrow A\otimes B$ and the meet of a \emph{finite} family $\{A_i:i\in I\}$ is computed as the limit over the poset of nonempty subsets $J\subset I$ of the functor $J\mapsto\otimes_{j\in J}A_j$ (in the underlying category). Note that the limit diagram would be a cubical diagram as \HA{Proposition}{1.2.4.13}.
    $$A\wedge B= A \underset{A\otimes B}{\times} B$$
    $$\wedge _{i\in I} A_i = \lim_{J\subset I \text{, nonempty}}\bigotimes_{j\in J}A_j$$
    \item \cite[Theorem 4 of Lecture 8]{ClausendeRham} One can define a Grothendieck topology on $\Cc^{\idem,\op}$ as follows: a family of maps $\{f_i:A\rightarrow A_i\in\Cc^\idem\}$ is a cover if it contains a finite sub-family of maps $\{f_i:A\rightarrow A_i\in\Cc^\idem\}$ presenting $A$ as the meet for $\{A_i\}$ in $\Cc^\idem$.
\end{enumerate}
\begin{thm} The presheaf $\Mod_{(-)}(\Cc):\Cc^\idem\rightarrow\SMCat$ which takes $A$ to $\Mod_A(\Cc)$ is a sheaf for above topology.
\end{thm}
\begin{proof}
    This is the same as Theorem 4 in \cite[\href{https://drive.google.com/file/d/1xXzBAFnRwooCVrnLS4YyDlZKVQydqJNy/view?usp=sharing}{Lecture 8}]{ClausendeRham}.
\end{proof}
Now we run this machine in practice. The most important example for us is the following:
\begin{example}[Zariski descent in algebraic geometry]
    For a scheme $X$ and an open $U\subset X$, $*$-pushforward of the structure sheaf $i_*\Oc_U$ is an idempotent algebra in $\QCoh(X)$ (equipped with standard tensor product of quasi-coherent sheaves). The category of modules can be identified as 
    \[\Mod_{i_*\Oc_U}\QCoh(X)\simeq\QCoh(U).\]    
    If a finite family $\{U_i\}$ form a Zariski cover of $X$, one can show that the family $\{\mathbb{1}_{\QCoh(X)}\rightarrow i_*\Oc_U\}$ is a cover. Evaluating $\Mod_{(-)}(\QCoh(X))$ on this cover recovers Zariski descent for $\QCoh(-)$.
\end{example}
Our goal is to formulate a convolution-of-sheaf version of such phenomenon. We fix a \stress{smooth projective} fan $\Sigma$ on $N$ until the end of the subsection. Let's consider the family of idempotent algebras for the convolution product \[\omega_{\sigma^\svee}\in\Shv(\MR;\Sp)\] introduced in \cref{idempotent algebra in sheaf category indexed by the fan}. 
\begin{prop}
\label{the idempotent algebras for a smooth projective fan glues to the unit}
    Let $\Sigma$ be a smooth projective fan. Take the subset $\Sigma(n)\subset\Sigma$ of the top dimensional cones, then $\{\mathbb{1}_{\Shv(M_\RR;\Sp)}\rightarrow\omega_{\sigma^\svee}:\sigma\in\Sigma(n)\}$ is a cover of $\mathbb{1}_{\Shv(M_\RR;\Sp)}$. More explicitly, one has the following equivalences
 \[\mathbb{1}_{\Shv(M_\RR;\Sp)}\overset\simeq\longrightarrow\lim_{\sigma\in\Sigma^\op}\omega_{\sigma^\svee}\overset\simeq\longrightarrow \lim_{S\in\mathcal{P}_{\neq\emptyset}(\Sigma(n))}\bigotimes_{\tau\in S}\omega_{\tau^\svee},\]
    where the second map is an equivalence by the following observations.
\end{prop}
\begin{rem}We make the following observations about the diagrams.
    \begin{itemize}
        \item Fix a smooth projective fan $\Sigma$. There exits an adjunction \[ l:\mathcal{P}_{\neq\emptyset}(\Sigma(n))\rightleftharpoons\Sigma^\op :r\] between the poset $\Sigma(n)$ of nonempty subsets of top dimensional cones and the opposite poset of all cones in the fan $\Sigma$. The map $l$ sends a subset $S\subseteq\Sigma(n)$ to the intersection \[l(S)\defeq\bigcap_{\sigma\in S}\sigma\in\Sigma^\op.\]
        The map $r$ sends a cone $\tau\in\Sigma^\op$ to all the top dimensional cones containing it \[r(\tau)\defeq\{\sigma\in\Sigma(n):\tau\subseteq\sigma\}.\]
        
        The adjunction reduces to the following observation: the intersection of cones in a subset $S$ contains $\tau$ if and only if $S$ is contained in $r(\tau)$ which is the subset of all the cones containing $\tau$. In particular, the map $l$ is a final functor, or a limit-equivalence. Moreover, note that the composition $l\circ r$ is the identity map on $\Sigma^\op$. 
        
        \item The composite of functors  
        \[\mathcal{P}_{\neq\emptyset}(\Sigma(n))\overset{l}{\longrightarrow}\Sigma^\op\longrightarrow\CAlg(\Shv(\MR;\Sp))^\idem\]
        \[S\mapsto l(S)\mapsto\omega_{l(S)^\svee}\]
        can be identified with the \v Cech diagram
        \[\mathcal{P}_{\neq\emptyset}(\Sigma(n))\longrightarrow\CAlg(\Shv(\MR;\Sp))^\idem\]
        \[S\mapsto\bigotimes_{\tau\in S}\omega_{\tau^\svee}.\]
        This follows from that $\omega_{\sigma^\svee}*\omega_{\tau^\svee}\simeq\omega_{(\sigma\cap\tau)^\svee}$, which is a consequence of the combinatorial fact $\sigma^\svee+\tau^\svee=(\sigma\cap\tau)^\svee$ and the computation of convolution (as in \cref{proposition with direct computation of convolution}). The combinatorial fact $\sigma^\svee+\tau^\svee=(\sigma\cap\tau)^\svee$  is a direct consequence of the separation lemma in \cite[(11) and (12) of Section 1.2]{fulton}.
    \end{itemize}
\end{rem}
\begin{proof}[Proof of \cref{the idempotent algebras for a smooth projective fan glues to the unit}]
    By finality, we can switch to the diagram indexed by $\Sigma^\op$. We need to show that
    $$\mathbb{1}_{\Shv(M_\RR;\Sp)}\rightarrow\lim_{\sigma\in\Sigma^\op}\omega_{\sigma^\svee}$$
    is an equivalence. Let's compute the stalk of the limit.
    At the origin, the stalk is
    \[\lim_{\sigma\in\Sigma^\op}(\omega_{\sigma^\svee})_0\simeq\lim_{\sigma\in\Sigma^\op}\SS_{\{0\}}[n]\in\Sp,\]
    where $\SS_{\{0\}}:\Sigma^\op \rightarrow \Sp$ is the presheaf on $\Sigma$ that takes value $\SS$ at the origin and zero otherwise.
    To evaluate the limit, note that there is a fiber sequence in
    $\Fun(\Sigma^\op,\Sp)$
    \[
    \SS_{\{0\}} \rightarrow \underline\SS\rightarrow\SS_{\Sigma^\op\setminus\{0\}},
    \]
    where $\underline\SS$ is the constant presheaf and $\SS_{\Sigma^\op\setminus\{0\}}$ is the right Kan extension of the constant presheaf on $\Sigma\setminus\{0\}$.
    Taking global sections, we get the fiber sequence
    \[
        \lim_{\sigma\in\Sigma^\op} \SS_{\{0\}} \rightarrow 
        \SS \rightarrow \lim_{\sigma\in \Sigma^\op}\SS_{\Sigma^\op\setminus\{0\}},
    \]
    where the last term can be further computed by
    \begin{align*}
    \lim_{\sigma\in \Sigma^\op}\SS_{\Sigma^\op\setminus\{0\}} &
    \simeq \lim_{\sigma \in \Sigma^\op \setminus \{0\}} \underline{\SS} \\
    & \simeq \SS \oplus \SS[-n + 1].
    \end{align*}
    Indeed, $\Sigma^\op\setminus\{0\}$ can be identified with the opposite of the exit path category of $S^{n-1}$ with the stratification induced by the fan,
    and taking global sections of the constant presheaf $\underline{\SS}$
    thus computes the cotensor
    \[
    \SS^{S^{n - 1}} \simeq \SS \oplus \SS[-n + 1].
    \]
    Under this identification, $\SS \rightarrow \SS \oplus \SS[- n + 1]$
    is the inclusion of the first factor.
    Consequently,
    \[
    \lim_{\sigma\in\Sigma^\op}\SS_{\{0\}} \simeq \SS[-n]
    \]
    and thus
    \[
    \lim_{\sigma\in\Sigma^\op}(\omega_{\sigma^\svee})_0 \simeq \lim_{\sigma\in\Sigma^\op}\SS_{\{0\}}[n] \simeq \SS
    \]
    as desired.
    
    Next we compute the stalk of the limit at $m\in M_\RR$ (which is away from the origin).
    Similarly, we look at the limit 
    $$\lim_{\sigma\in\Sigma^\op}\SS_{m,+}[n],$$
    where $\SS_{m,+}$ is the functor which evaluates on $\sigma$ to be $\SS$ if $m\in\sigma^{\svee,\circ}$ and $0$ otherwise.
    To be precise, it fits into a fiber sequence
    \[\SS_{m,+}\rightarrow\underline\SS\rightarrow\SS_{m,-}\] 
    in $\Fun(\Sigma^\op,\Sp)$.
    Here $\underline\SS$ is the constant functor and $\SS_{m,-}$ is right Kan extended from the constant presheaf on the sub-poset \[\Sigma_{m,-} \defeq\{\sigma:m\notin\sigma^{\svee,\circ}\}\subseteq\Sigma\] (one can check that the right Kan extension takes everything outside of $\Sigma^\op_{m,-}$ to $0$).
    We claim that the limit along $\Sigma^\op$ of the map
    \[\underline{\SS}\rightarrow\SS_{m,-}\]
    is an isomorphism $\SS\rightarrow\SS$.
    It suffices to show that the poset $\Sigma^\op_{m,-}$ is contractible.
    For that, we make the following combinatorial argument.

    We will adapt the proof of \cite[Proposition 3.7]{FLTZ} to our situation. We fix a moment polytope $P$ for the fan $\Sigma$. Consider the poset $F(P)$ of faces of $P$ under inclusion, then there is an (inclusion) order reversing bijection between $F(P)$ and $\Sigma$. For example, the codimension $0$ face of $P$ (which is $P$ itself) corresponds to the $0$ dimensional cone of the origin. Now we consider the following subposet: (informally, the subset of $F(P)$ that's visible from $\infty$ through the direction $m$) 
    \[F(P)_{m,-}\defeq\{C\in F(P):\forall c\in C\text{, the ray $c+\RR_{>0}\cdot m$ doesn't meet $P^\circ$}\}.\]
    We claim that the canonical bijection between $F(P)$ and $\Sigma^\op$ induces a bijection between $F(P)_{m,-}$ and $\Sigma^\op_{m,-}$.
    This follows readily from the definition: if $m\notin\sigma^{\svee,\circ}$, then $m$ is also not in $\tau^{\svee,\circ}$ for $\sigma\subseteq\tau$.
    Consider the corresponding face $C_\sigma$ in $P$, at each point $c\in C$, the angle spanned by $P$ is $\tau^\svee$ for some $\sigma\subseteq\tau$, which means that the ray $c+t\cdot m$ will not pass through $P^\circ$.
    Conversely, let $m\in\sigma^{\svee,\circ}$ for some $\sigma$ (so $\sigma\notin\Sigma^\op_{m,-}$) and $C_\sigma$ be the corresponding face in $P$,
    it follows that at a relative interior point $c$ of $C_\sigma$,
    the angle spanned by $P$ is precisely $\sigma^\svee$, and that $m\in\sigma^{\svee,\circ}$ means that the ray $c+t\cdot m$ will pass through $P^\circ$. 
    Now the topological space $P_{m,-}$ of union of faces in $F(P)_{m,-}$ is contractible because if one fixes a hyperplane $H$ perpendicular to $m$ and consider projection to $H$ along $m$, the image of $P_{m,-}$ is the same as $P$, which is convex.
    HOwever, the map from $P_{m,-}$ to its image is a homotopy equivalence.
    Hence we conclude that $P_{m,-}$ is contractible.
    Now $F(P)_{m,-}$ is the exit-path category for the stratification on $P_{m,-}$ by the faces, hence also contractible.
    We conclude that $\Sigma_{m,-}^\op$ is also contractible as desired.
\end{proof}


\begin{rem}
    For the fan corresponding to $\mathbb{P}^n$, one can give a slick proof by noting that the limit diagram for $\Sigma^\op$ is the same as the \v Cech diagram for the open cover of $\MR$ as a topological space by $\{\sigma^{\svee,\circ}\rightarrow M_\RR : \sigma\in\Sigma(1)\}$ and use \HA{Proposition}{1.2.4.13}. However, it is not true in general that the diagram as above is the \v Cech diagram for an open cover of $\MR$. Therefore, we opt for a different proof as above instead.
\end{rem}
\begin{cor}
    For a smooth projective fan $\Sigma$, the functor $\Psi_\Sigma$ assembled in \cref{construction of combinatorial to constructible comparison functor} is symmetric monoidal.
\end{cor}
\begin{proof}
    Recall from \cref{construction of combinatorial to constructible comparison functor} that 
    $\Psi_\Sigma$ is a composition:
    \[\Psi_\Sigma:\lim_{\Sigma^\op}\Fun(\Theta(\sigma)^\op,\Sp)\overset{\lim\Psi_\sigma}{\longrightarrow}\lim_{\Sigma^\op}\Mod_{\omega_{\sigma^\svee}}\Shv(\MR;\Sp)\longrightarrow\Shv(\MR;\Sp).\]
    The first functor is always symmetric monoidal, and we are concerned with the second functor. Note that it is defined as a right adjoint to the functor
    \[\Shv(\MR;\Sp)\longrightarrow\lim_{\Sigma^\op}\Mod_{\omega_{\sigma^\svee}}\Shv(\MR;\Sp),\]
    which is an equivalence when the fan is smooth and projective, given \cref{the idempotent algebras for a smooth projective fan glues to the unit}. So we conclude that the second functor is also symmetric monoidal, and so is $\Psi_\Sigma$.
\end{proof}
\begin{rem}
\label{Remark on idempotent algebra and log perfectoid quasicoherent of vaintrob}
    More generally, the result of Dmitry Vaintrob in \cite{Vaintrob2} could be interpreted to suggest that the limit of the family of idempotent algebras in $\Shv(\MR;\Sp)$ as in \cref{the idempotent algebras for a smooth projective fan glues to the unit} should only depend on the support, but not a particular fan. This is closely related to his construction \cite{vaintrob2017categoricallogarithmichodgetheory} of log quasi-coherent category of toroidal compactifications.
    A direct adaptation of the construction of our comparison functor to Dmitry Vaintrob's setting will produce a symmetric monoidal equivalence between the category of `almost' quasi-coherent sheaves on smooth projective toric schemes and the categories of sheaves of spectra on real vector spaces \stress{without} constructibility constraints.
\end{rem}

\section{Singular support}
The aim of this section is to characterize $\im(\kappa)$ for a \stress{smooth projective fan} $\Sigma$ in terms of singular support. The idea of using singular support to describe the image of $\kappa$ was originally due to \cite{FLTZ}. Let $\Lambda_\Sigma$ be the conic Lagrangian subset of the cotangent bundle $T^*M_\RR$ given in \cref{definition of FLTZ skeleton}. We will define a full subcategory of $\Shv(M_\RR;\Sp)$ spanned by sheaves with singular support contained in $\Lambda_\Sigma$. It follows directly from the definition that the functor $\kappa$ factors through $\Shv_{\Lambda_\Sigma}(\MR;\Sp)$, so
$$\Shv_{\Lambda_\Sigma}(M_\RR;\Sp)\supseteq\im(\kappa).$$
Then we follow the idea of \cite{zhou2017twistedpolytopesheavescoherentconstructible} to show that the inclusion is an equality. The benefit of our approach is that along the way we will construct an explicit family of compact generators of $\Shv_{\Lambda_\Sigma}(M_\RR;\Sp)$.
\par We will first take a quick tour of the theory of singular support for polyhedral sheaves. This is particularly simple, since locally we are working with conic sheaves on a real vector space. Then we revisit the interplay between twisted polytopes and sheaves. Finally we invoke the non-characteristic deformation lemma from \cite{robalo2016lemmamicrolocalsheaftheory} to prove the result.

The reason why our proof is less straightforward as opposed to what appears in
\cite{FLTZ,zhou2017twistedpolytopesheavescoherentconstructible} is the following.
We find that there is a lack of a general theory of singular supports for sheaves of spectra, so that the arguments one can make in its classical counterpart \cite{kashiwara2002sheaves} would carry over without much modification (see, however, \cite{Jin_2024} for an exposition in this direction.)
We hope that this section serves as an invitation to homotopy theorists to revisit the notion of singular supports in greater generality and to investigate questions like \cref{question about compact generation of singular support}.
\subsection{Singular support for polyhedral sheaves}
Following \cite[Section 4]{FLTZ}, we define the notion of singular supports for \stress{polyhedrally} constructible sheaves on real vector spaces (and also tori). `Polyhedral' means that we fix a stratification $P$ on a real vector space $V$, specified (as in \cref{definition of FLTZ stratification}) by an affine hyperplane arrangement. We will consider sheaves which are constructible for such `polyhedral' stratification. Locally, these sheaves are modeled on conic sheaves $F$ on a real vector space $V$, which we will first study. All vector spaces appearing here will be \stress{finite dimensional}.
\begin{rem}
We will make use of results in  \cite{kashiwara2002sheaves}, but the reader should be warned that the book was written in the classical language of bounded derived category of sheaves. So it is not directly applicable in our situation. However, the results we make use of could be verified with the same proof from there. We will revisit these facts in future work.
\end{rem}
\begin{defn}
    Recall that the topological group $\RR_+=\{r\in \RR:r>0\}$ acts continuously on a real vector space $V$ via multiplication. We define the \stress{ category of conic sheaves} on $V$ to be the full subcategory of sheaves that are constant when restricted to each orbit, and write it as 
    $$\Shv^\mathrm{conic}(V;\Sp)\subseteq\Shv(V;\Sp).$$
\end{defn}
\begin{defn}[Fourier-Sato transform]
    Let $V$ be a  real vector space with dual $V^*$. The \stress{Fourier-Sato transform} is defined to be
    $$\mathcal{FS}:\conicshv(V;\Sp)\longrightarrow\conicshv(V^*;\Sp)$$
    $$F\mapsto p_!q^*F$$
    where $p:K\rightarrow V^*$ and $q:K\rightarrow V$ are projections from the kernel
    $$K\defeq\{(x,y)\in V\times V^*:\langle x,y \rangle \leq 0\}\subset V\times V^*.$$
    We define the \stress{singular support at the origin} of a conic sheaf $F$ to be the support (i.e. closure of the points where the stalk doesn't vanish) of $\mathcal{FS}(F)\subseteq V^*$, where we identify $V^*$ with the cotangent space of $V$ at the origin
    $$\musupp_0(F)\defeq\supp(\mathcal{FS}(F))\subset V^*\simeq T^*_0(V).$$
    
\end{defn}
\begin{prop}[{\cite[Theorem 3.7.9]{kashiwara2002sheaves}}]
    The Fourier-Sato transform is an equivalence of categories between conic sheaves on $V$ and conic sheaves on $V^*$ :
    $$\mathcal{FS}:\conicshv(V;\Sp)\overset{\simeq}{\longrightarrow}\conicshv(V^*;\Sp).$$
    
\end{prop}
\begin{rem}[An alternative definition]
    One can also define a notion of singular supports using a Morse-type construction as in \cite[Definition 4.5]{robalo2016lemmamicrolocalsheaftheory}. It coincides with this definition, but we will not use it here.
\end{rem}
One particular feature of such definition we will use is that it interacts nicely with cones.
\begin{lem}[{\cite[Lemma 3.7.10]{kashiwara2002sheaves}}]
\label{fourier-sato of conic sheaf of an open cone}
    Let $V$ be a real vector space with $V^*$ its dual. Let $\tau\subseteq V$ be an open convex cone and $-\tau^\svee\subseteq V^*$ be the negative of its dual cone. Then
    $$\mathcal{FS}(\omega_\tau)=\underline\SS_{-\tau^\svee}.$$
    In particular, the singular support at the origin of $\omega_\tau$ is 
    $$\musupp(\omega_\tau)_0=-\tau^\svee.$$
\end{lem}
Now we globalize the previous definition:
\begin{defn}[Singular support] 
\label{definition of singular support}
Let $V$ be a vector space equipped with the stratification $P$ specified by an affine hyperplane arrangement as in \cref{definition of FLTZ stratification}. For a constructible sheaf $F\in\Cons_{P}(V;\Sp)$, one can specify a subset of the cotangent bundle of $V$,
$$\musupp(F)\subseteq T^*V\simeq V\times V^*$$
to be the (global) \stress{ singular support of $F$}. Its fiber at a point $v\in V$, denoted by $\musupp_v(F)$ is determined as follows: pick an open ball $U$ centered at $v$ that only meets the hyperplanes passing through $v$. Pick an exponential map from the tangent space:
$$\mathrm{exp}:V\overset{\simeq}{\longrightarrow}U$$
and it pulls $F$ back to a conic sheaf $\mathrm{exp}^*F\in\conicshv(V;\Sp)$. We define
$\musupp(F)_v\defeq\musupp_0(\mathrm{exp}^*F)\subseteq V^*$
and we identify canonically $V^*$ with $T^*_vV$.
\end{defn}
\begin{rem}[Singular support is well-defined]
    We remark that at each point $v$ the subset $\musupp_v(F)$ doesn't depend on the choice of the open ball $U$ nor the exponential map $\exp$. To compare different choices we end up with a transition map $$V\rightarrow V$$ which is given by multiplication of a continuous function valued in $\RR_+$ on $V$. Since all the orbits are contractible and the sheaf involved is conic, one can produce an equivalence between sheaves $\exp^*(F)$ under different choices. We don't spell out the details here.
\end{rem}
\begin{defn}[Sheaves with prescribed singular support] Following the notation as \cref{definition of singular support}. Let $\Lambda\subset T^*V\simeq V\times V^*$ be a subset. We define a full subcategory $\Shv_\Lambda(V;\Sp)$ of $\Cons_{P}(V;\Sp)$ to be
$$\Shv_\Lambda(V;\Sp)\defeq\{F:\musupp(F)\subseteq \Lambda\}.$$
This is the subcategory of \stress{$P$-constructible sheaves with singular support contained in $\Lambda$}. 
\end{defn}
\begin{war}
\label{dependence of singular support on stratification}
    Note that the notation didn't make explicit the dependence on $P$, but we  always fix such a stratification and work inside the full subcategory of $P$-constructible sheaves.
    This should not cause confusion as we will work with a single fixed stratification at a time.
    It is true that $\musupp(F)$ doesn't depend on the ambient stratification - and in fact one can define singular support of a sheaf without the help of constructibility and arrive at the same notion. But beware that, given $\Lambda$, the category of $P$-constructible sheaves with singular support contained in $\Lambda$ can vary as $P$ changes. It is also true that they will be the same as long as conormal variety of $P$ contains $\Lambda$. We will not prove these facts nor use them.
\end{war}
\begin{var}
    The definition makes sense also for a quotient of a vector space by a lattice $V/\Gamma$, in particular for tori $\RR^n/\ZZ^n$: fix a polyhedral stratification $P$ on $V/\Gamma$ and a constructible sheaf $F$ for $(V/\Gamma,P)$, one can define a subset $\musupp(F)\subseteq T^*V/\Gamma$, and thus talk about the subcategory of $P$-constructible sheaves with prescribed singular support. We will make use of this notion in the final section.
\end{var}
Then we make several quick observations with the definition.
\begin{rem}[Locality]
\label{the notion of singular support is local}
The definition is local in nature. This in particular implies that one can check if a constructible sheaf $F$ on $V/\Gamma$ has the prescribed singular support by pulling back and checking on $V$, since the projection map is a local homeomorphism preserving the linear structure.
\end{rem}
\begin{rem}[Closed under colimits]
\label{sheaves with singular support is closed under colimit}
Given a polyhedral stratification $P$ on $V$ and a subset $\Lambda$ in $T^*V$. The subcategory $\Shv_\Lambda(V;\Sp)$ is a stable subcategory closed under colimits in $\Cons_P(V;\Sp)$ and hence also in $\Shv(V;\Sp)$. This follows from that the $*$-pullback functor and the Fourier-Sato transform preserve colimits, and that support condition is closed under colimits.
\end{rem}
The most important example of the computation with global singular support is the following:
\begin{lem} 
\label{global singular support of a omega sheaf}
\cite[Proposition 5.1]{FLTZ} Let $\Sigma$ be a smooth projective fan and consider the category of $\mathcal{S}_\Sigma$-constructible sheaves. We can bound the singular support of the sheaf $\omega_{m+\sigma^{\svee}}\in\Cons_{\Sc_\Sigma}(\MR;\Sp)$ for $\sigma\in\Sigma$:
$$\musupp(\omega_{m+\sigma^{\svee}})\subseteq\bigsqcup_{\tau\subset\sigma}m+\tau^\perp\times -\tau\subset M_\RR\times N_\RR\simeq T^*M_\RR.$$
\end{lem}
We refer to the original treatment for the proof: it is a direct application of \cref{fourier-sato of conic sheaf of an open cone}. \par
One feature of the notion of singular support is that it supports Morse theory. In our context, the foundational \stress{non-characteristic deformation lemma} is supplied by \cite[Theorem 4.1]{robalo2016lemmamicrolocalsheaftheory}:
\begin{prop}
\label{non characteristic deformation}
     Let $M\in\LCH$ and $F\in \Shv^\mathrm{hyp}(M;\Sp)$ be hypercomplete. Let $\{U_s\}_{s\in\RR}$ be a family of open subsets of $M$. Assume:
     \begin{enumerate}
         \item For all $t\in\RR$, $U_t=\cup_{s<t}U_s$.
         \item For all pairs $s\leq t$, the set $\overline{U_t\setminus U_s}\cap \supp(F)$ is compact.
         \item Setting $Z_s\defeq \cap_{t>s}\overline{U_t\setminus U_s}$, we have for all pairs $s\leq t$ and all $x\in Z_s$:
         $$i^!(F)_x=0$$ where $i:X\setminus U_t\rightarrow X$ is the inclusion. Note that by the recollement sequence where $j:U_t\rightarrow X$ is the inclusion
         $$i_!i^!(F)\longrightarrow F\longrightarrow j_*j^*(F)$$
         this is the same as asking $F_x\rightarrow j_*j^*(F)_x$ be an isomorphism for each $x\in Z_s$.
     \end{enumerate}
     Then we have for all $t\in \RR$:
     $$F(\bigcup_s U_s)\overset{\simeq}{\longrightarrow}F(U_t).$$
\end{prop}
\begin{rem}
    As we will be working with a finite dimensional real vector space, every sheaf is automatically hypercomplete. Beware that it is crucial that the coefficient category $\Sp$ is compactly generated presentable - otherwise one needs to change the definition of singular support. See \cite[Remark 4.24]{efimov2024ktheorylocalizinginvariantslarge}.
\end{rem}
Having prepared ourselves with enough abstract nonsense, here we present the crucial part of this subsection: we will provide a refinement of the category of $\Sc_\Sigma$-constructible sheaves such that the image of $\kappa$ lies in it:
\begin{defn}[FLTZ skeleton]
\label{definition of FLTZ skeleton}\footnote{The name `FLTZ skeleton' is borrowed from symplectic geometry.}
    Let $\Sigma$ be a smooth projective fan. We define a conic Lagrangian subset of $T^*M$ as follows:
    $$\Lambda_\Sigma\defeq\bigsqcup_{m\in M,\sigma\in\Sigma} m+\sigma^\svee\times-\sigma\subseteq M_\RR\times N_\RR\simeq T^*M_\RR.$$
    From now on we will focus on  the category $\Shv_{\Lambda_\Sigma}(M_\RR;\Sp)$ of \stress{$\Sc_\Sigma$-constructible sheaves with singular support in $\Lambda_\Sigma$}.
\end{defn}
\begin{lem}
The image of $\kappa$ lies in $\Shv_{\Lambda_\Sigma}(M_\RR;\Sp)$.
\end{lem}
\begin{proof}
    The category $\im(\kappa)$ is generated as a stable category under colimit  by the objects of the form $\omega_{m+\sigma^{\svee}}$, and each of them has singular support contained in $\Lambda_\Sigma$ by \cref{global singular support of a omega sheaf}. By \cref{sheaves with singular support is closed under colimit}, the category of sheaves with prescribed singular support is closed under colimits, so the proof is done.
\end{proof}
\subsection{Combinatorics of smooth projective fan}
\label{subsection:combinatorics of smooth projective fan}
One distinguishing feature of a \stress{smooth projective} fan $\Sigma$ in $N_\RR$ is that it can be presented as the dual fan of an integral polytope $P$. See \cite[Section 1.5]{fulton} for the construction. Such polytope $P$ has the following properties:
\begin{enumerate}
    \item The Minkowski sum of $P$ with any dual cone of $\sigma\in\Sigma$ is an integral translation of the dual cone of $\sigma$.
    \item Each of the dual cone $\sigma^\svee$ can be written as an increasing union of translations of polytopes of the form $n\cdot P$, where each $n\cdot P$ is an integral multiple of the polytope $P$.
\end{enumerate}
We will see that these properties imply that after fixing one such $P$, the objects $\{\omega_{m+n\cdot P}\}$ for varying $n$ and translation along $m\in M$ form an explicit collection of compact generators for $\im(\kappa)$. On the mirror side, this is reminiscent of the familiar fact from algebraic geometry: tensor powers of ample line bundles generate the category of quasi-coherent sheaves under colimits. \par We will explain how the association $P\mapsto\omega_P$ generalizes to a bigger collection of combinatorial objects, namely, \stress{twisted polytopes}.  To start with, we will make use of the following description of $\im(\kappa)$.
\begin{prop} 
\label{charaterization of image of kappa}
The category $\im(\kappa)$ enjoys the following properties and characterizations.
\begin{enumerate}
    \item The category $\im(\kappa)$ is closed under colimits and shifts in $\Shv(M_\RR;\Sp)$.
    \item The category $\im(\kappa)$ can be characterized explicitly as
    $$\{\Fc\in\Shv(M_\RR;\Sp):\Fc*\omega_{\sigma^\svee}\in\langle\omega_{m+\sigma^\svee}:m\in M\rangle\}.$$
    \item The category $\im(\kappa)$ is generated under colimits and shifts of the following collection of objects:
    $$\{\omega_{m+\sigma^\svee}:\sigma\in \Sigma,m\in M\}.$$
    \item The category $\im(\kappa)$ is closed under convolution products in $\Shv(M_\RR;\Sp)$.
\end{enumerate}
\end{prop}
\begin{proof} The first point comes from the fact that $\kappa$ is a fully faithful, colimit-preserving functor from a presentable stable category, as $\kappa$ is constructed from taking limit of a diagram in $\Prl$. The second point follows directly from the limit description of $\kappa$. For the third point, using descent along idempotent algebras, every object $X\in\Shv(M_\RR;\Sp)$ is a finite limit of a diagram, whose terms are of the form $X*\omega_{\sigma^\svee}$. Each of them lies in the category spanned by $\omega_{m+\sigma^\svee}$ by point two, so $X$ also lies in the category spanned by $\omega_{m+\sigma^\svee}$ as desired. Finally, since $\kappa$ is symmetric monoidal, its image is closed under tensor products.
\end{proof}
With this knowledge at hand, let's try to write down some objects in the category $\im(\kappa)$.
\begin{prop}    
\label{existence of moment polytopes}
    Let $\Sigma$ be a smooth projective fan, there exist (in fact, many) polytopes $P$ in $M_\RR$ with integral vertices such that $\Sigma$ can be realized as the dual fan of $P$. Coversely $P$ might be called a \stress{moment polytope} of $\Sigma$ (actually, associated to some line bundle).  More precisely, $P$ has the following properties:
    \begin{itemize}
    \item The Minkowski sum of $P$ with any dual cone $\sigma^\svee$ of $\sigma\in\Sigma$ is an integral translation of the dual cone of $\sigma$. Concretely, this says that for each $\sigma\in \Sigma$, there exists some $m\in M$ such that 
    $$P+\sigma^\svee=m+\sigma^\svee.$$
    \item Each of the dual cone $\sigma^\svee$ can be written as an increasing union of integral translations of polytopes of the form $n\cdot P$, where each $n\cdot P$ is an integral multiple of the polytope $P$. Concretely this says that for each $\sigma\in \Sigma$, one can pick a collection of $m_i\in M$ and form a increasing union
    $$\bigcup_{i\geq 0}m_i+i\cdot P=\sigma^\svee.$$
\end{itemize}
\end{prop}
\begin{proof}
    The existence of such a polytope is assumed by the definition of a projective fan (see also \cite[Section 1.5]{fulton}). Both claims about the polytopes are direct combinatorics and we omit the proof. 
\end{proof}

We will consider the object $\omega_P\in \Shv(M_\RR;\Sp)$. 

\begin{prop}
    For such a moment polytope $P$ as above:
    \begin{enumerate}
        \item The object $\omega_P$ lies in $\im(\kappa)$.
        \item The object $\omega_P$ is a compact object in $\Cons_{\Sc_\Sigma}(M_\RR;\Sp)$ and hence also compact in $\im(\kappa)$.
        \item The same is true for $\omega_{m+n\cdot P}$ for each $m\in M$ and $n\in \ZZ_{>0}$. Moreover, these objects compactly generate the category $\im(\kappa)$.
    \end{enumerate}
\end{prop}
\begin{proof}
The first point comes from the characterization of $\im(\kappa)$ in point (2) of \cref{charaterization of image of kappa}. We can compute $$\omega_P*\omega_{\sigma^\svee}\simeq\omega_{P+\sigma^\svee}\simeq\omega_{m+\omega_{\sigma^\svee}}$$  using that  $P$ is a moment polytope. The second point comes from an application of the exodromy equivalence and the description of compact objects in the functor categories by \cref{characterization of compact in functor category on poset}, using that such polytope $P$ is assumed to be bounded. For the final point, one can write each $\sigma^\svee$ as an increasing union of polytopes of the form $m+n\cdot P$. It follows that there is a filtered colimit presentation
$$\colim_{m+n\cdot P\subseteq \sigma^\svee}\omega_{m+n\cdot P}\simeq \omega_{\sigma^\svee}.$$
Up to translation, this shows that every $\omega_{m+\sigma^\svee}$ can be written as a colimit of $\omega_{m+n\cdot P}$. Hence $\im(\kappa)$ is generated by $\omega_{m+n\cdot P}$ for varying $m\in M$ and $n>0$.
\end{proof}
\begin{rem}[Divisors and piecewise linear functions]
\label{divisors and piecewise linear functions}
    Here we give two more combinatorial ways to present the data of such polytope $P$. Firstly as `divisors': the polytope $P$ is the intersection of several half-spaces in $M_\RR$, indexed by the $1$-cones $\eta\in\Sigma(1)$. Let us fix the \stress{ primitive integral vectors} $v_\eta\in N$ for each $\eta\in\Sigma(1)$, then we can write
    $$P=\bigcap_{\eta\in\Sigma(1)}\{m\in M_\RR:\langle m,v_\eta\rangle\geq -n_\eta\in \ZZ\}.$$
    Hence we can recover the polytope $P$ from the collection of integers $\{n_\eta:\eta\in\Sigma(1)\}$. More generally by a \stress{divisor} we would mean such a sequence of integers $\{n_\eta:\eta\in\Sigma(1)\}$ and we write $D$ for a divisor. In case of a moment polytope $P$ we write $D_P$ for the associated divisor as above. Note that one can make sense of the addition of divisors as componentwise addition.\par
    Secondly as \stress{piecewise linear functions}: given a divisor $D_P=\{n_\eta:\eta\in\Sigma(1)\}$ coming from a moment polytope $P$, one may extend the assignment $\nu_\eta\mapsto -n_\eta$  $\RR$-linearly on each cone to obtain an $\RR$-valued function $f_P$ on $N_\RR$ (here we use that the fan is smooth and projective). For each top dimensional cone $\sigma$, there is a unique $m_\sigma\in M$ such that when restricted to $\sigma$
    $$\langle m_\sigma,-\rangle = f_P(-)_{|_{\sigma}}.$$
    Such $\{m_\sigma\}$ is precisely the collection of vertices of $P$, see \cite[Section 3.4]{fulton}. So one might recover the polytope $P$ from the data of $f_P$. This is part of the beautiful connection between line bundles, divisors and piecewise linear functions, as explained in Fulton's book.
\end{rem}
\begin{var}[Twisted polytopes]
\label{assign sheaf to twisted polytopes}
It is not true that every divisor $D=\{n_\eta\}$ or every integral piecewise linear function $f$ corresponds to a polytope. However, one can still write down an object in $\im(\kappa)$ starting from such data. Let us explain the idea here: fix a collection of integers $\{n_\eta\}$ as a divisor $D$. We may construct an  piecewise linear integral function $f$ on $\MR$ in the same way as above. This piecewise linear integral function $f$ determines and is determined by a collection of elements $\{m_\sigma\in M:\sigma\in\Sigma(n)\}$. We might consider the collection of closed subsets
$$\{m_\sigma+\sigma^\svee\subseteq M_\RR:\sigma\in\Sigma(n)\}.$$
The fact that $m_\sigma$ and $m_\tau$ agrees as functions on $\sigma\cap\tau$ (as they both agree with $f$) implies that 
$$m_\sigma+(\sigma\cap\tau)^\svee=m_\tau+(\sigma\cap\tau)^\svee.$$
In fact, the function $f$ (or the divisor $D$) determines an integral element 
$$m_\sigma\in M/\sigma^\perp$$
for each $\sigma\in \Sigma$. Thus the subset
$$m_\sigma+\sigma^\svee\subseteq M_\RR$$
is well defined. Note that, by definition, we have
$$(m_\sigma+\sigma^\svee) +\tau^\svee=m_\tau+\tau^\svee$$
for $\tau\subseteq\sigma$.
We claim that the collection of objects
$$\{\omega_{m_{\sigma}+\sigma^\svee}\in \Mod_{\omega_{\sigma^\svee}}:\sigma\in\Sigma\}$$
determines an object in $\im(\kappa)$ using descent along idempotent algebras, as described in \cref{subsection on idempotent descent}. In other words, we claim that there exists an object $\omega(D)\in \im(\kappa)$ such that, functorially in $\sigma$,
$$\omega(D)*\omega_{\sigma^\svee}\simeq\omega_{m_\sigma+\sigma^\svee}.$$
To do so, it requires to provide isomorphisms
$$\omega_{m_\sigma+\sigma^\svee}*\omega_{\tau^\svee}\overset{\simeq}{\longrightarrow}\omega_{m_\tau+\tau^\svee}$$
for $\tau\subset\sigma$, and the homotopies between compositions and so on. The existence of such isomorphisms should follow from the equality
$$(m_\sigma+\sigma^\svee) +\tau^\svee=m_\tau+\tau^\svee$$
above. Formally, to construct $\omega(D)$, one can apply the functor $\Gamma_{M_\RR}$ to the collection of subsets $\{m_\sigma+\sigma^\svee\}$ with inclusions between them. We leave the details of the proof to the readers.\par If the divisor $D$ corresponds to an actual polytope $P$, this construction  recovers $\omega_P$. We  call a divisor \stress{twisted polytope} as it needs not to come from a polytope and the assignment $D\mapsto\omega(D)$ generalizes $P\mapsto\omega_P$. 
\end{var}
\begin{rem}
\label{assign sheaf to divisor is additive}
    The passage from moment polytopes $P$ to  divisors $D_P$ is additive in the sense that it takes Minkowski sum of moment polytopes to componentwise addition of divisors. In a similar fashion, the passage from divisors $D$ to sheaves $\omega_D$ is additive: it takes componentwise addition of divisors to convolution product of sheaves
    $$\omega(D_1+D_2)\simeq\omega(D_1)*\omega(D_2).$$
 This can be checked after convolution with each $\omega_{\sigma^\svee}$: one has
    $$\omega_{m_1+\sigma^\svee}*\omega_{m_2+\sigma^\svee}\simeq\omega_{m_1+m_2+\sigma^\svee}.$$
    One can carefully prove that the assignment $D\mapsto\omega(D)$ is a symmetric monoidal functor, but we will not do so.
\end{rem}
\begin{rem}[Every divisor is dominated by an ample divisor]
\label{divisor+high power of moment polytope is moment polytope}
    Even though not every divisor $D$ comes from a polytope, it is true that  after adding a large multiple of a divisor $D_p$ coming from a polytope, the divisor $D+n\cdot D_P$ corresponds to a polytope. To see this, use the characterization of the divisors corresponding to a polytope as strictly convex functions, proved in \cite[Section 3.4]{fulton}. For algebraic geometers, this is similar to the fact that a line bundle will become ample after tensoring with a sufficiently ample (positive) line bundle. 
\end{rem}
\begin{var}[Sheaves and polytopes from $\RR$-coefficient divisors]
\label{real coefficient divisors}
    The assumption that a divisor $D=\{n_\eta\}$ is a collection of integers or a piecewise linear function $f_P$ is integral on each cone is not essential in the above discussion: one can write down objects in $\Shv(M_\RR;\Sp)$ from the data of an $\RR$-coefficient `divisor' $\{r_\eta\}$, or equivalently, a piecewise linear function $f$ on $N_\RR$. We leave the details to the reader as we will not use them in our exposition.
\end{var}
\subsection{Microlocal characterization of the image of \texorpdfstring{$\kappa$}{κ}}
\label{subsection: microlocal characterization of image}
In this section we prove the promised characterization of $\im(\kappa)$. Before presenting the details of the proof, we briefly explain the proof idea here. We are going to show that the right orthogonal of the image $\im(\kappa)$ in $\Shv_{\Lambda_\Sigma}(M_\RR;\Sp)$ is zero. This  follows from the following explicit construction: for each $x\in M_\RR$, we construct an object $\omega(D_x)\in\im(\kappa)$ in the image of $\kappa$, such that $\omega(D_X)$ corepresents the functor of taking stalk at $x$ in $\Shv_{\Lambda_\Sigma}(M_\RR;\Sp)$. Formally,
$$\map(\omega(D_x),\Fc)[n]\simeq\Fc_x$$
holds for all $\Fc\in \Shv_{\Lambda_\Sigma}(M_\RR;\Sp).$
It follows that an object in the right orthogonal of $\im(\kappa)$ will have vanishing stalks everywhere, so such object has to be zero. With the help of adjoint functor theorem, this proves that \[\im(\kappa)=\Shv_{\Lambda_\Sigma}(\MR;\Sp).\] To prove such a statement about $\omega(D_x)$, we will apply the non-characteristic deformation lemma, after convoluting with a large enough multiple of $\omega_P$, where $P$ is a moment polytope of $\Sigma$.\par
The use of convolution product in the proof introduces some complication - as we don't know a priori if the category $\Shv_{\Lambda_\Sigma}(M_\RR;\Sp)$ is closed under convolution product (we will prove it nonetheless is, a fortiori).
This is where our narrative diverges from \cite{zhou2017twistedpolytopesheavescoherentconstructible}: we get around this issue by introducing an intermediate category as in \cref{warning on not finishing the proof}. Note that in \cite{zhou2017twistedpolytopesheavescoherentconstructible} this complication was not explicitly addressed.
\begin{rem}
\label{question about compact generation of singular support}
As a consequence of the proof, we will show that the category 
$\Shv_{\Lambda_\Sigma}(M_\RR;\Sp)$ is compactly generated, and specify an explicit collection of generators.
Now, in general, for each conic Lagrangian $L$ in the cotangent bundle of a manifold $X$,
one can define (as in \cite{Jin_2024}) a category of sheaves (of spectra) with singular support lying inside $L$.
We are curious if said category is always compactly generated and if there is a natural procedure to pick out compact generators in that category.
Specifically, as the microlocal stalk functor is one profitable perspective offered by the microlocal analysis of sheaves,
we are curious if there is any natural way to write down corepresenting objects for microlocal stalk functor and compute the mapping spectra between them.
\end{rem}
We begin by defining the object $\omega(D_x)$ mentioned above.
\begin{defn} \cite[Definition 4.1]{zhou2017twistedpolytopesheavescoherentconstructible} For a point $x\in M_\RR$, we define the \stress{probing sheaf at $x$}
    $$\omega(D_x)\in\Shv(M_\RR;\Sp)$$
    to be the object associated to the divisor
    $$D_x=\{n_\eta(D_x)=\floor{-\langle x,\nu_\eta\rangle}+1:\eta\in\Sigma(1)\}$$
    via the construction of \cref{assign sheaf to twisted polytopes}. The integer $\floor{-\langle x,\nu_\eta\rangle}+1$ is the smallest integer strictly larger than $-\langle x,\nu_\eta\rangle$. Note that by \cref{charaterization of image of kappa}  we have $\omega(D_x)\in \im(\kappa)$.
\end{defn}
The naming comes from the following theorem, whose proof takes up the rest of the section:
\begin{thm}
\label{the theorem of Dx corepresents taking stalk on sheaf with singular support}
    For an arbitrary sheaf $\Fc\in\Shv_{\Lambda_\Sigma}(M_\RR;\Sp)$, there exists an isomorphism (which we will construct explicitly in the proof)
    $$\map(\omega(D_x),\Fc)[n]\overset{\simeq}{\longrightarrow}\Fc_x\in\Sp.$$
\end{thm}
Given this, one can look at the inclusion $\im(\kappa)\rightarrow\Shv_{\Lambda_\Sigma}(M_\RR;\Sp)$: the right orthogonal of $\im(\kappa)$ vanishes because any object in there would have vanishing stalks everywhere. Applying the adjoint functor theorem to the inclusion functor, one obtains a right adjoint $\Shv_{\Lambda_\Sigma}(M_\RR;\Sp)\rightarrow\im(\kappa)$ such that the composition with inclusion is identity on $\Shv_{\Lambda_\Sigma}(M_\RR;\Sp)$ - which proves that the inclusion is essentially surjective. We have obtained the following corollary.
\begin{cor}
\label{Microlocal characterization of the image}
    There is an identification of full subcategories in $\Shv(M_\RR;\Sp)$:
    $$\im(\kappa)=\Shv_{\Lambda_\Sigma}(M_\RR;\Sp).$$
\end{cor}
\begin{notation}
From now on we fix a moment polytope $P$ for $\Sigma$, and we assume that the origin is contained in the interior of $P$. The polytope $P$ is given by the combinatorial data of a divisor (\cref{divisors and piecewise linear functions}) as a collection of integers $\{n_\eta(P):\eta\in\Sigma(1)\}$. In particular, this gives a presentation of $P$ as the intersection of half spaces
\[P=\bigcap_{\eta\in\Sigma(1)}\{m:\langle m,\nu_\eta\rangle\geq-n_\eta(P)\}\]
(recall that $\nu$ is a fixed primitive element of $\eta$). Because the origin is in the interior of $P$, we also know that  \stress{
\[n_\eta(P)>0\]}for each $\eta$. Moreover, we fix a fundamental domain $W\subset \MR$ of $M_\RR/ M$ as follows. Pick a basis $\{m_i\}$ for the lattice $M$ and take the half-closed hypercube
$$W\defeq\{\Sigma_i r_i m_i:m_i\in M;r_i\in[0,1)\}.$$
By replacing $P$ with some large multiple $n\cdot P$ , we assume for each $x\in W$, the divisor
$$D_x+D_P$$
also corresponds to a moment polytope, which we call $P_x$. One can achieve this by observing that there are only finitely many different divisors $D_x$ for $x\in W$. For each fixed $D_x$ we can dominate it with a large multiple of $P$ by \cref{divisor+high power of moment polytope is moment polytope}.
\end{notation}

\begin{rem}
\label{translation to other points out of fundamental domain}
We will prove that $\omega(D_x)$ corepresents taking stalks at $x\in W$. The same statement for every point $x \in \MR$ will then follow. Indeed, observe that for $m\in M$,
$$\omega(D_{x+m})\simeq\omega_m*\omega(D_x)$$
while convolution with $\omega_m$ is just $!$-pushforward along translation by $m$. So we can translate other points into the fundamental domain and obtain the above statement for other points. Another way to see this is that, such $P$ as above actually dominates $D_x$ for all points $x$ so the proof carries through.
\end{rem}
Now we consider a family of polytopes deforming $P_x$.
\begin{defn}[Non-characteristic deformation of the probing sheaf]
Fix $x\in W$ and a small positive real number $0<\epsilon\ll 1$ so that
\[-\langle x,\nu_\eta\rangle+\epsilon\cdot n_\eta <\floor{-\langle x,\nu_\eta\rangle}+1\] 
for all $\eta\in\Sigma(1)$.
Consider the following increasing family of polytopes indexed by $s\in [0,1]$:
$$P_{x,s}\defeq s\cdot {\color{purple}P_x}+ (1-s) \cdot {\color{blue}(x+(1+\epsilon)\cdot P)}. $$
It grows from (when $s=0$) {\color{blue}$x+(1+\epsilon)\cdot P$} to (when $s=1$) {\color{purple}$P_x$}. Note the definition of $P_{x,s}$ depends on $\epsilon$ implicitly.
\end{defn}
We will apply the non-characteristic deformation lemma to this family. To do so, we start with an observation about its interaction with $\Lambda_\Sigma$.
\begin{lem}
\label{no integer coefficient in the family}
    Write $P_{x,s}$ as the intersection of half-planes 
    $$P_{x,s}=\bigcap_{\eta\in\Sigma(1)}\{m\in M_\RR:\langle m,v_\eta\rangle\geq -n_{\eta,x,s}\in \RR\}.$$
    Then for $s\in[0,1)$, none of the real numbers $-n_{\eta,x,s}$ will be an integer. (In terms of \cref{real coefficient divisors}, these $\{n_{\eta,x,s}\}$ give the real coeffient divisors for $P_{x,s}$.)
\end{lem}
\begin{proof}
    Since the passage from polytopes to divisors is linear, we might look at the two ends of the interpolation, and inspect the coefficients of the corresponding divisors
    $$n_{\eta,x,1}=n_{\eta}(P_x)=n_{\eta}(P)+\floor{-\langle x,v_\eta\rangle}+1,$$
    $$n_{\eta,x,0}=n_{\eta}(P)-\langle x,v_\eta\rangle+\epsilon \cdot n_\eta(P).$$
    As long as 
    \[-\langle x,\nu_\eta\rangle+\epsilon\cdot n_\eta <\floor{-\langle x,\nu_\eta\rangle}+1\] 
    for each $\eta\in\Sigma(1)$, there will be no integer between $n_{\eta,x,0}$ and $n_{\eta,x,1}$. The claim follows.
\end{proof}

\begin{lem}\cite[Lemma 3.13]{zhou2017twistedpolytopesheavescoherentconstructible}
\label{estimation of singular support with the family}
    Let $s\in[0,1)$. Let $y\in\partial P_{x,s}$ be on the boundary of the polytope $P_{x,s}$, then we have the following estimate on the fiber of $\Lambda_\Sigma$ at $y$:$$\Lambda_{\Sigma,y}\cap-\sigma(y)={0}\subseteq N_\RR\simeq T^*_y(M_\RR).$$
    Here $\sigma(y)\in\Sigma$ is the cone dual to the angle spanned by $P_{x,s}$ at $y$, formally determined as follows. The vectors $\{p-y:p\in P_{x,s}\}\subseteq M_\RR$ span a cone $\sigma(y)^\svee$ in $M_\RR$. The cone $\sigma(y)$ is defined to be the dual cone of $\sigma(y)^\svee$.
\end{lem}
\begin{proof}
    For each $s\in[0,1)$, $P_{x,s}$ is also a moment polytope (but with non-integral vertices, as it is a convex linear combination of moment polytopes), so $\sigma(y)\in \Sigma$. Suppose, for sake of contradiction, that there exists a non-zero cotangent vector
    \[0\neq u\in -\sigma(y)\cap \Lambda_{\Sigma,y},\]
    one can find some $\tau\in \Sigma$ and $m\in M$ such that 
    \[(y,u)\in m+\tau^\perp\times-\tau\]
    and thus $u\in -\sigma(y)\cap-\tau$. This implies $\sigma(y)\cap\tau\neq\{0\}$. Note that $\tau$ cannot be the origin, so it must contain some $1$-cone $\rho$. It follows that $\langle y,\nu_\rho\rangle=\langle m,\nu_\rho\rangle$ is an integer. On the other hand, $\rho\subseteq\sigma(y)$ is equivalent to \[\langle\nu_\rho,p-y\rangle\geq 0\] for any $p\in P_{x,s}$. This implies that $\nu_\rho$ attains its minimum at $y$ on $P_{x,s}$, which means $-n_{\rho,x,s}=\langle y,\nu_\rho\rangle$ is an integer. This contradicts \cref{no integer coefficient in the family}.
\end{proof}
With this we study the family of open polytopes given by the interiors $P^\circ_{x,s}$ for $s\in[0,1)$.
\begin{lem}
  Consider a sheaf $F\in \Shv_{\Lambda_\Sigma}(M_\RR;\Sp)$ and the family of open polytopes given by the \stress{interior} $P^\circ_{x,s}$ for $s\in(-1,1)\simeq\RR$, where we extend the original family over $[0,1)$ by constant to the left: $P_{x,s}\defeq P_{x,0}$ for $s<0$. Then the assumptions of the non-characteristic deformation lemma \cref{non characteristic deformation} are satisfied for the sheaf $F$ and the family of open subsets $U_s=P^\circ_{x,s}$. 
\end{lem}
\begin{proof}
The point 1 and 2 in the assumptions of \cref{non characteristic deformation} follow directly from the definition of $P^\circ_{x,s}$. Unwinding the final point, we see that $Z_s$ is empty for $s\in(-1,0)$ and $Z_s=\partial P_{x,s}$ for $s\in[0,1)$. Following the notations from \cref{non characteristic deformation}, we write 
\[i:\MR\setminus U_s\longrightarrow\MR \]
for the inclusion maps. The goal is to show $(i^! F)_y=0$ for every $y\in Z_s$. Applying the recollement sequence for the open-closed decomposition $\MR=U_s\cup\MR\setminus U_s$, this is equivalent to showing that the canonical map
$$F_y\rightarrow j_*j^*(F)_y$$
is an isomorphism for $y\in\partial P_{x,s}$ and $s\in[0,1)$, where $j:U_s=P^\circ_{x,s}\rightarrow\MR$ is the inclusion map.
Since the stalk at $y$ only depends on the sheaf locally, we might choose a small enough open ball $U$ centered at $y$ and pullback $F$ along an exponential map 
\[\mathrm{exp}:\MR\overset{\simeq}\longrightarrow U\]
as in \cref{definition of singular support}. This reduces the question to the situation of a sheaf $\Fc$ on a vector space $M_\RR$ constructible for a stratification by linear subspace (hence in particular, conic). The sheaf $\Fc$ has the same singular support at origin as $F$ at $y$. In this case we are asking if the comparison map of the stalks at origin is an isomorphism:
$$\Fc_0\rightarrow j_*j^*(\Fc)_0$$
where $j:\sigma^{\svee,\circ}(y)\rightarrow M_\RR$ is the inclusion of an open cone $\sigma^{\svee,\circ}(y)$ determined as in \cref{estimation of singular support with the family} (whose dual is named $\sigma(y)\subseteq N_\RR$). By stratified homotopy invariance \cite[Corollary 3.3]{clausenJansen2023reductiveborelserre} (or \cite[Corollary 3.7.3]{kashiwara2002sheaves}), we may identify this map with the restriction map
$$\Fc(M_\RR)\rightarrow\Fc(\sigma^{\svee,\circ}(y)).$$
Now we can apply the Fourier-Sato transform: the map becomes
$$\map(\mathcal{FS}(\underline\SS_{M_\RR}),\mathcal{FS}(\Fc))\longrightarrow\map(\mathcal{FS}(\underline\SS_{\sigma^{\svee,\circ}(y)}),\mathcal{FS}(\Fc)).$$
To show that it is an isomorphism, it suffices to show 
$$\map(\cofib(\mathcal{FS}(\underline\SS_{\sigma^{\svee,\circ}(y)})\rightarrow\mathcal{FS}(\underline\SS_{M_\RR})),\mathcal{FS}(\Fc))=0.$$
By \cref{fourier-sato of conic sheaf of an open cone}, we know that 
$$\mathcal{FS}(\underline\SS_{M_\RR})\simeq\underline{\SS}_{\{0\}}[-n]\in \Shv(N_\RR;\Sp)$$
$$\mathcal{FS}(\underline\SS_{\sigma^{\svee,\circ}(y)})\simeq\underline\SS_{-\sigma(y)}[-n]\in \Shv(N_\RR;\Sp),$$
and the map between them is induced by inclusion. It follows that we can identify the cofiber as
\[\cofib(\mathcal{FS}(\underline\SS_{\sigma^{\svee,\circ}(y)})\rightarrow\mathcal{FS}(\underline\SS_{M_\RR}))\simeq \cofib(\underline\SS_{-\sigma(y)}\longrightarrow\underline{\SS}_{\{0\}})[-n]\simeq h_!\underline\SS[1-n].\]
The map $h:-\sigma(y)\setminus \{0\}\rightarrow N_\RR$ is the inclusion.
Now the assumption on singular support $\musupp(\Fc)_0=\musupp(F)_y\subseteq\Lambda_{\Sigma,y}$ implies 
$$\supp(\mathcal{FS}(\Fc))\subseteq \Lambda_{\Sigma,y}\subseteq N_\RR.$$ Moreover, from \cref{estimation of singular support with the family} we learn that  $\supp(\mathcal{FS}(\Fc))\cap -\sigma(y)\subseteq\{0\}$. This implies that the map $h$ above factors through the open complement of support of $\mathcal{FS}(\Fc)$, thus we must have
$$\map(\cofib(\mathcal{FS}(\underline\SS_{\sigma^{\svee,\circ}(y)})\rightarrow\mathcal{FS}(\underline\SS_{M_\RR})),\mathcal{FS}(\Fc))=\map(h_!\underline\SS[1-n],\mathcal{FS}(\Fc))=0.$$
This concludes the proof.
\end{proof}
\begin{cor}
\label{consequence of noncharacteristic deformation}
    For $F\in\Shv_{\Lambda_\Sigma}(M_\RR;\Sp)$ and $\epsilon$ sufficiently small, the restriction map 
    $$F(P_x^\circ)\longrightarrow F(x+(1+\epsilon)\cdot P^\circ)$$
    is an isomorphism.
\end{cor}
We will use this to prove that $\omega(D_x)$ corepresents taking stalk at $x$. To do so, we first prove a statement slightly different from \cref{the theorem of Dx corepresents taking stalk on sheaf with singular support}.
\begin{prop}
Let $\Gc\in\Shv(M_\RR;\Sp)$. If
$$\Gc*\omega_P\in\Shv_{\Lambda_\Sigma}(M_\RR;\Sp),$$
then for sufficiently small $\epsilon$ and $x\in W$, we have
$$\Gc(x+\epsilon\cdot P^\circ)\overset{\simeq}{\longrightarrow}\map(\omega(D_x),\Gc)[n].$$
Taking colimit along $\epsilon\rightarrow 0$ provides an isomorphism
$$\Gc_x\overset{\simeq}{\longrightarrow}\map(\omega(D_x),\Gc)[n]$$
for $x\in W$. The same is true for all $x\in M_\RR$.
    
\end{prop}
\begin{proof}
    Given that $\omega_P$ is a convolution invertible object, we have
    $$\Gc(x+\epsilon\cdot P^\circ)\simeq\map(\underline\SS_{x+\epsilon\cdot P^\circ},\Gc)\overset{\simeq}{\longrightarrow}\map(\underline\SS_{x+\epsilon\cdot P^\circ}*\omega_P,\Gc*\omega_P)\simeq (\Gc*\omega_P)(x+(1+\epsilon)\cdot P^\circ).$$
    Now by the assumption that $\Gc*\omega_P$ lies in $\Shv_{\Lambda_\Sigma}(M_\RR;\Sp)$, we can apply \cref{consequence of noncharacteristic deformation} and learn that the restriction map 
    $$(\Gc*\omega_P)(x+(1+\epsilon)\cdot P^\circ)\overset\simeq\longleftarrow (\Gc*\omega_P)(P_x^\circ)$$
    is an isomorphism. Finally again using $\omega_P$ is convolution invertible, we have
    (recall that $P_x$ is associated to the divisor $D_x+D_P$)
    $$\map(\omega(D_x),\Gc)\overset\simeq\longrightarrow\map(\omega(D_x)*\omega_P,\Gc*\omega_P)\simeq\map(\omega_{P_x},\Gc*\omega_P)\simeq (\Gc*\omega_P(P_X^\circ))[-n].$$
    Putting the above equivalences together we arrive at
    $$\Gc(x+\epsilon\cdot P^\circ)\simeq\map(\omega(D_x),\Gc)[n].$$
     This isomorphism is compatible with the restriction maps along shrinking $\epsilon$, and hence we get
    $$\Gc_x\simeq\map(\omega(D_x),\Gc)[n]$$
    as promised, for $x\in W$.
    As explained in \cref{translation to other points out of fundamental domain}, the same result holds for any $x\in M_\RR$.
\end{proof}
\begin{war}
\label{warning on not finishing the proof}
    This does not conclude the proof of \cref{the theorem of Dx corepresents taking stalk on sheaf with singular support}: the missing point is that we don't know if $(-)*\omega_P$ preserves the subcategory $$\Shv_{\Lambda_\Sigma}(M_\RR;\Sp)\subseteq\Shv(M_\RR;\Sp).$$
\end{war}
To circumvent this problem, we consider the following subcategory of $\Shv(M_\RR;\Sp)$:
$$\Cc\defeq\{\Gc\in\Shv(M_\RR;\Sp):\Gc*\omega_P\in\Shv_{\Lambda_\Sigma}(M_\RR;\Sp)\}.$$
A quick observation is that, since $\im(\kappa)$ is contained in $\Shv_{\Lambda_\Sigma}(M_\RR;\Sp)$ and closed under convolutions, we have $\im(\kappa)\subseteq\Cc$. The above argument effectively shows the following.
\begin{prop}
    The functor of taking stalk at $x$ is corepresented by $\omega(D_x)$ (up to a shift) in $\Cc$.
\end{prop}
A second observation we will need is that the category $\Cc$ is closed under colimits and limits in $\Shv(M_\RR;\Sp)$, and in particular presentable (but we actually only need cocompleteness for our argument).
\begin{prop}
    The inclusion $\im(\kappa)\subseteq \Cc$ is an equality.
\end{prop}
\begin{proof}
Same as the argument immediately following \cref{the theorem of Dx corepresents taking stalk on sheaf with singular support}.
\end{proof}
The final observation we will use is that, since $\omega_P$ is a convolution-invertible object in $\Shv(M_\RR;\Sp)$, we have a functor
$$(-)*\omega_P^{-1}:\Shv_{\Lambda_\Sigma}(M_\RR;\Sp)\rightarrow\Cc.$$
Applying the above proposition, we learn that for each $\Fc\in\Shv(M_\RR;\Sp)$,
$$\Fc*\omega_P^{-1}\in\Cc=\im(\kappa).$$
However, now that $\im(\kappa)$ is closed under convolution, one learns that 
$$\Fc=\Fc*\omega_P^{-1}*\omega_P\in \im(\kappa).$$
Consequently, we have
$$\im(\kappa)=\Shv_{\Lambda_\Sigma}(M_\RR;\Sp),$$
and \cref{the theorem of Dx corepresents taking stalk on sheaf with singular support} follows easily (beware of the flip of logic here).
\section{Epilogue}
In the final section, we exploit the results developed thus far to derive some tangible ramifications.
Firstly, we apply the folklore method of de-equivariantization to obtain the `non-equivariant' version 
of the equivalence.
Next, as a concrete consequence, 
we provide a proof of Beilinson's equivalence for flat $\mathbb{P}^n$  
over $\SS$. More generally, we introduce a definition of the toric construction in an abstract 
setting and explain how the equivalence fits into this framework.
As an example, we demonstrate how this method recovers a family version of the 
equivalence as in \cite{hu2023coherentconstructiblecorrespondencetoricfibrations}. \par
Throughout the section we always work with a \stress{smooth projective} fan.
\subsection{De-equivariantization}
One of the most basic notions in the theory of stacks is that of quotient stacks.
Given a group object $G$ acting on $X\in\Stk$, one can form the quotient stack $[X/G]\in\Stk$. The fundamental insight is that it encodes all the $G$-equivariant information about $X$.
In this regard, $\QCoh([X/G])$ is just the category of objects in $\QCoh(X)$ together with a $G$-action, i.e., the category of $G$-modules in $\QCoh(X)$.
Therefore, $\QCoh([X/G])$ is completely determined by $\QCoh(X)$, along with the action of $G$ on $\QCoh(X)$.

This process of determining $F([X/G])$ from $F(X)$, together with the information of a $G$-action on $F(X)$,
is colloquially referred to as \stress{equivariantization},
where $F$ is a sheaf, with $F = \QCoh(-)$ in the previous example.

A less-exploited point of view, dubbed \stress{de-equivariantization},
allows us to sometimes go in the other direction.
Indeed, observe that there is a pullback diagram in $\Stk$
\[\begin{tikzcd}
	X & {[X/G]} \\
	\ast & BG
	\arrow[from=1-1, to=1-2]
	\arrow[from=1-1, to=2-1]
	\arrow[from=1-2, to=2-2]
	\arrow[from=2-1, to=2-2]
\end{tikzcd}.\]
With both $[X/G]$ and $BG$ being perfect,
$F = \QCoh(-)$ takes this pullback square to a pushout square
in $\CAlg(\Prl)$ and we have
\[
\QCoh(X) \simeq \QCoh([X/G]) \otimes_{\QCoh(\mathrm{BG})} \QCoh(\ast),
\]
where the relative tensor product is taken in $\Prl$ (see \cite[Proposition 4.6]{BFN2010}\SAG{Corollary}{9.4.2.3}).
Now we apply this method to the case that is interesting to us.
First, we need some preparations. 
\begin{lem}
    For each $\sigma\in\Sigma$, the stack $[X_\sigma/\TT]$ is a perfect stack in the sense of \SAG{Definition}{9.4.4.1}. Similarly, the stack $B\TT$ is also a perfect stack.
\end{lem}
\begin{proof}
We only present the proof for $[X_\sigma/\TT]$, the other case could be proved with similar arguments. We need to check three things:
\begin{itemize}
    \item The stack $[X_\sigma/\TT]$ is a quasi-geometric stack. It is in fact geometric. Given \SAG{Corollary}{9.3.1.4}, this follows (in the same way as \cite[Remark 2.1]{Moulinos}) from the fact that $[X_\sigma/\TT]$ is a colimit of the action diagram of $\TT$ acting on $X_\sigma$, where the degree $0$ term $X_\sigma$ is affine and the map $d_0:X_\sigma\times\TT\rightarrow X_\sigma$ is representable, affine and faithfully flat. 
    \item The structure sheaf $\Oc$ is a compact object in $\QCoh([X_\sigma/\TT])$. Via \cref{theorem of combinatorics compares to quasicoherent}, the structure sheaf is sent to a representable presheaf, which is certainly a compact object.
    \item The category $\QCoh([X_\sigma/\TT])$ is generated by compact objects. Via \cref{theorem of combinatorics compares to quasicoherent} this reduces to the fact that the spectral presheaf category is compactly generated.
\end{itemize}
\end{proof}
\begin{thm}[De-equivariantization for $\QCoh$]
\label{de-equivariantization for quasicoherent}
    De-equivariantization applies to the following stacks:
    \begin{itemize}
        \item For each $\sigma\in\Sigma$, we have a symmetric monoidal equivalence
        \[\QCoh([X_\sigma/\TT])\otimes_{\QCoh(B\TT)}\QCoh(*)\overset{\simeq}\longrightarrow\QCoh(X_\sigma).\]
        \item We have a symmetric monoidal equivalence
        \[\QCoh([X_\Sigma/\TT])\otimes_{\QCoh(B\TT)}\QCoh(*)\overset{\simeq}\longrightarrow\QCoh(X_\Sigma).\]
    \end{itemize}
\end{thm}
\begin{proof}
    The first point is a direct application of \SAG{Corollary}{9.4.2.3}, given that both $[X_\sigma/\TT]$ and $B\TT$ are perfect stacks. For the second point, note that by the colimit presentation of $[X_\Sigma/\TT]$ one has
    \[\QCoh([X_\Sigma/\TT])\simeq\lim_{\Sigma^\op}\QCoh([X_\sigma/\TT]).\]
    Hence the relative tensor product gives
\begin{align*}
\QCoh([X_\Sigma/\TT])\otimes_{\QCoh(B\TT)}\QCoh(*)
& \simeq (\lim_{\Sigma^\op}\QCoh([X_\sigma/\TT]))\otimes_{\QCoh(B\TT)}\QCoh(*)\\
& \simeq \lim_{\Sigma^\op}(\QCoh([X_\sigma/\TT])\otimes_{\QCoh(B\TT)}\QCoh(*))\\
& \simeq\lim_{\Sigma^\op}\QCoh(X_\sigma) \\
& \simeq \QCoh(X_\Sigma),
\end{align*}
where the relative tensor product commutes with limits since $\QCoh(*)$ is dualizable over $\QCoh(B\TT)$ by \SAG{Corollary}{9.4.2.2}.
\end{proof}
\begin{rem}
    It seems plausible to directly show that $[X_\Sigma/\TT]$ is a perfect stack by adapting the proof of \cite[Proposition 3.21]{BFN2010} in the spectral setting. We opt for the above proof because it is straightforward from what we have done so far.
\end{rem}
We move on to the mirror side. Note that for the category of sheaves on real vector spaces, the de-equivariantization is reflected as the equivariantization, as we explain now.
We need the following fact whose proof will be provided later.
\begin{lem}\label{D-shriek-is-symmetric-monoidal}
    The lax symmetric monoidal functor given in \cref{construction:convolution}  \[D_!(-):\LCH\longrightarrow\Cat\] lifts to a symmetric monoidal functor (by abuse of notation, we  give it the same name)
    \[D_!(-):\LCH\longrightarrow\Prlst.\]
\end{lem}

\begin{thm}[Equivariantization for $\Shv$]
\label{equivariantization for large sheaf category}
    There is a  commutative square of symmetric monoidal categories
\[\begin{tikzcd}
	{\Fun(M,\Sp)} & {\Shv(M_\RR;\Sp)} \\
	{\Shv(M;\Sp)} & {\Shv(M_\RR;\Sp)}
	\arrow[hook, from=1-1, to=1-2]
	\arrow["\simeq", from=1-1, to=2-1]
	\arrow["{=}", from=1-2, to=2-2]
	\arrow["{i_!}"', hook, from=2-1, to=2-2]
\end{tikzcd}. \]
    The symmetric monoidal functor
    \[\Fun(M,\Sp)\longrightarrow\Shv(M_\RR;\Sp)\]
    is from \cref{comb-constructible functor is compatible with lattice}. The map $i:M\rightarrow M_\RR$ is the inclusion of the topological groups, hence $!$-pushforward along $i$ induces a fully faithful symmetric monoidal functor
    \[i_!:\Shv(M;\Sp)\longrightarrow\Shv(\MR;\Sp)\]
    where both categories are equipped with the convolution monoidal structure.
    Moreover, $\Shv(M_\RR/M;\Sp)$ can be identified with the relative tensor product:
    \[
    \Shv(M_\RR/M;\Sp) \simeq \Shv(M_\RR;\Sp)\otimes_{\Fun(M,\Sp)}\Sp  \in\CAlg(\Prlst).
    \]
\end{thm}
\begin{proof}
    Recall from \cref{comb-constructible functor is compatible with lattice} that the functor 
    \[\Fun(M,\Sp)\longrightarrow\Shv(M_\RR;\Sp)\]
    is defined as a composition
    \[\Fun(M,\Sp)\longrightarrow\lim_{\sigma}\Fun(\Theta(\sigma),\Sp)\longrightarrow\Shv(M_\RR;\Sp).\]
    By \cref{the idempotent algebras for a smooth projective fan glues to the unit}, we know that it takes $m\in M$ to the skyscraper $\underline\SS_{\{m\}}\in \Shv(M_\RR;\Sp)$. Note  that it  preserves colimits, so its image is contained in the image of the fully faithful functor $i_!$. Hence we get a symmetric        monoidal factorization
    \[\Fun(M,\Sp)\longrightarrow\Shv(M;\Sp)\overset{i_!}\longrightarrow\Shv(\MR;\Sp),\]
    and the first functor is readily checked to be an equivalence.
    Next we study the relative tensor product
    \[
    \Shv(M_\RR;\Sp)\otimes_{\Fun(M,\Sp)}\Sp.
    \]
    Note that we may replace $\Fun(M,\Sp)$ by $\Shv(M;\Sp)$ and $\Sp$ by $\Shv(*;\Sp)$. Hence we might as well study the relative tensor product
    \[\Shv(M_\RR;\Sp)\otimes_{\Shv(M;\Sp)}\Shv(*;\Sp),\]
    formed along the symmetric monoidal functors
    \[i_!:\Shv(M;\Sp)\longrightarrow\Shv(M_\RR;\Sp)\]
    and
    \[p_!:\Shv(M;\Sp)\longrightarrow\Shv(*;\Sp).\]
    From \cref{D-shriek-is-symmetric-monoidal}, we have a symmetric monoidal functor
    \[D_!(-):\LCH\longrightarrow\Prlst\]
    and one can left Kan extend it to a symmetric monoidal colimit-preserving functor on the category of presheaves \footnote{Alternatively, one can left Kan extend to the category of \'etale sheaves and simplify some of the arguments below.} on $\LCH$ (ignoring size issues)
    \[D_!(-):\Fun(\LCH^\op,\Spc)\longrightarrow\Prlst.\]
    By \cref{Taking module category}, we know that $D_!(-)$ is compatible with forming relative tensor products. In particular, we get an identification 
    \[D_!(\MR)\otimes_{D_!(M)}D_!(*)\overset\simeq\longrightarrow D_!(h_\MR\times_{h_M}h_*)\in\CAlg(\Prlst)\]
    where the underlying object of the right-hand side is computed as the colimit of $D_!(-)$ applied to the Bar complex of the relative tensor product
    \[h_\MR\times_{h_M}h_*\in\Fun(\LCH^\op,\Spc).\]
    It remains to identify this colimit with $\Shv(\MR/M;\Sp)$ in $\CAlg(\Prl)$. We have a map
    \[h_\MR\times_{h_M}h_*\longrightarrow h_{\MR/M}\in\CAlg(\Fun(\LCH^\op,\Spc))\]
    and we claim it becomes an equivalence once we apply $D_!(-)$. This essentially follows from \'etale descent for the functor $\Shv(-)$. We supply a detailed explanation as follows: given the map, it suffices to show that after applying $D_!(-)$ one gets an equivalence of categories. We will identify the Bar complex computing the relative tensor product\[h_\MR\times_{h_M}h_*\in\Fun(\LCH^\op,\Spc)\]
    with the Yoneda image of \v Cech nerve of the covering map
    \[\MR\longrightarrow \MR/M\in\LCH.\]
    Note that we are comparing two simplicial diagrams sitting inside in a sub-1-category in $\Fun(\LCH^\op,\Spc)$. It is direct to check that these two diagrams agree. It follows that after applying $D_!(-)$ we have an identification of simplicial diagram of categories
    \[[n\mapsto\Shv(M;\Sp)^{\otimes n}\otimes\Shv(\MR;\Sp)]\simeq[n\mapsto\Shv( M^{\times n}\times\MR)],\]
    where the colimit of the left-hand side by definition computes $D_!(h_\MR\times_{h_M}h_*)$. Now the question is reduced to checking that
\[\begin{tikzcd}
	\cdots & {\Shv(M\times M\times M_\RR;\Sp)} & {\Shv(M\times M_\RR;\Sp)} & {\Shv(M_\RR;\Sp)} & {\Shv(M_\RR/M;\Sp)}
	\arrow[shift left=3, from=1-1, to=1-2]
	\arrow[shift right=3, from=1-1, to=1-2]
	\arrow[shift right, from=1-1, to=1-2]
	\arrow[shift left, from=1-1, to=1-2]
	\arrow[shift right=2, from=1-2, to=1-3]
	\arrow[shift left=2, from=1-2, to=1-3]
	\arrow[from=1-2, to=1-3]
	\arrow[shift left, from=1-3, to=1-4]
	\arrow[shift right, from=1-3, to=1-4]
	\arrow[from=1-4, to=1-5]
\end{tikzcd}\]
    is a colimit diagram in $\Prl$, where all the arrows are given by $!$-pushforward. Equivalently, we can take right adjoints everywhere and check that the outcome is a limit diagram in $\Cat$. Note that all the non-degenerate maps are \'etale so $!$-pullback is canonically identified with $*$-pullback. Using that taking $\Shv(-)$ with $*$-pullback has \'etale descent, it follows that
    \[D_!(h_\MR\times_{h_M}h_*)\overset{\simeq}{\longrightarrow} D_!(h_{\MR/M})\in\CAlg(\Prlst).\]
    Thus we have a symmetric monoidal equivalence
    \[\Shv(M_\RR;\Sp)\otimes_{\Fun(M,\Sp)}\Sp\simeq\Shv(M_\RR/M;\Sp) \in\CAlg(\Prlst).\qedhere\]
\end{proof}

\begin{rem} More informally and concretely, we can interpret the above argument as follows. To compute the tensor product
    \[\Shv(M_\RR;\Sp)\otimes_{\Shv(M;\Sp)}\Shv(*;\Sp), \]
    one can look at the colimit of the simplicial diagram in $\Prl$
\[\begin{tikzcd}
	\cdots & {\Shv(M;\Sp)\otimes\Shv(M;\Sp)\otimes\Shv(M_\RR;\Sp)} & {\Shv(M;\Sp)\otimes\Shv(M_\RR;\Sp)} & {\Shv(M_\RR;\Sp)}
	\arrow[shift left=3, from=1-1, to=1-2]
	\arrow[shift right=3, from=1-1, to=1-2]
	\arrow[shift right, from=1-1, to=1-2]
	\arrow[shift left, from=1-1, to=1-2]
	\arrow[shift right=2, from=1-2, to=1-3]
	\arrow[shift left=2, from=1-2, to=1-3]
	\arrow[from=1-2, to=1-3]
	\arrow[shift left, from=1-3, to=1-4]
	\arrow[shift right, from=1-3, to=1-4]
\end{tikzcd}\]
    given by the Bar complex calculating  the relative tensor product. By the K\"unneth formula \cite[Proposition 2.30]{volpe2023sixoperation}, one might identify each term  with 
\[\begin{tikzcd}
	\cdots & {\Shv(M\times M\times M_\RR;\Sp)} & {\Shv(M\times M_\RR;\Sp)} & {\Shv(M_\RR;\Sp)}
	\arrow[shift left=3, from=1-1, to=1-2]
	\arrow[shift right=3, from=1-1, to=1-2]
	\arrow[shift right, from=1-1, to=1-2]
	\arrow[shift left, from=1-1, to=1-2]
	\arrow[shift right=2, from=1-2, to=1-3]
	\arrow[shift left=2, from=1-2, to=1-3]
	\arrow[from=1-2, to=1-3]
	\arrow[shift left, from=1-3, to=1-4]
	\arrow[shift right, from=1-3, to=1-4]
\end{tikzcd},\]
where all the functors are now given by $!$-pushforward. In other words, this diagram is the outcome of applying $D(-)_!$ to the diagram of \v Cech nerve of the map
\[\MR\longrightarrow\MR/M\in\LCH.\]
Now one can take right adjoints and compute the limit of the following diagram in $\Cat$
\[\begin{tikzcd}
	\cdots & {\Shv(M\times M\times M_\RR;\Sp)} & {\Shv(M\times M_\RR;\Sp)} & {\Shv(M_\RR;\Sp)}
	\arrow[shift right=3, from=1-2, to=1-1]
	\arrow[shift left=3, from=1-2, to=1-1]
	\arrow[shift left, from=1-2, to=1-1]
	\arrow[shift right, from=1-2, to=1-1]
	\arrow[shift left=2, from=1-3, to=1-2]
	\arrow[shift right=2, from=1-3, to=1-2]
	\arrow[from=1-3, to=1-2]
	\arrow[shift right, from=1-4, to=1-3]
	\arrow[shift left, from=1-4, to=1-3]
\end{tikzcd},\]
where all the functors are now $!$-pullback. Since all the maps are \'etale, the $!$-pullback are canonically identified with $*$-pullback. Thus the diagram is identified with the outcome of taking $\Shv(-)$ and $*$-pullback of the \v Cech nerve of the covering map $M_\RR\rightarrow M_\RR/M$. So we might conclude that the limit
\[\lim_\Delta\Shv(M^{\times n}\times M_\RR;\Sp)\simeq\Shv(M_\RR/M;\Sp),\]
by \'etale descent of taking $\Shv(-)$ and $*$-pullback.
\end{rem}

\begin{proof}[Proof of \cref{D-shriek-is-symmetric-monoidal}]
    It follows from \HA{Remark}{4.8.1.9} and 
    \begin{itemize}
        \item On objects each $X$ is taken to a stable presentable category $\Shv(X;\Sp)$.
        \item On morphisms each $f$ is taken to a colimit-preserving functor $f_!$
        \item The box tensor product on $\Shv(X;\Sp)$ is colimit-preserving in each variable.
        \item The K\"unneth formula holds \cite[Proposition 2.30]{volpe2023sixoperation}.
    \end{itemize}
\end{proof}

Finally, we can apply equivariantization to the category of sheaves with prescribed singular support:
\begin{rec}
    The condition of being constructible and having prescribed singular support is preserved and can be checked after pullback along an \'etale cover map. This follows from the local nature of the definition \cref{the notion of singular support is local}. See \cite[Lemma 3.7]{Jansen_2024} for a related result that one can check local constancy and constructibility \'etale locally.
\end{rec}
\begin{cor}[Equivariantization for $\Shv_\Lambda$]
\label{equivariantization for microlocal sheaf category}
    There is a symmetric monoidal equivalence 
    \[\Shv_{\Lambda_\Sigma}(M_\RR;\Sp)\otimes_{\Fun(M,\Sp)}\Sp\simeq\Shv_{\overline\Lambda_\Sigma}(M_\RR/M;\Sp)\]
    where the right-hand side is the subcategory of sheaves of spectra on $M_\RR/M$ characterized by the following two conditions:
    \begin{itemize}
        \item It is constructible for the stratification $\overline\Sc_\Sigma\defeq\pi(\Sc_\Sigma)$ inherited from the projection map $\pi$. 
        \item It has singular support lying in $\overline\Lambda_\Sigma\defeq d\pi(\Lambda_\Sigma)\subset T^*M_\RR/M$ inherited form the projection map $\pi$.
    \end{itemize}
    The convolution symmetric monoidal structure on $\Shv(\MR/M;\Sp)$ restricts to a symmetric monoidal structure on $\Shv_{\overline\Lambda_\Sigma}(M_\RR/M;\Sp)$.
\end{cor}
\begin{proof}
    Functoriality of the relative tensor product provides a functor
    \[\Shv_{\Lambda_\Sigma}(M_\RR;\Sp)\otimes_{\Fun(M,\Sp)}\Sp\longrightarrow\Shv(M_\RR;\Sp)\otimes_{\Fun(M,\Sp)}\Sp\simeq\Shv(M_\RR/M;\Sp)\in\CAlg(\Prl).\]
    We will show that this is fully faithful 
    and describe its image in terms of singular support. To do so, we consider the following map between the Bar complexes.
\[\begin{tikzcd}
	\cdots & {\Shv_{\Lambda_\Sigma}(M_\RR;\Sp)\otimes\Fun(M,\Sp)} & {\Shv_{\Lambda_\Sigma}(M_\RR;\Sp)} \\
	\cdots & {\Shv(M_\RR;\Sp)\otimes\Fun(M,\Sp)} & {\Shv(M_\RR;\Sp)}
	\arrow[shift left=2, from=1-1, to=1-2]
	\arrow[shift right=2, from=1-1, to=1-2]
	\arrow[from=1-1, to=1-2]
	\arrow[shift left, from=1-2, to=1-3]
	\arrow[shift right, from=1-2, to=1-3]
	\arrow[from=1-2, to=2-2]
	\arrow[from=1-3, to=2-3]
	\arrow[from=2-1, to=2-2]
	\arrow[shift right=2, from=2-1, to=2-2]
	\arrow[shift left=2, from=2-1, to=2-2]
	\arrow[shift left, from=2-2, to=2-3]
	\arrow[shift right, from=2-2, to=2-3]
\end{tikzcd}.\]
After taking colimit, it recovers the above functor. All the vertical functors are fully faithful (since tensoring with a dualizable category preserves fully faithful functors \cite[Theorem 2.2]{efimov2024ktheorylocalizinginvariantslarge}). We may identify each term in the top row with its image along the vertical functors:
\[\Shv_{\Lambda_\Sigma}(M_\RR;\Sp)\otimes\Fun(M,\Sp)^{\otimes n}\simeq \Shv_{\Lambda_\Sigma}(M_\RR\times M^{\times n};\Sp)\subseteq\Shv(\MR\times M^{\times n};\Sp)\]
where we have implicitly used K\"unneth formula for the bottom row. The right-hand side is the category of sheaves $\Fc$ on $M_\RR \times M^{\times n}$ such that on each component $M_\RR$, $\Fc$ is constructible for $\Sc_\Sigma$ and has singular support contained in $\Lambda_\Sigma$. Now we observe the following: the right adjoint of each functor in the bottom row is the $*$-pullback along an \'etale map. In particular it preserves the condition of constructibility and singular support. Thus taking right adjoints of the bottom row restricts to taking right adjoints of the top row:
\[\begin{tikzcd}
	\cdots & {\Shv_{\Lambda_\Sigma}(M_\RR\times M;\Sp)} & {\Shv_{\Lambda_\Sigma}(M_\RR;\Sp)} & {\Shv_{\Lambda_\Sigma}(M_\RR;\Sp)\otimes_{\Fun(M,\Sp)}\Sp} \\
	\cdots & {\Shv(M_\RR\times M;\Sp)} & {\Shv(M_\RR;\Sp)} & {\Shv(M_\RR/M;\Sp)}
	\arrow[shift right=2, from=1-2, to=1-1]
	\arrow[shift left=2, from=1-2, to=1-1]
	\arrow[from=1-2, to=1-1]
	\arrow[from=1-2, to=2-2]
	\arrow[shift right, from=1-3, to=1-2]
	\arrow[shift left, from=1-3, to=1-2]
	\arrow[from=1-3, to=2-3]
	\arrow[from=1-4, to=1-3]
	\arrow[from=1-4, to=2-4]
	\arrow[from=2-2, to=2-1]
	\arrow[shift right=2, from=2-2, to=2-1]
	\arrow[shift left=2, from=2-2, to=2-1]
	\arrow[shift right, from=2-3, to=2-2]
	\arrow[shift left, from=2-3, to=2-2]
	\arrow[from=2-4, to=2-3]
\end{tikzcd}\]
Note that both rows are now limit diagrams in $\Cat$. We thus learn that there is a fully faithful functor
\[\Shv_{\Lambda_\Sigma}(M_\RR;\Sp)\otimes_{\Fun(M,\Sp)}\Sp\hookrightarrow\Shv(M_\RR/M;\Sp).\]
As a full subcategory, $\Shv_{\Lambda_\Sigma}(M_\RR;\Sp)\otimes_{\Fun(M,\Sp)}\Sp$ is spanned by the objects which are sent to the full subcategory
\[\Shv_{\Lambda_\Sigma}(M_\RR;\Sp)\hookrightarrow\Shv(M_\RR;\Sp)\]
through $*$-pullback along the projection 
\[\MR\longrightarrow\MR/M.\]
By the observation we have made in the very beginning that we can check constructibility and singular support locally, this is precisely the category of sheaves on $M_\RR/ M$ which are constructible for $\overline{S}_\Sigma$ and has prescribed singular support contained in $\overline{\Lambda}_\Sigma$. The proof is now done.
\end{proof}
\begin{cor}
\label{nonequivariant version of the correspondence}
    There is a symmetric monoidal equivalence
    \[\overline\kappa:\QCoh(X_\Sigma)\overset\simeq\longrightarrow\Shv_{\overline\Lambda_\Sigma}(M_\RR/M;\Sp)\]
    where the right-hand side is the category appearing in \cref{equivariantization for microlocal sheaf category}.
\end{cor}

\begin{proof}
    It follows from the commutative diagram of \cref{mainthmB} that the relative tensor products are identified in $\CAlg(\Pr^L)$:
    \[\QCoh([X_\Sigma/\TT)\otimes_{\QCoh(B\TT)}\Sp\simeq\Shv_{\Lambda_\Sigma}(M_\RR;\Sp)\otimes_{\Fun(M,\Sp)}\Sp\]
    Now the result follows from \cref{de-equivariantization for quasicoherent} and \cref{equivariantization for microlocal sheaf category}.
\end{proof}
\subsection{Beilinson's theorem on projective spaces}
As a concrete example of the abstract nonsense we have developed, we now give an explanation of Beilinson's linear algebraic description of quasi-coherent sheaves of $\mathbb{P}^1_\SS$, the flat projective line over $\SS$. Recall that the toric data corresponding to the projective line is given by the lattice $N=\ZZ$ and the fan $\{\{0\},\RR_{\geq 0},\RR_{\leq 0}\}$ inside $\RR^1$. 
\begin{example}
\label{beilinson}
There are equivalences of categories:
\[\QCoh(\mathbb{P}^1_\SS)\simeq \Cons_{\overline\Sc_\Sigma}(S^1;\Sp)\simeq \Fun(\bullet\rightrightarrows\bullet;\Sp)\]
where the stratification $\overline\Sc_\Sigma$ has two strata: the origin and its complement. The first equivalence is given by $\overline{\kappa}$ and the second is given by exodromy \cite{haine2024exodromyconicality}.
\end{example}
\begin{proof}The first functor is $\overline\kappa$ supplied by \cref{nonequivariant version of the correspondence}. More precisely, it embeds $\QCoh(\mathbb{P}^1_\SS)$ as a full subcategory in $\Cons_{\overline\Sc_\Sigma}(S^1;\Sp)$. However, one checks readily that the condition on singular support is vacuous. Away from the origin, every $\overline \Sc_\Sigma$ constructible sheaf becomes locally constant, hence the singular support is always contained in the zero section. At the origin, the singular support asks for the support of some sheaf on $\RR^1$ to have support contained in $\RR^1$, which is again no restriction. We thus conclude that the first functor is an equivalence. The second functor is an direct application of exodromy equivalence from \cite{haine2024exodromyconicality}. Note that the exit path category of $(S^1,\overline\Sc_\Sigma)$ is precisely the quiver $\bullet\rightrightarrows\bullet$.
\end{proof}
\begin{rem}
    It is possible to obtain the similar result for $\mathbb{P}^n_\SS$ which states that the category $\QCoh(\mathbb{P}^n_\SS)$ is compactly generated by a collection of objects $\Oc(1),\cdots,\Oc(n+1)$, and they form an exceptional collection. This, however is more involved since the condition on singular support puts an actual constraint so one needs further arguments beyond applying exodromy equivalence. We only present a sketch of the proof idea here. Pick some equivariant lifts $\Tilde{\Oc}(i)\in\QCoh([\mathbb{P}^n_\SS/\TT])$. The image of these $\Tilde{\Oc}(i)$ under $\kappa$ are dualizing sheaves on some explicit moment polytopes in $\RR^n$ as in \cref{subsection:combinatorics of smooth projective fan}. To show that they generate, one can run the argument in \cref{subsection: microlocal characterization of image} to see that these images $\kappa(\Tilde{\Oc}(i))$ corepresents taking stalks at each points in a fundamental domain of $\RR^n/\ZZ^n$, so by adjunction $\overline{\kappa}(\Oc(i))$ also corepresents taking stalks at each point on $\RR^n/\ZZ^n$. This proves that they generate, and the mapping spectra can be directly computed by looking at the intersections of these moment polytopes, which we omit. This computation also recovers the presentation of $\QCoh(\mathbb{P}^n_\SS)$ as the category of presheaves of spectra on an explicit quiver with relations defined by Beilinson.
\end{rem}
\begin{rem}
     This suggests we might dream of the exodromy for constructible sheaves with prescribed singular supports: can one read off Beilinson's quiver directly from the singular support $\overline{\Lambda}_{\Sigma_n}$ where $\Sigma_n$ is the fan for $\mathbb{P}^n$? In general, we might ask how to describe the category of constructible sheaves with prescribed singular supports in terms of presheaf categories: this is not always possible, and the question of when it is possible remains largely open.
     However, see below for the example of $\mathbb{P}^2$ where it is indeed possible.
\end{rem}

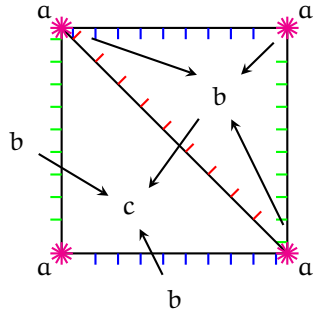
\begin{figure}[h]
    \centering
\begin{tikzpicture}[scale=3,>=stealth,thick]


\draw[
  decoration={
    markings,
    mark=between positions 0.05 and 0.95 step 0.1
      with { \draw[blue] (0,0) -- (0,0.15); }
  },
  postaction=decorate
] (1,0) -- (0,0);

\draw[
  decoration={
    markings,
    mark=between positions 0.05 and 0.95 step 0.1
      with { \draw[green] (0,0) -- (0,0.15); }
  },
  postaction=decorate
] (1,0) -- (1,1);

\draw[
  decoration={
    markings,
    mark=between positions 0.05 and 0.95 step 0.1
      with { \draw[blue] (0,0) -- (0,-0.15); }
  },
  postaction=decorate
] (0,1) -- (1,1);

\draw[
  decoration={
    markings,
    mark=between positions 0.05 and 0.95 step 0.1
      with { \draw[green] (0,0) -- (0,-0.15); }
  },
  postaction=decorate
] (0,1) -- (0,0);

\draw[
  decoration={
    markings,
    mark=between positions 0.05 and 0.95 step 0.1
      with { \draw[red] (0,0) -- (0,0.15); }
  },
  postaction=decorate
] (0,1) -- (1,0);


\node (aSW) at (0,0) [below left]  {$a$};
\node (aSE) at (1,0) [below right] {$a$};
\node (aNE) at (1,1) [above right] {$a$};
\node (aNW) at (0,1) [above left]  {$a$};

\node (cLeft) at (-0.2,0.5) {$b$};

\node (cBot) at (0.5,-0.2) {$b$};

\node (bReg) at (0.3,0.2) {$c$};

\node (cReg) at (0.7,0.7) {$b$};




\draw[->,shorten >=2pt,shorten <=12pt]
  (aNW.south east) -- (cReg.north west);

\draw[->,shorten >=2pt,shorten <=8pt]
  (aNE) -- (cReg);

\draw[->,shorten >=2pt,shorten <=2pt]
    (cLeft) -- (bReg);

\draw[->,shorten >=2pt,shorten <=2pt]
  (cReg.south west) -- (bReg.north east);

\draw[->,shorten >=2pt,shorten <=2pt]
    (cBot) -- (bReg);

\draw[->,shorten >=2pt,shorten <=12pt]
  (aSE) -- (cReg);

\foreach \ang in {0,30,...,330} {
  \draw[thick,magenta] (0,0) -- ++(\ang:0.05);
}

\foreach \ang in {0,30,...,330} {
  \draw[thick,magenta] (1,0) -- ++(\ang:0.05);
}

\foreach \ang in {0,30,...,330} {
  \draw[thick,magenta] (1,1) -- ++(\ang:0.05);
}

\foreach \ang in {0,30,...,330} {
  \draw[thick,magenta] (0,1) -- ++(\ang:0.05);
}
\end{tikzpicture}
    \caption{An illustration of a sheaf in $\Shv_{\overline\Lambda_{\mathbb{P}^2}}(\RR^2/\ZZ^2)$, drawn in a fundamental domain of $\RR^2/\ZZ^2$.
    The short directional strokes—drawn along the edges and diagonal, fanning out at the corners—schematically represent $\overline\Lambda_{\mathbb{P}^2}$ in each cotangent fiber.
    Three distinguished stalks and ways that they are allowed to exit are drawn.}
    \label{fig:P2}
\end{figure}
\begin{rem}
    When the fan $\Sigma$ is \stress{zonotopal} and \stress{unimodular}\footnote{Unfortunately these assumptions are quite restrictive.} (see \cite[Definition 4.2]{treumann2010remarksnonequivariantcoherentconstructiblecorrespondence}), the conic Lagrangian $\Lambda_\Sigma$ is identified with the conormal variety of the stratification $\Sc_\Sigma$ (similarly for $\overline{\Lambda}_\Sigma$) \cite[Theorem 4.4]{treumann2010remarksnonequivariantcoherentconstructiblecorrespondence}. From this one can argue that the singular support condition is automatically satisfied for all $\Sc_\Sigma$- (resp. $\overline\Sc_\Sigma$-) constructible sheaves.
    Thus \cref{nonequivariant version of the correspondence} identifies $\QCoh(X_\Sigma)$ with the category of constructible sheaves \[\QCoh(X_\Sigma)\simeq\Cons_{\overline{\Sc}_\Sigma}(\MR/M;\Sp)\] where exodromy \cite{haine2024exodromyconicality} applies.
    However, it is not clear to us how to write down the exit path category explicitly from the combinatorics of the fan.
\end{rem}
\subsection{Relative toric bundle}
The proof of \cref{nonequivariant version of the correspondence} depends on the base change  functor
\[(-)\otimes_{\QCoh(B\TT)}\QCoh([X_\Sigma/\TT]):\CAlg(\Prl)_{\QCoh(B\TT)/}\longrightarrow\CAlg(\Prl)\]
applied to the symmetric monoidal functor
\[\QCoh(B\TT)\simeq\Fun(M,\Sp)\overset{\colim}\longrightarrow\Sp\in\CAlg(\Prl).\]
There is no reason to stop at this case, and we can make the formal definition:
\begin{defn}[Relative toric construction]
\label{relative toric construction}
Fix a lattice $N$ and fan $\Sigma$. Given a symmetric monoidal functor 
\[f:M\longrightarrow \Cc\]
where $\Cc\in\CAlg(\Prl_\mathrm{st})$, it induces a map
\[F:\QCoh(B\TT)\simeq\Fun(M;\Sp)\longrightarrow\Cc\in\CAlg(\Prl)\]
and we define
\[\Mod_{X_{\Sigma,f}}\Cc\defeq\Cc\otimes_{\QCoh(B\TT)}\QCoh([X_\Sigma/\TT])\in\CAlg(\Prl)\]
to be the \stress{relative toric bundle over $\Cc$ associated with $\Sigma$ and $f$}.
\end{defn}

From the definition it follows that both $\QCoh([X_\Sigma/\TT])$ and $\QCoh(X_\Sigma)$ are examples of the relative toric construction.
\begin{rem}
\label{twisted sheaf}
Tautologically, we have
\[
    \Mod_{X_{\Sigma,f}} \Cc \simeq \Shv_{\Lambda_\Sigma}(M_\RR;\Sp) \otimes_{\Fun(M, \Sp)} \Cc.
\]
We believe the right hand side admits a sheaf theoretic interpretation.
In particular, it should by descent describe the category
of (twisted) sheaves on the torus $M_\RR/M$ valued in a local system
of categories specified by the delooping
\[
Bf: BM \simeq \Pi_\oo(M_\RR/M) \rightarrow \CAlg(\Prl).
\]
A general theory of twisted sheaves that allows us to make such descriptions will be pursued in future work.
\end{rem}
\begin{example}
\label{toric fibration}
In \cite{hu2023coherentconstructiblecorrespondencetoricfibrations}, the second named author with Pyongwon Suh considered the data of a classical scheme $S$ and $n$ line bundles $\{L_n\in\mathrm{Pic}(S)\}$ on $S$. Such collection of line bundles defines a symmetric monoidal functor
\[f:\ZZ^n\longrightarrow\QCoh(S)\]
and the relative toric bunlde over $\QCoh(S)$ associated with $\Sigma$ and $f$ could be identified with the category of quasi-coherent sheaves on an $S$-scheme $\mathcal{X}_{\Sigma,f}$:
\[\Mod_{X_{\Sigma,f}}\QCoh(S)\simeq\QCoh(\mathcal{X}_{\Sigma,f}).\]
The relative toric scheme (or so-called toric fibration) $\mathcal{X}_{\Sigma,f}$ is constructed affine locally on $S$, as a toric scheme with respect to the torus associated with $\oplus L_i$ over $S$. Equivalently, it can be identified with the base change of $[X_\Sigma/\TT]\rightarrow B\TT$ along the map $S\rightarrow B\TT$ classifying these line bundles $\{L_i\}$. On the mirror side, the base change can be interpreted as sheaves on the torus $\RR^n/\ZZ^n$ with twisted-coefficient (see \cref{twisted sheaf}) - roughly the stalk of the coefficient category is $\QCoh(S)$ and the monodromy is given by tensoring with $L_i$.
\end{example}
\begin{rem}
    Such $f:\ZZ^n\longrightarrow\Cc$ classifies $n$ strict Picard elements in $\Cc$ that also strictly commute with each other. Beware that such datum is rare in the wild, see \cite{carmeli2022strictpicardspectrumcommutative}. 
\end{rem}
\section{Appendix}
\subsection{Modules over grouplike monoid}
We find the following lemma straightforward, but can not locate a proof in the literature.
\begin{lem}
\label{Grouplike action gives groupoid object}
Let $T$ be an $\oo$-category admitting finite limits, $G \in \Mon(T)$, and $X$ a $G$-module. If $G$ is grouplike, then $(X//G)_\bullet$ is a groupoid object.
\end{lem}
\begin{proof}
Unwinding the definitions,
there is a canonical map
\[
p:(X//G)_\bullet \rightarrow (*//G)_\bullet,
\]
where the latter can be identified with
the underlying simplicial object of $G$,
hence a groupoid object \HA{Remark}{5.2.6.5}.
Therefore it suffices to show that this
map is a Cartesian natural transformation (see \HTT{Definition}{6.1.3.1}.)

In other words, we want to show that for every $\alpha:[m]\to [n]$, the diagram
\[\begin{tikzcd}
	{X \times G^n} & {(X//G)_n} & {(X//G)_m} & {X \times G^m} \\
	{G^n} & {(*//G)_n} & {(*//G)_m} & {G^m}
	\arrow["\simeq", from=1-1, to=1-2]
	\arrow["\alpha", from=1-2, to=1-3]
	\arrow[from=1-2, to=2-2]
	\arrow["\simeq", tail reversed, no head, from=1-3, to=1-4]
	\arrow[from=1-3, to=2-3]
	\arrow["\simeq", from=2-1, to=2-2]
	\arrow["\alpha", from=2-2, to=2-3]
	\arrow["\simeq", tail reversed, no head, from=2-3, to=2-4]
\end{tikzcd}\]
is a pullback, i.e., $p(\alpha):p([n]) \rightarrow p([m]) \in \Fun([1],T)$ is a Cartesian morphism.

We proceed by induction and show that $p|_{\Delta^\op_{\leq n}}$ is a Cartesian transformation for each $n$.
For the base case $n = 0$, there is nothing to prove.
For $n \geq 1$, note that every map in $\Delta_{\leq n}$ can be factored into
a sequence of maps in which each is either in $\Delta_{\leq n - 1}$ or one of the following:
the injective maps $\delta_k:[n - 1] \rightarrow [n]$ and the surjective maps $\sigma_k:[n] \rightarrow [n - 1]$.
Therefore it suffices to show that $p(\delta_k)$ and $p(\sigma_k)$ are Cartesian morphisms.

For $p(\delta_k)$, we claim that it suffices to prove $p(\delta_0)$ and $p(\delta_n)$ are Cartesian:
indeed, for $0 < k < n$, consider the decomposition $[0, k] \cup [k, n] = [n]$ and the diagram
\[\begin{tikzcd}[cramped]
	{p([n])} & {p([0,k])} \\
	{p([k,n])} & {p(\{k\})} \\
	& {p(\{\cdots<k-1<k+1 <\cdots\})} & {p([0,k-1])} \\
	& {p([k + 1,n])}
	\arrow["\shortmid"{marking}, from=1-1, to=1-2]
	\arrow["\shortmid"{marking}, from=1-1, to=2-1]
	\arrow[squiggly, from=1-2, to=2-2]
	\arrow[squiggly, from=2-1, to=2-2]
	\arrow[curve={height=-6pt}, dashed, from=1-1, to=3-2]
	\arrow[squiggly, from=3-2, to=4-2]
	\arrow[squiggly, from=3-2, to=3-3]
	\arrow[curve={height=18pt}, squiggly, from=2-1, to=4-2]
	\arrow[curve={height=-6pt}, squiggly, from=1-2, to=3-3]
\end{tikzcd}.\]
By induction hypothesis, all the squiggly arrows are Cartesian.
By the 2-out-of-3 property of Cartesian morphisms,
to show the dashed arrow is Cartesian (and hence every arrow is Cartesian),
it suffices to show either of the barred arrows is Cartesian.
However $[0,k]\into [n]$ factors as a map in $\Delta_{\leq n - 1}$ followed by $\delta_n$.

Using the identifications \[
(X//G)_n \simeq X \times G^n,
\]
and
\[
\prod_i {([i < i + 1]\into [n])^*}:(*//G)_n \simeq G^n,
\]
$p(\delta_0)$ is equivalent to
\[\begin{tikzcd}[cramped]
	{X\times G^n} & {G^n} \\
	{X\times G^{n-1}} & {G^{n-1}}
	\arrow[from=1-1, to=1-2]
	\arrow[from=1-1, to=2-1]
	\arrow[from=2-1, to=2-2]
	\arrow[from=1-2, to=2-2]
\end{tikzcd},\]
where all the maps are projection, hence Cartesian.

Similarly, $p(\delta_{n})$ is equivalent to
the product of
\[\begin{tikzcd}[cramped]
	{X\times G} & G \\
	X & {*}
	\arrow[from=1-1, to=1-2]
	\arrow["a"', from=1-1, to=2-1]
	\arrow[from=2-1, to=2-2]
	\arrow[from=1-2, to=2-2]
\end{tikzcd}\]
with $G^{n-1}$.
Therefore it suffices to show the map $X\times G \xrightarrow{(a, \pr)} X \times G$ is an equivalence, which is indeed true as it admits a homotopy inverse given by shearing.

To see $p(\sigma_k)$ is Cartesian,
simply note that both its source and target are (induced by) diagonal maps.
\end{proof}

\subsection{Functoriality of forming module categories}
The following is a direct consequence of \HA{Theorem}{4.8.5.16}. The reader might compare it to \HA{Remark}{4.8.5.19}.
\begin{prop}
\label{Taking module category}
    Let $\Cc$ and $\Dc$ be symmetric monoidal categories  admitting all geometric realizations. Let $F:\Cc\rightarrow\Dc$ be a symmetric monoidal functor. Assume that:
    \begin{enumerate}
        \item Tensor products in $\Cc$ and $\Dc$ commute with geometric realizations.
        \item The functor $F$ commutes with geometric realizations.
    \end{enumerate}
    Then there is a diagram
\[\begin{tikzcd}
	{\CAlg(\Cc)} && {\CAlg(\Cat)}
	\arrow[""{name=0, anchor=center, inner sep=0}, "{\Mod_{(-)}(\Cc)}", curve={height=-12pt}, from=1-1, to=1-3]
	\arrow[""{name=1, anchor=center, inner sep=0}, "{\Mod_{F(-)}(\Dc)}"', curve={height=12pt}, from=1-1, to=1-3]
	\arrow[shorten <=3pt, shorten >=3pt, Rightarrow, from=0, to=1]
\end{tikzcd}.\]
    When evaluated at $A\rightarrow B\in\CAlg(\Cc)$, the diagram reads
\[\begin{tikzcd}
	{\Mod_A(\Cc)} & {\Mod_B(\Cc)} \\
	{\Mod_{F(A)}(\Dc)} & {\Mod_{F(B)}(\Dc)}
	\arrow[from=1-1, to=1-2]
	\arrow[from=2-1, to=2-2]
	\arrow[from=1-1, to=2-1]
	\arrow[from=1-2, to=2-2]
\end{tikzcd}.\]
\end{prop}
\begin{proof}
    We pick up notations in \HA{Theorem}{4.8.5.16} and fix $\mathcal{K}$ to be just $\{\Delta^\op\}$ (in particular the following items refer to items there). The symmetric monoidal coCartesian fibrations in (1) and the functor $\Theta^\otimes$ in (3) straighten (symmetric monoidally) to lax symmetric monoidal functors and natural transformations \footnote{Note that $\Theta^\otimes$ preserves coCartesian edges over $\text{Mon}_\text{Assoc}^\mathcal{K}(\Cat)^\otimes$. This follows from the following two facts: from (4) we know it is a symmetric monoidal functor hence preserves coCartesian lift from $\mathrm{Fin}_*$ and from \HA{Proposition}{4.8.5.1} we know the underlying functor $\Theta$ preserves coCartesian edges over $\text{Mon}_\text{Assoc}^\mathcal{K}(\Cat)$.}
\[\begin{tikzcd}
	{\text{Mon}_\text{Assoc}^\mathcal{K}(\Cat)} &&& \Cat
	\arrow[""{name=0, anchor=center, inner sep=0}, "{\Alg(-)}", curve={height=-18pt}, from=1-1, to=1-4]
	\arrow[""{name=1, anchor=center, inner sep=0}, "{\Mod_{(-)}(\Cat(\mathcal{K}))}"', curve={height=18pt}, from=1-1, to=1-4]
	\arrow["\overline{\Theta}", shorten <=5pt, shorten >=5pt, Rightarrow, from=0, to=1]
\end{tikzcd}.\]
One applies further $\CAlg$ on both sides and obtain 
\[\begin{tikzcd}
	{\CAlg(\text{Mon}_\text{Assoc}^\mathcal{K}(\Cat))} &&& {\CAlg(\Cat)}
	\arrow[""{name=0, anchor=center, inner sep=0}, "{\Alg(-)}", curve={height=-18pt}, from=1-1, to=1-4]
	\arrow[""{name=1, anchor=center, inner sep=0}, "{\Mod_{(-)}(\Cat(\mathcal{K}))}"', curve={height=18pt}, from=1-1, to=1-4]
	\arrow["\overline{\Theta}", shorten <=5pt, shorten >=5pt, Rightarrow, from=0, to=1]
\end{tikzcd}.\]
The assumption on $F:\Cc\rightarrow\Dc$ ensures that it lifts to a map in $\CAlg(\text{Mon}_\text{Assoc}^\mathcal{K}(\Cat))$. We evaluate the above natural transformation $\overline{\Theta}$ on $F$ and obtain a commuting diagram in $\CAlg(\Cat)$
\[\begin{tikzcd}
	{\Alg(\Cc)} & {\Mod_\Cc(\Cat(\mathcal{K}))} \\
	{\Alg(\Dc)} & {\Mod_{\Dc}(\Cat(\mathcal{K}))}
	\arrow["{\Mod_{(-)}(\Cc)}", from=1-1, to=1-2]
	\arrow["F", from=1-1, to=2-1]
	\arrow["{(-)\otimes_\Cc\Dc}", from=1-2, to=2-2]
	\arrow["{\Mod_{(-)}(\Dc)}", from=2-1, to=2-2]
\end{tikzcd}.\]
Applying $\CAlg$ again to the diagram gives the commutative square in $\Cat$
\[\begin{tikzcd}
	{\CAlg(\Cc)} & {\CAlg(\Mod_\Cc(\Cat(\mathcal{K})))} \\
	{\CAlg(\Dc)} & {\CAlg(\Mod_{\Dc}(\Cat(\mathcal{K})))}
	\arrow["{\Mod_{(-)}(\Cc)}", from=1-1, to=1-2]
	\arrow["{\Mod_{(-)}(\Dc)}", from=2-1, to=2-2]
	\arrow["F", from=1-1, to=2-1]
	\arrow["{(-)\otimes_\Cc\Dc}", from=1-2, to=2-2]
\end{tikzcd}.\]
Note that by \HA{Lemma}{4.8.4.2}, the functor $(-)\otimes_\Cc\Dc$ is a symmetric monoidal left adjoint
\[(-)\otimes_\Cc\Dc:\Mod_\Cc(\Cat(\mathcal{K}))\longrightarrow\Mod_\Dc(\Cat(\mathcal{K})).\]
It follows that there is an adjunction $(-)\otimes_\Cc\Dc\dashv\text{fgt}$ between $\CAlg(\Mod_\Cc(\Cat(\mathcal{K})))$ and $\CAlg(\Mod_{\Dc}(\Cat(\mathcal{K})))$. Combining this adjunction with the commutative square, we obtain a natural transformation
\[\begin{tikzcd}
	{\CAlg(\Cc)} && {\CAlg(\Mod_\Cc(\Cat(\mathcal{K})))}
	\arrow[""{name=0, anchor=center, inner sep=0}, "{\Mod_{(-)}(\Cc)}", curve={height=-12pt}, from=1-1, to=1-3]
	\arrow[""{name=1, anchor=center, inner sep=0}, "{\Mod_{F(-)}(\Dc)}"', curve={height=12pt}, from=1-1, to=1-3]
	\arrow[shorten <=3pt, shorten >=3pt, Rightarrow, from=0, to=1]
\end{tikzcd}.\]
Post-composing with the forgetful (see \HA{Corollary}{4.8.1.4}) to $\CAlg(\Cat)$ gives what we claimed.
\end{proof}
\subsection{Reminders on Day convolutions}
\label{reminders on Day convolution}
\begin{rem}\label{rem:day}
    Given a small symmetric monoidal category $(\Cc,\otimes)$, there is a symmetric monoidal structure on the spectral presheaf category $\Fun(\Cc^\op,\Sp)$ called `Day convolution'. The stable Yoneda embedding $\yoneda$\footnote{We abuse notation by writing $h$ for both the unstable and the stable Yoneda embedding, when there is no danger of confusion.} has a structure of symmetric monoidal functor and has the following universal property.\par
For any presentably symmetric monoidal stable category $\Dc$ with a symmetric monoidal functor $F:\Cc\rightarrow\Dc$, there exists a unique symmetric monoidal, colimit-preserving lift to $\Fun(\Cc^\op,\Sp)$:
\[\begin{tikzcd}
	\Cc &&&& {\Fun(\Cc^\op,\Sp)} \\
	&& \Dc & {}
	\arrow["h", from=1-1, to=1-5]
	\arrow["F"', from=1-1, to=2-3]
	\arrow["{\exists!}", dashed, from=1-5, to=2-3]
\end{tikzcd}.\]
To be precise, one learns from \HA{Proposition}{4.8.1.10} that for each small symmetric monoidal category $(\Cc,\otimes)$, the presheaf category $\Fun(\Cc^\op,\Spc)$ has the structure of a presentably symmetric monoidal category, and the (unstable) Yoneda functor
$$\yoneda:\Cc\longrightarrow\Fun(\Cc^\op,\Spc)$$
has a structure of symmetric monoidal functor. Moreover, the restriction map
$$\Fun^{\lax\otimes,L}(\Psh{\Cc}{\Spc},\Dc)\overset{h^*}{\longrightarrow}\Fun^{\lax\otimes}(\Cc,\Dc)$$
is an equivalence for any presentably symmetric monoidal category $\Dc$. The restriction of  above functor to the full subcategory of symmetric monoidal functors
$$\Fun^{\otimes,L}(\Psh{\Cc}{\Spc},\Dc)\overset{h^*}{\longrightarrow}\Fun^\otimes(\Cc,\Dc)$$
is also an equivalence. Using the symmetric monoidal adjunction
\[\begin{tikzcd}
	\Prl & {} & \Prlst
	\arrow["{\text{forgetful}}", curve={height=-12pt}, from=1-3, to=1-1]
	\arrow["{-\otimes\Sp}", curve={height=-12pt}, from=1-1, to=1-3]
\end{tikzcd}\]
one learns that the stable analogues 
\[
h^* \colon \Fun^{\lax\otimes,L}(\Psh{\Cc}{\Sp},\Dc)\xlongrightarrow{\simeq}\Fun^{\lax\otimes}(\Cc,\Dc)
\]
\[
h^* \colon \Fun^{\otimes,L}(\Psh{\Cc}{\Sp},\Dc)\xlongrightarrow{\simeq}\Fun^\otimes(\Cc,\Dc)
\]
also hold for any presentably symmetric monoidal stable category $\Dc$.
\end{rem}
\begin{rem}[Day convolution as a partial adjunction] The equivalence above could be understood as a partial adjunction between forgetful and forming category of presheaves:
\[\begin{tikzcd}
	{\CAlg(\bigcat)} & {} & {\CAlg(\Prlst)} \\
	{\CAlg(\smallcat)}
	\arrow["{\text{forgetful}}"', from=1-3, to=1-1]
	\arrow["i"', from=2-1, to=1-1]
	\arrow["{\Psh{-}{\Sp}}"', from=2-1, to=1-3]
\end{tikzcd}.\]
See, for example, \cite[\href{https://florianadler.github.io/AlgebraBonn/KTheory.pdf\#dummy.1.32}{1.32}]{FabianFerdinand} on how to extract the adjoints functorially. In particular, the equivalences 
\[\Fun^{\lax\otimes,L}(\Psh{\Cc}{\Sp},\Dc)^{\simeq} \overset{h^*}{\longrightarrow}\Fun^{\lax\otimes}(\Cc,\Dc)^{\simeq}\]
\[\Fun^{\otimes,L}(\Psh{\Cc}{\Sp},\Dc)^{\simeq} \overset{h^*}{\longrightarrow}  \Fun^\otimes(\Cc,\Dc)^{\simeq}\]
are functorial in $\Cc$ and $\Dc$\footnote{This essentially boils down to the universal property of Day convolution as stated above. It is however convenient for us to phrase it in terms of partial adjunction.}.
This implies that the construction
\begin{align*}
    \CAlg(\smallcat) & \rightarrow \CAlg(\Prlst) \\
    \Cc & \mapsto \Fun(\Cc^\op, \Sp)^{\textbf{Day-}\otimes}
\end{align*}
exhibits $\Fun(\Cc^\op, \Sp)^{\textbf{Day-}\otimes}$ as the
symmetric monoidal stable cocompletion of $\Cc$.
\end{rem}
\BiblatexSplitbibDefernumbersWarningOff
\printbibliography[keyword=alpha]
\printbibliography[notkeyword=alpha, heading=none]
\end{document}